\documentclass[reqno]{amsart}

\usepackage[foot]{amsaddr}

\usepackage{amsmath,amssymb,amsfonts,mathrsfs}

\usepackage{bbold}

\usepackage[english]{babel}
\usepackage[utf8]{inputenc}
\usepackage[T1]{fontenc}
\usepackage{mathtools}
\usepackage{lmodern}

\usepackage{cancel}

\usepackage[
top=3.5cm,
bottom=3.5cm,
left=3.5cm,
right=3.5cm,
marginparwidth=2.5cm
]{geometry}

\usepackage[colorinlistoftodos]{todonotes}
\usepackage[hypertexnames=false]{hyperref}
\hypersetup{
    colorlinks,
    linkcolor={red!50!black},
    citecolor={blue!50!black},
    urlcolor={blue!80!black},
}

\usepackage[shortlabels]{enumitem}
\setlist[1]{itemsep=2pt}
\usepackage[capitalise]{cleveref}

\usepackage{autonum} 

\theoremstyle{plain}
\newtheorem{theorem}{Theorem}[section]
\newtheorem{prop}[theorem]{Proposition}
\newtheorem{corollary}[theorem]{Corollary}
\newtheorem{lemma}[theorem]{Lemma}
\crefname{prop}{Proposition}{Propositions}				

\newtheorem{theoremintro}{Theorem}						
\crefname{theoremintro}{Theorem}{Theorems}				
\newtheorem{corollaryintro}[theoremintro]{Corollary}	
\crefname{corollaryintro}{Corollary}{Corollary}			

\theoremstyle{definition} 
\newtheorem{defi}[theorem]{Definition}
\newtheorem*{defi*}{Definition}

\theoremstyle{remark} 
\newtheorem{remark}[theorem]{Remark}
\newtheorem{example}[theorem]{Example}

\AddToHook{env/lemma/begin}{\crefalias{theorem}{lemma}}
\AddToHook{env/proposition/begin}{\crefalias{theorem}{proposition}}
\AddToHook{env/prop/begin}{\crefalias{theorem}{proposition}}
\AddToHook{env/corollary/begin}{\crefalias{theorem}{corollary}}
\AddToHook{env/corollaryintro/begin}{\crefalias{theorem}{corollary}}
\AddToHook{env/theoremintro/begin}{\crefalias{theorem}{theorem}}
\AddToHook{env/defi/begin}{\crefalias{theorem}{definition}}
\AddToHook{env/remark/begin}{\crefalias{theorem}{remark}}
\AddToHook{env/rmk/begin}{\crefalias{theorem}{remark}}
\AddToHook{env/example/begin}{\crefalias{theorem}{example}}
\AddToHook{env/conjecture/begin}{\crefalias{theorem}{conjecture}}

\newcommand{\RP}{\mathbb{RP}}
\newcommand{\R}{\mathbb{R}}
\newcommand{\N}{\mathbb{N}}

\newcommand{\foc}{\mathrm{foc}}
\newcommand{\inj}{\mathrm{inj}}
\newcommand{\tight}{\mathrm{tight}}

\DeclareMathOperator{\spn}{\mathrm{span}}

\DeclareRobustCommand{\SkipTocEntry}[5]{}

\title[Quantitative tightness: a sub-Riemannian approach]{Quantitative tightness for three-dimensional contact manifolds: a sub-Riemannian approach}

\author{Andrei A. Agrachev$^1$}

\author{Stefano Baranzini$^2$}

\author{Eugenio Bellini$^3$}

\author{Luca Rizzi$^1$}
\email{\href{mailto:lrizzi@sissa.it}{lrizzi@sissa.it}}

\address{$^1$SISSA, via Bonomea 265, 34136 Trieste, Italy}
\address{$^2$Dipartimento di Matematica ``Giuseppe Peano'', Università degli Studi di Torino, Torino, Italy}
\address{$^3$Dipartimento di Matematica ``Tullio Levi-Civita'', Università degli Studi di Padova, via Trieste 63, 35131 Padova (PD), Italy}

\date{\today}

\begin{document}

\begin{abstract}
Through the use of sub-Riemannian metrics we provide quantitative estimates for the maximal tight neighbourhood of a Reeb orbit on a three-dimensional contact manifold. Under appropriate geometric conditions we show how to construct closed curves which are boundaries of overtwisted disks. We introduce the concept of \emph{contact} Jacobi curve, and prove lower bounds of the so-called tightness radius (from a Reeb orbit) in terms of Schwarzian derivative bounds. We compare these results with the corresponding ones from \cite{Massot,QDarboux}, and we show that our estimates are sharp for classical model structures. We also prove similar, but non-sharp, estimates in terms of sub-Riemannian canonical curvature bounds. We apply our results to K-contact sub-Riemannian manifolds. In this setting, we prove a contact analogue of the celebrated Cartan--Hadamard theorem.
\end{abstract}
	
\maketitle

\tableofcontents

\section{Introduction}

A (co-oriented) contact structure $\xi$ on a $3$-manifold $M$ is a totally non-integrable plane field, i.e.\ $\xi = \ker\omega$ for some one-form $\omega$ satisfying the non-integrability condition $\omega\wedge d\omega\neq 0$. A contact manifold is called \emph{overtwisted} if it contains an embedded disk such that $\xi$ is tangent to the disk in a unique interior point and along the boundary (see \cref{def:ot_disk}). A contact manifold is called \emph{tight} if it is not overtwisted. 

The two categories defined by the tight/overtwisted dichotomy have demonstrated surprisingly different properties. Even though overtwisted contact structures on 3-manifolds have been classified \cite{Eliashberg} and tight contact structures have been extensively studied, it is not easy to determine whether a given contact structure is tight or overtwisted.

The well-known Darboux theorem states that contact structures can be locally normalized: any point in a contact manifold has a neighbourhood diffeomorphic to the standard contact structure $(\mathbb R^3, \omega_{\mathrm{st}})$, where $\omega_{\mathrm{st}} = dz+\tfrac{1}{2}(xdy-ydx)$. It is well-known that the latter structure is tight \cite{bennequin}. Combining this fact and Darboux theorem we deduce that every contact structure is locally tight. If we endow the contact manifold with a metric structure, then it makes sense to ask the following question: what is the size of the maximal tight neighbourhood of a given point? In \cite{Massot,QDarboux} the authors estimated the answer in the context of Riemannian geometry, unveiling surprising connections between different notions of convexity, Riemannian curvature, and tightness.

In this paper, instead, we investigate tightness criteria and geometric detection of overtwisted disks via sub-Riemannian geometry. The fact that the sub-Riemannian geodesics are tangent to $\xi$ makes this framework natural and well-adapted to contact geometry. In this regard, recent works \cite{Daniele, BellBosc} have linked sub-Riemannian geometry to contact topology.

In order to introduce our results, we anticipate some definitions. A (three-dimensional) contact sub-Riemannian manifold $(M,\omega, g)$ is a co-oriented contact manifold $(M,\xi)$, with $\xi = \ker \omega$, endowed with a bundle metric $g$ on $\xi$. Any positive rescaling of $\omega$ defines the same contact structure. We always normalize $\omega$  via the sub-Riemannian metric, see \eqref{eq:compatibility}, so that $\omega$ is uniquely determined by $g$ and $\xi$. This choice, in turn, yields a canonical complement of $\xi$: the so-called \emph{Reeb field}, denoted by $f_0$, is defined by the following equations
\begin{equation}
	\omega(f_0)=1,\qquad d\omega(f_0,\cdot)=0.
\end{equation}
To every $(M,\omega,g)$ there is an associated \emph{sub-Riemannian distance} $d$. For any $p,q \in M$, $d(p,q)$ is defined as the infimum of the lengths of the curves tangent to $\xi$ and joining the two points. Thanks to the contact condition, $(M, d)$ is a metric space with the same topology of $M$. In contact sub-Riemannian manifolds, geodesics (i.e., locally length-minimizing curves parametrized by constant speed) are projections of solutions of a Hamiltonian system on $T^*M$. To introduce it, we define the sub-Riemannian Hamiltonian $H:T^* M\to \mathbb R$ as
	\begin{equation}\label{eq:SR_ham_intro}
	H(\lambda):=   \max\left\{\langle \lambda, v \rangle -\frac{1}{2}\|v\|_g^2\, \Big| \, v \in \xi_{\pi(\lambda)} \right\},\qquad \lambda \in T^*M,
	\end{equation}
which is a smooth and fiber-wise non-negative quadratic form (albeit degenerate and with non-compact level sets). We denote with $\vec{H}$ the vector field  on $T^*M$ defined as the symplectic gradient of $H$. Sub-Riemannian geodesics on a contact structure are precisely the projections of integral curves of $\vec{H}$. In the following, we tacitly assume that $(M,d)$ is a complete metric space, in which case $\vec{H}$ is a complete vector field.

\subsection{Tightness radius from a Reeb orbit}

Let $\Gamma \subset M$ be an embedded piece of Reeb orbit, i.e.\ a connected one-dimensional submanifold of $M$, tangent to the Reeb vector field (see \cref{sec:expcoords} for details). The annihilator bundle of $\Gamma$ is
		\begin{equation}
			A\Gamma:=\{\lambda\in T^{*} M \mid \pi(\lambda)\in\Gamma,\,\,\langle\lambda\,, f_0\rangle=0\}.
		\end{equation}
		Given $r>0$ we also define the following bundles:
		\begin{equation}
			A^{<r}\Gamma:=A\Gamma\cap \{\sqrt{2H}<r\},\qquad A^1\Gamma:=A\Gamma\cap \{\sqrt{2H}=1\}.
		\end{equation}
$A\Gamma$ (resp.\ $A^1\Gamma$) plays the same role of the normal bundle (resp.\ unit normal bundle) in Riemannian geometry. The injectivity radius of $\Gamma$ is defined as
	\begin{equation}
		r_{\inj}(\Gamma)=\sup\{r >0 \mid \text{$E :A^{<r}\Gamma\to M$ is a diffeomorphism onto its image}\},
	\end{equation}
	where $E$ is the tubular neighbourhood map, namely the restriction to $A\Gamma$  of the sub-Riemannian Hamiltonian flow:
	\begin{equation}
	E :  A\Gamma  \to M, \qquad E(\lambda) := \pi\circ e^{\vec{H}}(\lambda).
\end{equation}

	It turns out that, for $\lambda \in A^{<r}\Gamma$, and $r = r_{\inj}(\Gamma)$, we have $d(\Gamma, E(\lambda)) = \sqrt{2H(\lambda)}$ so that sub-Rieman\-nian geodesic involved in this construction realize the distance from $\Gamma$. We are interested in the size of the largest such neighbourhood on which the contact structure is tight.

\begin{defi*}[Tightness radius]
Let $(M,\omega, g)$ be a three-dimensional contact sub-Riemannian manifold and let $\Gamma$ be an embedded piece of Reeb orbit with $r_{\inj}(\Gamma)>0$. The \emph{tightness radius} is
\begin{equation}
		r_{\tight}(\Gamma):=\sup\left\{0<r<r_{\inj}(\Gamma) \mid \text{$(A^{<r}\Gamma, \ker E^*\omega)$ is tight}\right\}.
\end{equation}
\end{defi*}
	If $\Gamma$ is a (whole) Reeb orbit with $r_\inj(\Gamma)>0$, then letting $B_r(\Gamma)$ the set of points at sub-Riemannian distance from $\Gamma$ smaller than $r$, it holds $E\left(A^{<r}\Gamma\right)=B_r(\Gamma)$ for all $0<r<r_\inj(\Gamma)$. Therefore the tightness radius can be expressed (for whole orbits) as
      \begin{equation}\label{eq:r_tight-tube-intro}
          r_\tight(\Gamma)=\sup\left\{0<r<r_\inj(\Gamma)\mid \left(B_r(\Gamma),\ker\omega|_{B_r(\Gamma)}\right)\,\,\text{is tight}\right\}.
      \end{equation}
See \cref{rmk:tube,rmk:tube_tight} for details.      

\subsection{Contact Jacobi curves}

The main tool that we introduce to study the tightness radius is the \emph{contact Jacobi curve}. These are inspired by Jacobi curves in (sub-)Riemannian geometry and geometric control theory, which are curves in the Lagrange Grassmannian whose dynamics is intertwined with the presence of conjugate points and curvature (see \cite{AG-Feeback,AZ-JacobiI,AZ-JacobiII,CurvVar} and references within). The dynamics of contact Jacobi curves, instead, is related with the presence of overtwisted disks.

\begin{defi*}[Contact Jacobi curve]\label{def:contactjacobi-intro}
	Let $(M,\omega, g)$ be a complete three-dimensional  contact sub-Riemannian manifold, and let $\Gamma$ be an embedded piece of a Reeb orbit. Consider the following one-parameter family of one-forms on $A^1\Gamma$:
	\begin{equation}\label{eq:moving_omega-intro}
		\omega_r:=\left.\left((\pi\circ e^{r\vec{H}})^{*}\omega\right)\right|_{A^1\Gamma},\qquad r\in\mathbb [0,+\infty).
	\end{equation}
	The \emph{contact Jacobi curve} at $\lambda\in A^1\Gamma$ is the projectivization of $\omega_r|_{\lambda}$:
	\begin{equation}\label{eq:contact_Jacobi-intro}
			\Omega_\lambda:[0,+\infty) \to P (T_{\lambda}^* A^{1}\Gamma)\simeq\RP^1,\qquad
			\Omega_\lambda(r):=P\left(\omega_{r}|_{\lambda}\right).
	\end{equation}
	The \emph{first singular radius} of the contact Jacobi curve is
	\begin{equation}\label{eq:r_o-intro}
		r_o(\lambda):=\inf\{ r>0 \mid \Omega_{\lambda}(r)=\Omega_{\lambda}(0)\}.
	\end{equation}
\end{defi*}

A first key result is a structural theorem for the contact Jacobi curve, namely \cref{thm:singwelldef}, where we also prove that the definition above is well-posed for all $r\in \R$. \cref{thm:singwelldef} is technical, and for this reason it is omitted from this introduction.

The first singular time of contact Jacobi curves detects the presence of overtwisted disks. This is our first main result, corresponding to \cref{thm:dist=conj} in the body of the paper.

\begin{theoremintro}[Geometric tightness radius estimates]
\label{thm:dist=conj-intro}
	Let $(M,\omega, g)$ be a complete three-di\-men\-sional contact sub-Riemannian manifold and let $\Gamma$ be an embedded piece of Reeb orbit with $r_{\inj}(\Gamma)>0$. Then, letting
	\begin{equation}
r_o^-(\Gamma) := \inf \{ r_o(\lambda)\mid \lambda \in A^1\Gamma \}, \qquad 	r_o^+(\Gamma) := \sup \{ r_o(\lambda)\mid \lambda \in A^1\Gamma \},
	\end{equation}	
	the following estimates hold:
	\begin{equation}\label{eq:upperandlower-intro}
	\min\left\{r_{\inj}(\Gamma), r_o^-(\Gamma)\right\}\leq r_{\tight}(\Gamma) \leq \min\left\{r_{\inj}(\Gamma),r_o^+(\Gamma)\right\}.
	\end{equation}
	Moreover, if $r_o^+(\Gamma) < r_{\inj}(\Gamma)$, then for any $q\in \Gamma$, the set
	\begin{equation}\label{eq:Dqdisk-intro}
	D_q:=\{E(r\lambda) \mid \lambda \in A_q^1 \Gamma,\,\, r\leq r_o(\lambda)\}
	\end{equation}
	is an overtwisted disk, and thus $(M,\omega)$ is an overtwisted contact manifold.
\end{theoremintro}

\subsection{Tightness radius estimates}

We apply \cref{thm:dist=conj-intro} by estimating the first singular radius in two conceptually different ways. The first method assumes a control on the Schwarzian derivative of the contact Jacobi curve, while the second method requires bounds on the so-called sub-Riemannian canonical curvature. In the next two sections, we state the results obtained with either method. Then, we compare our tightness radius estimates with the corresponding ones obtained in \cite{Massot,QDarboux} under Riemannian curvature bounds.

\subsubsection{Tightness radius estimates via Schwarzian derivative}

This first method, natural in view of the definition of the contact Jacobi curve as curve in $\RP^1$, is based on the concept of Schwarzian derivative (see \cref{sec:Schwarzianderivative}). This yields a sharp comparison theorem for the tightness radius, corresponding to \cref{thm:comparisonschwarzian}.

\begin{theoremintro}[Tightness radius estimates with Schwarzian derivative bounded above]\label{thm:comparisonschwarzian-intro}
		Let $(M,\omega, g)$ be a complete three-di\-men\-sional contact sub-Riemannian manifold and let $\Gamma$ be an embedded piece of Reeb orbit with $r_{\inj}(\Gamma)>0$. Assume that there exist $k_1,k_2\in \R$ such that the Schwarzian derivative of the contact Jacobi curves are bounded above by:
		\begin{equation}\label{eq:schwarzianestimate-intro}
			\frac{1}{2}\mathcal S(\Omega_{\lambda})(r)\leq -\frac{3}{4r^2}+k_1r+k_2r^2,\qquad \forall\, r\in (0,r_{\inj}(\Gamma)],\quad \forall\, \lambda\in A^1\Gamma.
		\end{equation}
Then for the tightness radius of $\Gamma$ it holds
\begin{equation}\label{eq:rtightestimate-intro}
r_\tight(\Gamma) \geq \min\left\lbrace r_*(k_1,k_2),r_{\inj}(\Gamma) \right\rbrace,
\end{equation}
where
\begin{equation}\label{eq:k12-intro}
r_*(k_1,k_2):= \begin{cases}
\frac{-k_1+\sqrt{8 \pi  k_2^{3/2}+k_1^2}}{2 k_2}, & \text{if } k_1,k_2>0, \\
\frac{\sqrt{2 \pi} }{k_2^{1/4}}, & \text{if }  k_1 \leq 0, \, k_2>0,\\
			\left(\frac{3 j_{2/3}}{2}\right)^{2/3}\frac{1}{k_1^{1/3}}, & \text{if }  k_1>0, \,k_2 \leq 0 ,\\
			+\infty & \text{if } k_1,k_2\leq 0,
			\end{cases}
\end{equation}
and $j_{2/3}\sim 3.37$ is the first positive root of the Bessel function of first kind $J_{2/3}$.
\end{theoremintro}

The upper bound \eqref{eq:schwarzianestimate-intro} is motivated by the structure of contact Jacobi curves, i.e.\ \cref{thm:singwelldef}. All three terms in the r.h.s.\ of \eqref{eq:schwarzianestimate-intro} are needed in order to get sharp and general results. 
We note that the upper bound in \eqref{eq:schwarzianestimate-intro} is always verified for a
given $\lambda$ if $r_\inj(\Gamma)<+\infty$ (e.g.\ when $M$ is compact). We refer to \cref{rmk:asymptotics} and \cref{rmk:kappa1} for details. 
Furthermore \cref{thm:comparisonschwarzian-intro} is sharp for the standard sub-Rieman\-nian contact structure (resp.\ overtwisted structure) of \cref{ex:st_sr} (resp.\ \cref{ex:st_ot_sr}), and all left-invariant K-contact structures. See \cref{rmk:sharp}.

\subsubsection{Tightness radius estimates via canonical curvature}

A second way to estimate the singular radius of a contact Jacobi curve, and thus the tightness radius from a Reeb orbit, is via curvature-type invariants. Our result if formulated in terms of the so-called \emph{canonical curvatures}. The canonical curvature plays the role of sectional curvature in sub-Riemannian comparison geometry. Historically, it was introduced in \cite{AG-Feeback,AZ-JacobiI}, formalized in \cite{ZeLi}, and further studied and developed in \cite{ZeLi2,LLZ,AAPL,AL-Bishop,CurvVar,ABR-contact,conj,BR-BakryEmery,BR-Jacobi}. See \cref{sec:canonical} for details. In the contact case, there are only two such curvatures, associated to a given sub-Riemannian geodesic $\gamma^\lambda :[0,T]\to M$, with initial covector $\lambda \in T^*M$, and are denoted by $R_a^\lambda,R_c^\lambda:[0,T]\to \R$. The next result corresponds to \cref{thm:comparisoncurvature}.

\begin{theoremintro}[Tightness radius estimates with curvature bounded above]\label{thm:comparisoncurvature-intro}
		Let $(M,\omega, g)$ be a complete three-di\-men\-sional contact sub-Riemannian manifold and let $\Gamma$ be an embedded piece of Reeb orbit with $r_{\inj}(\Gamma)>0$. Assume that there exist $A,C>0$ such that 
\begin{equation}\label{eq:boundAC-intro}
 \sqrt{1+R_a^\lambda(r)^2}\leq A,\quad \sqrt{1+R_c^\lambda(r)^2}\leq C,\qquad \forall\, r\in[0,r_{\inj}(\Gamma)],\quad \forall\,\lambda\in A^1\Gamma,
\end{equation}
where $R_a^\lambda(r),R_c^\lambda(r)$ are the canonical curvatures at $\lambda$.  Then for the tightness radius of $\Gamma$ it holds
	\begin{equation}
		r_{\tight}(\Gamma)\geq\min\left\{\tau(A,C),r_{\inj}(\Gamma)\right\},
	\end{equation}
	where 
	\begin{equation}\label{eq:tauAC-intro}
		\tau(A,C):=\int_{0}^\infty\frac{1}{A u^2+C u+1}\,\mathrm{d}u.
	\end{equation}
\end{theoremintro}
\cref{thm:comparisoncurvature-intro} links tightness with sub-Riemannian curvature, for general three-dimensional contact structures. As it will be clear from the proof, \cref{thm:comparisoncurvature-intro} is non-sharp, even for the standard contact structure.

\subsection{Cartan--Hadamard theorem for K-contact structures}

A contact sub-Riemannian manifold $(M,\omega, g)$ is called K-contact if the Reeb flow acts on $M$ by isometries.

Denote with $\kappa$ the Gaussian curvature of the surface obtained by locally quotienting $M$ under the action of the Reeb flow, see \cref{rmk:chikappa}. It is a well-defined function on $M$, constant along the Reeb orbits. We prove a contact analogue of the celebrated Cartan--Hadamard theorem, corresponding to \cref{thm:contactHadamard}.

\begin{theoremintro}[Contact Cartan--Hadamard]\label{thm:contactHadamard-intro}
	Let $(M,\omega,g)$ be a complete three-di\-men\-sional simply connected K-contact sub-Riemannian manifold. If $\kappa\leq 0$ then $(M, \ker\omega)$ is contactomorphic to the standard contact structure on $\mathbb R^3$  (see \cref{ex:standard}).
\end{theoremintro}

It follows from \cite{N-overtwisted,NP-resolution} that every K-contact structure is tight (see \cite[Rmk.\ 1.3]{QDarboux}). Under the non-positive curvature and simply connectedness assumptions, \cref{thm:contactHadamard-intro} tells us, in addition, that these structures are contactomorphic to the standard one. The proof of \cref{thm:contactHadamard-intro} has the following consequence, corresponding to \cref{cor_reeb_orb}.

	\begin{corollaryintro}\label{cor_reeb_orb-intro}
	Let $(M,\omega,g)$ be a complete three-di\-men\-sional K-contact sub-Riemannian manifold. If $\kappa\leq 0$ then any periodic orbit of the Reeb field is the generator of an infinite cyclic subgroup of the fundamental group $\pi_1(M)$.
	\end{corollaryintro}

	In \cref{sec:nonposi} we apply \cref{thm:contactHadamard-intro} to the contact structures coming from the Boothby-Wang construction. Such a construction yields K-contact structures on principal circle bundles over a Riemannian surface $(B,\eta)$ with prescribed curvature $\kappa_\eta$ (see \cref{thm_examples}). We show that any such contact manifold, provided that $\kappa_\eta \leq 0$, has tight universal cover.
	
	In \cref{sec:posi} we show that, analogously to the classical Cartan--Hadamard theorem, the assumption of non-positive $\kappa$ in \cref{thm:contactHadamard-intro} and \cref{cor_reeb_orb-intro} can be weakened. In particular, we show an example of K-contact sub-Riemannian manifold with $\kappa >0$ and for which the proofs of \cref{thm:contactHadamard-intro,cor_reeb_orb-intro} and their thesis holds unchanged.

\subsection{Comparison with the state of the art}

We compare \cref{thm:comparisonschwarzian-intro,thm:comparisoncurvature-intro} with corresponding results from \cite{QDarboux}. Let $\xi$ be a contact distribution on a three-di\-men\-sional manifold $M$. In \cite[Def.\,2.7]{QDarboux} the authors define a notion of Riemannian metric compatible with $\xi$ (see \cref{sec:compare}).  For such Riemannian structures, given a (whole) Reeb orbit $\Gamma\subset M$ the authors estimate the radius of the maximal tight embedded \emph{Riemannian} tube around $\Gamma$. More precisely, let $E_R:A\Gamma\to M$ denote the \emph{Riemannian} tubular neighbourhood map and let $r_{\inj}^R(\Gamma)$ denote the corresponding (Riemannian) injectivity radius. Let also
\begin{equation}
    B_r^R(\Gamma):=\{q\in M\mid d_R(q,\Gamma)<r\},
\end{equation}
where $d_R(\cdot,\Gamma)$ denotes the Riemannian distance from $\Gamma$. The \emph{Riemannian tightness radius of $\Gamma$} is the Riemannian counterpart of \eqref{eq:r_tight-tube-intro}, namely:
\begin{equation}\label{eq:r_tight-tube-riemannian-intro}
          r_\tight^R(\Gamma):=\sup\left\{0<r<r^R_\inj(\Gamma)\mid \left(B_r^R(\Gamma),\xi|_{B_r^R(\Gamma)}\right)\,\,\text{is tight}\right\}.
      \end{equation}
      
\begin{theorem}[{\cite[Thm.\ 1.8]{QDarboux}}]\label{thm:Riem_tube-intro}
Let $(M,\xi)$ be a three-di\-men\-sional contact manifold and $g$ a complete Riemannian metric that is compatible with $\xi$. Let $\Gamma$ be a Reeb orbit with positive injectivity radius, then the following estimate holds
\begin{equation}\label{eq:Riem_estimate-intro}
    r_{\tight}^{R}(\Gamma)\geq \min\left\{r_{\inj}^R(\Gamma),\frac{\inj(g)}{2}, \frac{\pi}{2\sqrt{K}}, \frac{2}{\sqrt{2a+b^2}+b}\right\},
\end{equation}
where $r_{\inj}^R(\Gamma)$ is the Riemannian injectivity radius from $\Gamma$, $\inj(g)$ is the Riemannian injectivity radius of $(M,g)$, and
\begin{equation}\label{eq:A_B-intro}
    a=\frac{4}{3}|\mathrm{sec}(g)|,\qquad b=\frac{\theta'}{2}+\sqrt{\frac{\left(\theta'\right)^2}{4}-\frac{1}{2}\min_{q\in M}\mathrm{Ric}_q(f_0)},
\end{equation}
where $\mathrm{Ric}$ is the Ricci tensor, $f_0$ is the Reeb field, $\theta'$ is the rotation speed of the compatible metric (see \cref{sec:compare}), $|\mathrm{sec}(g)|$ is the maximum value of the sectional curvature, and $K$ is any positive upper bound for it.
\end{theorem}
\begin{remark}
The statement of \cite[Thm.\ 1.8]{QDarboux} assumes the Reeb orbit to be periodic. However, the statement holds for any Reeb orbit, provided that the corresponding tube is embedded. In other words, one has to include $r^R_{\inj}(\Gamma)$ in the minimum in the right hand side of \cite[Eq.\ (1.2)]{QDarboux}, resulting in the lower bound \eqref{eq:Riem_estimate-intro}.
\end{remark}
Given a contact sub-Riemannian manifold $(M,\omega,g)$, one can obtain a Riemannian metric compatible with the contact distribution in the sense of \cite[Def.\,2.7]{QDarboux} by declaring the Reeb field to be an orthonormal complement of $\xi=\ker\omega$. We call the resulting metric the \emph{Riemannian extension} of the sub-Riemannian one, and we denote it with the same symbol $g$. Conversely, if $(M,\xi)$ is a contact manifold and $g$ is a Riemannian metric compatible with it in the sense of \cite[Def.\,2.7]{QDarboux}, then there exists a unique (normalized) contact form $\omega$ such that $\xi = \ker\omega$ and $(M,\omega,g)$ is a contact sub-Riemannian manifold. The estimate in \cite[Thm.\,1.8]{QDarboux} is formulated in terms of the Riemannian distance, rather than the sub-Riemannian one, compare \eqref{eq:r_tight-tube-riemannian-intro} with \eqref{eq:r_tight-tube-intro}. So, while a Riemannian tightness radius lower bound always implies an identical sub-Riemannian one (just because $d_R \leq d$), the converse is not true. However, the following observation allows to compare the corresponding estimates (see \cref{lem:same-dist-tube}).
\begin{lemma}\label{lem:same-dist-tube-intro}
Let $(M,\omega,g)$ be a complete three-di\-men\-sional contact sub-Riemannian manifold, and let $\Gamma$ be an embedded Reeb orbit. Assume that the sub-Riemannian distance from $\Gamma$ coincides with the one determined by the Riemannian extension, then
\begin{equation}
r_{\tight}^{SR}(\Gamma)=r_{\tight}^R(\Gamma).
\end{equation}
\end{lemma}
We examine model structures of particular importance in which \cref{lem:same-dist-tube-intro} applies, so that one can compare \cref{thm:Riem_tube-intro} with \cref{thm:comparisonschwarzian-intro,,thm:comparisoncurvature-intro}. To this purpose, we denote the right hand side of the corresponding tightness radii estimates as follows:
\begin{align}
    \rho_{\mathrm{Thm.B}}(\Gamma)& :=\min\{r_\inj(\Gamma), r_*(k_1,k_2)\},\label{eq:SR1_rho_estimates-intro}\\
    \rho_{\mathrm{Thm.C}}(\Gamma) & :=\min\{r_\inj(\Gamma), \tau(A,C)\},\label{eq:SR2_rho_estimates-intro}\\
        \rho_{\mathrm{EKM}}(\Gamma) & :=\min\left\{r_{\inj}^R(\Gamma),\frac{\inj(g)}{2}, \frac{\pi}{2\sqrt{K}}, \frac{2}{\sqrt{2a+b^2}+b}\right\},\label{eq:R_rho_estimates-intro}
    \end{align}
where $r_*(k_1,k_2)$ is defined in \eqref{eq:k12-intro}, $\tau(A,C)$ is defined in \eqref{eq:tauAC-intro}, and $a,b,K$ are defined in \eqref{eq:A_B-intro}. The next result, corresponding to \cref{thm:compare-K-cont}, compares the aforementioned lower bounds.

\begin{prop}[{Comparison with \cite{QDarboux}, I}]\label{thm:compare-K-cont-intro}
    Let $(M,\omega,g)$ be a simply connected three-di\-men\-sional K-contact left-invariant sub-Riemannian structure, and let $\Gamma$ be an embedded Reeb orbit. Then, the Riemannian and sub-Riemannian distances from $\Gamma$ coincide. Up to a constant rescaling of the contact form and the sub-Riemannian metric, we have the following three cases:
    \begin{enumerate}[label = $(\roman*)$]
        \item $(M,\omega, g)$ is the left-invariant sub-Riemannian structure on the Heisenberg group (see \cref{ex:left_R3}). Moreover
        \begin{equation}
            \rho_{\mathrm{Thm.B}}(\Gamma)=r_{\tight}(\Gamma)=+\infty, \qquad \rho_{\mathrm{Thm.C}}(\Gamma)=\frac{2\pi}{3\sqrt{3}}\approx 1.02,\qquad  \rho_{\mathrm{EKM}}(\Gamma)\leq 1,
        \end{equation}
    \item $(M,\omega, g)$ is the left-invariant sub-Riemannian structure on $\mathrm{SU}(2)$ (see \cref{ex:left_SU(2)}). Moreover
        \begin{equation}
            \rho_{\mathrm{Thm.B}}(\Gamma)=r_{\tight}(\Gamma)=\pi, \qquad \rho_{\mathrm{Thm.C}}(\Gamma)\approx 1.05,\qquad  \rho_{\mathrm{EKM}}(\Gamma) \leq 1.38,
        \end{equation}    
    \item $(M,\omega, g)$ is the left-invariant sub-Riemannian structure on $\widetilde{\mathrm{SL}}(2)$ (see \cref{ex:left_SL(2)}). Moreover 
        \begin{equation}
            \rho_{\mathrm{Thm.B}}(\Gamma)=r_{\tight}(\Gamma)=+\infty, \qquad \rho_{\mathrm{Thm.C}}(\Gamma)\approx 1.05,\qquad \rho_{\mathrm{EKM}}(\Gamma) \leq 0.74.
        \end{equation}  
    \end{enumerate}
\end{prop}

Finally, in the next result, corresponding to \cref{thm:compare_ot}, we compare \cref{thm:comparisonschwarzian-intro} with \cite[Thm.\ 1.8]{QDarboux}, in the standard sub-Riemannian overtwisted contact structure. Since in this case the curvature is not bounded, $\rho_{\mathrm{EKM}}(\Gamma)=0$. We can still use \cref{thm:Riem_tube-intro} (i.e.\ \cite[Thm.\ 1.8]{QDarboux}) to obtain a non-trivial bound, by replacing $\rho_{\mathrm{EKM}}(\Gamma)$ with an adapted version of the latter, that we call $\tilde{\rho}_{\mathrm{EKM}}(\Gamma)$, as explained in \cref{rmk:unbounded_curv}.

\begin{prop}[{Comparison with \cite{QDarboux}, II}]\label{thm:compare_ot-intro}
    Let $(\mathbb R^3,\omega_{\mathrm{ot}}, g_\mathrm{ot})$ be the standard sub-Rieman\-nian overtwisted structure (see \cref{ex:st_ot_sr}). Let $\Gamma$ be the Reeb orbit
    \[
    \Gamma=\{(0,0,z)\mid z\in\mathbb R\}.
    \]
Then the Riemannian and sub-Riemannian distances from $\Gamma$ coincide. It holds
    \begin{equation}
        \rho_{\mathrm{Thm.B}}(\Gamma)=r_\tight(\Gamma)=\sqrt{2\pi},
       \qquad \tilde{\rho}_{\mathrm{EKM}}(\Gamma)\leq \sqrt[3]{2}.
    \end{equation}
\end{prop}
We conclude that for the classes of structures covered by \cref{thm:compare-K-cont-intro,thm:compare_ot-intro} the tightness radius estimate from \cref{thm:comparisonschwarzian-intro} are sharp, while the ones of \cref{thm:comparisoncurvature-intro} and \cref{thm:Riem_tube-intro} are not. While it is difficult to compare all these results in general (lacking a direct formula linking the Schwarzian derivative of a contact Jacobi curve with Riemannian or sub-Riemannian curvatures), it emerges from this analysis that the contact Jacobi curves are a better tool to detect overtwisted disks with respect to other curvature-based methods.

\begin{remark}
In an earlier paper \cite{Massot}, the authors estimate the radius of the maximal tight embedded Riemannian \emph{ball}, instead of a tube around a Reeb orbit.  Under suitable assumptions, one can compare \cref{thm:comparisonschwarzian-intro} also with \cite[Thm.\ 1.1]{Massot}. The former improves and sharpens the latter for the models of \cref{thm:compare-K-cont-intro,thm:compare_ot-intro}. Details can be found in \cite{PhDEugenio}.
\end{remark}
	
\subsection{Structure of the paper}
In \cref{sec:preliminaries} we review some preliminaries in contact sub-Riemannian geometry. In \cref{sec:expcoords} we introduce special coordinates on the annihilator of a Reeb orbit, instrumental for our proofs. The tightness radius is defined in \cref{sec:tight_radius}, and it is computed for a class of models (\cref{sec:models}). In the same section we study the relation between the tightness radius and the so-called singular locus of a Reeb orbit (\cref{sec:tightandsingular}). The core of the paper is \cref{sec:contactJacobi}, where we introduce contact Jacobi curves, and we prove \cref{thm:dist=conj-intro,,thm:comparisonschwarzian-intro,,thm:comparisoncurvature-intro}. Finally, in \cref{sec:Kcontact} we focus on K-contact structures, proving \cref{thm:contactHadamard-intro,,cor_reeb_orb-intro}. The core examples, namely the standard contact structure (resp.\ the standard overtwisted structure) on $\R^3$, are presented throughout the paper, in \cref{ex:standard,,ex:st_sr,,cor:model_radii_st} (resp.\ \cref{ex:standard_ot,,ex:st_ot_sr,,cor:model_radii_ot}). In \cref{sec:compare} we prove \cref{thm:compare-K-cont-intro,thm:compare_ot-intro}.


\subsection{Acknowledgements}

This project has received funding from (i) the European Research Council (ERC) under the European Union's Horizon 2020 research and innovation programme (grant agreement GEOSUB, No. 945655); (ii) the PRIN project ``Optimal transport: new challenges across analysis and geometry'' funded by the Italian Ministry of University and Research. The authors also acknowledge the INdAM support.

	\section{Preliminaries}\label{sec:preliminaries}

	\subsection{Three-dimensional contact manifolds} We review some preliminaries on contact manifolds. We refer e.g.\ to \cite{Massot-LectureNotes,Geiges} for more details. A three-di\-men\-sional contact manifold $(M,\xi)$ is a smooth, connected $3$-manifold $M$ endowed with a contact distribution, i.e., a plane field $\xi\subset TM$ satisfying the non-integrability condition $\xi+[\xi,\xi]=TM$. The latter equation is called the contact condition. We say that two contact manifolds $(M,\xi)$ and $(M',\xi')$ are \emph{contactomorphic} if there exists a diffeomorphism $\psi: M \to M'$ such that $\psi_* \xi = \xi'$.

	In the present paper we assume contact manifolds to be co-orientable, i.e., we assume the existence of a smooth one-form $\omega$ such that $\xi=\ker\omega$. Such a differential form is called a \textit{contact form}. Any positive rescaling of $\omega$ determines the same contact structure. The contact condition can be expressed in terms of the contact form:
	\begin{equation}
		\xi+[\xi,\xi]=TM \iff  \text{$\omega\wedge d\omega$ is a volume form} \iff  \text{$d\omega|_{\xi}$ is symplectic}.
	\end{equation}
	The geometric meaning of such condition is that the plane field $\xi$ twists monotonically along horizontal foliations, i.e., foliations by curves which are tangent to $\xi$. This prevents $\xi$ from having any integral surface. A fundamental notion in the study of contact structures is that of a characteristic foliation. 
	\begin{defi}\label{def:singlocuscontact}
		Let $(M,\xi)$ be a contact manifold, and $\Sigma \subset M$ be an embedded surface. The characteristic foliation is the singular line field $\xi\cap T\Sigma$ i.e., a rank $1$ distribution with singularities. We say that $q \in \Sigma$ is \emph{singular} if
				\begin{equation}
		\xi_q=T_q\Sigma.
		\end{equation}
	\end{defi}
	
	Let us introduce some important examples, which are needed in the following.
	
	\begin{example}[Standard contact structure]\label{ex:standard}
		The standard contact structure on $\mathbb R^3$ is defined, in cylindrical coordinates $x=r\cos\theta$, $y=r\sin\theta$, $z$, by the contact form 
		\begin{equation}
			\omega_{\mathrm{st}}:=dz+\frac{r^2}{2}d\theta,\qquad \xi_{\mathrm{st}}:=\ker\omega_{\mathrm{st}}.
		\end{equation}
		Let $\Sigma=\{z=0\}$. Its characteristic foliation $\xi_{\mathrm{st}} \cap T\Sigma$ consists of radial lines and has a unique singularity at the origin. The distribution twists monotonically along radial lines: the angle between $\xi_{\mathrm{st}}$ and $\Sigma$ is the monotone function 
		\begin{equation}
			\phi_{\mathrm{st}}(r):=\arctan\left(\frac{r^2}{2}\right),
		\end{equation}
		which reaches $\pi/2$ asymptotically as $r \to \infty$, namely:
		\begin{equation}
			\lim_{r\to\infty}\phi_{\mathrm{st}}(r)=\frac{\pi}{2}.
		\end{equation}
	\end{example}
	\begin{example}[Standard overtwisted structure]\label{ex:standard_ot}
		The standard overtwisted structure on $\mathbb R^3$ is defined, in cylindrical coordinates $x=r\cos\theta$, $y=r\sin\theta$, $z$, by the contact form 
		\begin{equation}
			\omega_{\mathrm{ot}}:=\cos\left(\frac{r^2}{2}\right)dz+\sin\left(\frac{r^2}{2}\right)d\theta,\qquad \xi_{\mathrm{ot}}:=\ker\omega_{\mathrm{ot}}.
		\end{equation}
		Let $\Sigma:=\{z=0\}$. Its characteristic foliation $\xi_{\mathrm{st}} \cap T\Sigma$ consists of radial lines and the singularities are the origin and the circles $C_{k}:=\{(x,y,0) \mid x^2+y^2=2k\pi\}$, $k\in\mathbb N$.
		The distribution twists monotonically along radial lines: the angle between $\xi_{\mathrm{ot}}$ and $\Sigma$ is the monotone function 
		\begin{equation}
			\phi_{\mathrm{ot}}(r):= \frac{r^2}{2},
		\end{equation}
		which rotates infinitely many times as $r \to \infty$, namely:
		\begin{equation}
			\lim_{r\to\infty}\phi_{\mathrm{st}}(r)= + \infty.
		\end{equation}
		The disk delimited by the first singular circle, 
  \begin{equation}\label{eq:st_ot_disk}
      D_{\mathrm{ot}}:=\{(x,y,0) \mid x^2+y^2 \leq 2\pi\},
  \end{equation}
  together with its characteristic foliation, is called the standard overtwisted disk.
	\end{example}
	\begin{defi}\label{def:ot_disk}
		A three-di\-men\-sional contact manifold $(M,\xi)$ is \emph{overtwisted} if it contains an overtwisted disk. Namely, there exists an embedded disk $D \subset M$ and a diffeomorphism $\psi : D \to D_{\mathrm{ot}}$, such that
		\begin{equation}
			\psi_* \left( \xi \cap TD\right) = \left(\xi_{\mathrm{ot}} \cap T D_{\mathrm{ot}}\right).
		\end{equation}
		A contact manifold is called tight if it is not overtwisted.
	\end{defi}
	\begin{remark}\label{rmk:ot_disk}
		\cref{def:ot_disk} is equivalent to the existence of neighbourhoods $U \subset \R^3$ of $D_{\mathrm{ot}}$, $V \subset M$, and a diffeomorphism $\Psi : V \to U$ such that $\Psi(D) = D_{\mathrm{ot}}$ and $\Psi_* \xi = \xi_{\mathrm{ot}}$. See \cite[Lemma 13 and Exercise p.\ 48]{Massot-LectureNotes}.
	\end{remark}
\begin{remark}[Universal tightness]\label{rmk:univ_tight}
    	    Given a three-di\-men\-sional contact manifold $(M,\xi)$ with universal cover $p:\tilde M\to M$, there is a natural way to endow $\tilde M$ with a contact structure, which at each $q\in\tilde M$ is determined by the equation $p_*\tilde\xi_q=\xi_{p(q)}$. 
            If $(\tilde M, \tilde\xi)$ is tight then $(M,\xi)$ is called \emph{universally tight}. If $(M,\xi)$ is overtwisted, then any overtwisted disk can be lifted to the universal cover. Therefore universal tightness implies tightness.
    	\end{remark}
	The standard overtwisted structure is overtwisted, while the standard contact structure $\xi_{\mathrm{st}}$, as proven by Bennequin \cite{bennequin}, is tight. A well-known property of contact structures is that they can be locally normalized: according to the Darboux theorem every contact structure is locally contactomorphic to the standard one $(\mathbb R^3, \xi_{\mathrm{st}})$. Therefore every contact structure is locally tight. If a contact manifold is endowed with a metric structure, it makes sense to ask what is the size of the maximal tight neighbourhood of a point. In this work the selected metric structure is a sub-Riemannian one. This choice is natural in view of the fact that length-minimizing curves are horizontal and the distribution twists monotonically along them, so that we can expect to detect overtwisted disks in terms of the sub-Riemannian distance.

	\subsection{Contact sub-Riemannian geometry}\label{sec:contactSR}
	A three-di\-men\-sional contact sub-Rieman\-nian manifold $(M,\omega, g)$ is a co-orientable contact manifold $(M,\xi)$, with $\xi = \ker \omega$, endowed with a bundle metric $g$ on $\xi$, called sub-Riemannian metric. Note that as a consequence of the definition, $M$ is oriented by the non-vanishing top-form $\omega \wedge d\omega$.
	The \emph{Reeb field} $f_0$ is the vector field transverse to $\xi$ defined by the conditions
	\begin{equation}
		\omega(f_0)=1,\qquad d\omega(f_0,\cdot)=0.
	\end{equation}
	We extend the sub-Riemannian metric $g$ to a Riemannian one by declaring $f_0$ to be orthonormal to $\xi$. We say that $\omega$ is \emph{normalized} (with respect to the given sub-Riemannian metric) if
	\begin{equation}\label{eq:compatibility}
 \omega \wedge d\omega = \mathrm{Vol}_g,
	\end{equation}
	where $\mathrm{Vol}_g$ is the volume form induced by the extended metric on the oriented manifold $M$. 
	
	Note that for any contact sub-Riemannian manifold $(M,\omega,g)$ one can multiply $\omega$ by a suitable positive function $f: M \to \R$ so that the normalization condition \eqref{eq:compatibility} is satisfied for $(M,f \omega, g)$. For this reason, in the following, we always assume that \eqref{eq:compatibility} is satisfied.
	
	The sub-Riemannian structure induces an almost-complex structure, i.e.\ a smooth bundle morphism $J:\xi\to \xi$ such that $J^2 = - \mathbb{1}$, which is defined by
	\begin{equation}\label{eq:complex_strctr}
		d\omega(X, J Y)=g(X, Y), \qquad \forall \, X,Y \in \xi_p, \quad \forall p \in M.
	\end{equation}
	Let $f_0,f_1,f_2$ be a local oriented orthonormal frame (in particular $f_1,f_2$ are horizontal). Let $\nu_0,\nu_1,\nu_2$ the corresponding dual co-frame. Then $\omega = \nu_0$ and \eqref{eq:compatibility}-\eqref{eq:complex_strctr} read
		\begin{equation}
		d\nu_0=\nu_1\wedge\nu_2,\qquad Jf_1=f_2,\qquad Jf_2=-f_1.
	\end{equation}
	The structural coefficients of the frame $f_1,f_2,f_0$ are the locally defined functions 
	\begin{equation}
		c_{ij}^k:=g([f_i,f_j],f_k), \qquad i,j,k = 0,1,2,
	\end{equation}
	or equivalently
	\begin{equation}
		d\nu_k= -\sum_{0\leq i<j=2} c_{ij}^k\nu_i\wedge \nu_j, \qquad k = 0,1,2.
	\end{equation}
	The next result is taken from \cite[Ch.\ 17]{Agrachev} where a different convention is used for the sign of structure coefficient.
	\begin{prop}\label{eq:strcoeff}
		Let $(M,\omega, g)$ be a three-dimensional contact sub-Riemannian manifold. Let $f_0$ be its Reeb vector field, and let $f_0,f_1,f_2$ be a local oriented orthonormal frame, then 
		\begin{align}
				[f_1,f_2] & =c_{12}^1f_1+c_{12}^2f_2-f_0, \label{eq:strctr_coeff1}\\
				[f_1,f_0] & =c_{10}^1f_1+c_{10}^2f_2, \label{eq:strctr_coeff2}\\
				[f_2,f_0] & =c_{20}^1f_1+c_{20}^2f_2. \label{eq:strctr_coeff3}
		\end{align}
		Moreover, it holds
		\begin{equation}\label{eq:traceless}
			c_{01}^1+c_{02}^2=0.
		\end{equation}
		Furthermore, the following quantities 
		\begin{align}
				\chi & :=\sqrt{\left(c_{01}^1\right)^2+\frac{1}{4}\left(c_{01}^2+c_{02}^1\right)^2}, \label{eq:chi} \\
				\kappa & : =f_1\left(c_{12}^2\right)-f_2\left(c_{12}^1\right)-\left(c_{12}^1\right)^2-\left(c_{12}^2\right)^2 +
				\frac{c_{02}^1-c_{01}^2}{2}, \label{eq:kappa}
		\end{align}
		do not depend on the choice of $f_1,f_2$ and therefore are globally defined functions $\chi,\kappa:M\to \mathbb R$.
	\end{prop}
	\begin{remark}\label{rmk:chikappa}
	The functions $\chi$ and $\kappa$ are invariant under the action of smooth sub-Riemannian isometries of $(M,\omega,g)$ (diffeomorphisms that preserve the sub-Riemannian distribution and its metric). The invariant $\chi$ vanishes if and only if the flow of the Reeb flow acts by sub-Riemannian isometries. The invariant $\kappa$ has a less obvious geometric interpretation. If $\chi=0$, then $\kappa$ is the Gaussian curvature of the surface obtained by locally quotienting $M$ under the action of the Reeb flow. More precisely, in this case $\kappa$ is constant along the Reeb orbits and, if $N=M/f_0$, i.e. the quotient of $M$ by the Reeb flow, is a smooth surface, then it inherits a Riemannian structure with Gaussian curvature $\kappa$. More generally let $U\subset M$ be an open subset. Then $(U,\omega|_U, g|_{U})$, is a smooth sub-Riemannian manifold. If the quotient manifold $U/\left(f_0|_{U}\right)$ is well-defined (which is the case when $U$ is sufficiently small), then $U/\left(f_0|_{U}\right)$ has the natural structure of Riemannian surface with Gaussian curvature $\kappa$. For more details see \cite[Ch.\ 17]{Agrachev}.
	\end{remark}
	\begin{example}[Standard sub-Riemannian contact structure]\label{ex:st_sr}
		Let $(\mathbb R^3,\omega_{\mathrm{st}})$ be the standard contact structure of \cref{ex:standard}. We define a sub-Riemannian metric $g_{\mathrm{st}}$ on $\xi_{\mathrm{st}}=\ker \omega_{\mathrm{st}}$ such that the following global vector fields are orthonormal:
		\begin{equation}
			f_1:=\partial_x+\frac{y}{2}\partial_z,\qquad f_2:=\partial_y-\frac{x}{2}\partial_z.
		\end{equation}
The normalized contact form and the Reeb vector field are, respectively
\begin{equation}
\omega_{\mathrm{st}}=dz+\frac{1}{2}(xdy-ydx), \qquad f_0=\partial_z.
\end{equation}
The contact sub-Riemannian manifold $(\R^3,\omega_{\mathrm{st}},g_{\mathrm{st}})$ is called the \emph{Heisenberg group}, and its metric invariants vanish: $\chi=\kappa=0$.
	\end{example}
	
We introduce in this paper a sub-Riemannian metric on the standard overtwisted structure, which will turn out to be natural for tightness radius estimates.
	\begin{example}[The standard sub-Riemannian overtwisted structure]\label{ex:st_ot_sr}
Let $(\mathbb R^3,\omega_{\mathrm{ot}})$ be the standard overtwisted structure of \cref{ex:standard_ot}. We define a sub-Riemannian metric $g_{\mathrm{ot}}$ on $\xi_{\mathrm{ot}}=\ker \omega_{\mathrm{ot}}$ such that the following vector fields (in cylindrical coordinates) are orthonormal:
		\begin{equation}
			f_1:=\partial_r,\qquad f_2:= \frac{1}{r} \cos\left(\frac{r^2}{2}\right)\partial_\theta-\frac{1}{r} \sin\left(\frac{r^2}{2}\right)\partial_z .
		\end{equation}
The normalized contact form and the Reeb vector field are, respectively
\begin{equation}\label{eq:ot-Reeb}
\omega_{\mathrm{ot}}=\cos\left(\frac{r^2}{2}\right)dz+\sin\left(\frac{r^2}{2}\right)d\theta, \qquad 			f_0=\cos\left(\frac{r^2}{2}\right)\partial_z+\sin\left(\frac{r^2}{2}\right)\partial_\theta.
\end{equation}
The metric invariants of $(\R^3,\omega_{\mathrm{ot}},g_{\mathrm{ot}})$ are	 $\chi = \frac{r}{2}$ and $\kappa = \frac{r^2}{2}$.
\end{example}
%
\begin{remark}
Note that $f_1,f_2$ of \cref{ex:st_ot_sr} are defined only out of the $z$ axis, however, the corresponding sub-Riemannian metric is globally well-defined and smooth.
\end{remark}
	The sub-Riemannian metric induces a length structure on horizontal curves; given $T>0$ and absolutely continuous $\gamma:[0,T]\to M$, with $\dot\gamma(t)\in\xi_{\gamma(t)}$ for a.e.\ $t\in[0,T]$, we denote 
	\begin{equation}
		\ell(\gamma):=\int_{0}^T \|\dot\gamma(t)\|_g \, \mathrm{d} t,
	\end{equation}
	where $\|\cdot\|_g$ is the norm induced on $\xi$ by $g$. We define the \emph{sub-Riemannian distance} as
	\begin{equation}
		d(q,p):=\inf\{\ell(\gamma) \mid \gamma\,\,\text{is horizontal},\,\gamma(0)=q,\,\,\gamma(T)=p\}.
	\end{equation}
	The next result justifies the above definition, and it is a consequence of the Rashevskii-Chow theorem, see e.g. \cite[Ch.\ 3]{Agrachev}.
	\begin{theorem}
			Let $(M,\omega, g)$ be a three-di\-men\-sional contact sub-Riemannian manifold, then $(M,d)$ is a metric space and the metric topology coincides with the manifold topology.   
	\end{theorem}
	A \emph{sub-Riemannian geodesic} is a horizontal curve $\gamma: [0,T] \to M$, parametrized with constant speed $\|\dot\gamma(t)\|_g = C$ such that $d(\gamma(t),\gamma(s)) = C |t-s|$ for all $t,s \in [0,T]$. It is well-known that, for contact structures, sub-Riemannian geodesics are smooth and are projections of solutions of a Hamiltonian system on $T^* M$, which we now describe.

	Let $f_1,f_2$ be a local orthonormal frame for the sub-Riemannian structure, $f_0$ be the Reeb field and let $h_0,h_1,h_2:T^* M\to \mathbb R$ be the associated hamiltonian functions 
	\begin{equation}\label{hi2}
		h_i(\lambda):=\langle\lambda, f_i \rangle, \qquad \lambda \in T^* M, \qquad i=0,1,2.
	\end{equation}
	If $c_{ij}^k$ are the structural coefficients of the frame $f_0,f_1,f_2$, i.e., $[f_i,f_j]=c_{ij}^kf_k$, then the Poisson brackets $\{\cdot,\cdot\} : C^\infty(T^* M) \times C^\infty(T^* M) \to C^\infty(T^* M)$ satisfy
	\begin{equation}
		\{h_i,h_j\}=\sum_{k=0}^2 c_{ij}^k h_k.
	\end{equation}
For any smooth function $h :T^*M \to \R$ recall that the symplectic gradient is defined as the unique smooth vector field $\vec{h} \in \Gamma(T^*M)$ such that $dh = \sigma(\cdot,\vec{h})$, where $\sigma$ is the canonical symplectic form on $T^*M$. Equivalently, $\vec{h}(\cdot) = \{h,\cdot\}$. The \emph{sub-Riemannian Hamiltonian} $H:T^* M\to \mathbb R$ is
	\begin{equation}\label{eq:SR_ham}
		H:=\frac{1}{2}\left(h_1^2+h_2^2\right).
	\end{equation}
	The right hand side of \eqref{eq:SR_ham} is independent of the choice of $f_1,f_2$ and thus yields a well-defined smooth function on $T^* M$. The corresponding Hamiltonian vector field satisfies
	\begin{equation}\label{eq:ham_vect}
	\vec{H} = h_1\vec{h}_1+h_2\vec{h}_2.
	\end{equation}
	Every sub-Riemannian geodesic $\gamma:[0,T] \to M$ is the projection of an integral curve of $\vec{H}$, namely there exists $\lambda \in T^* M$ such that 
	\begin{equation}
	\gamma(t) = \pi \circ e^{t \vec{H}}(\lambda), \qquad \forall t\in [0,T],
	\end{equation}
	and conversely, any sufficiently short segment of such projections is a sub-Riemannian geodesic. The following rescaling property holds, as a consequence of the homogeneity of the Hamiltonian:
\begin{equation}\label{eq:rescaling}
e^{\vec{H}}(\alpha \lambda) = \alpha e^{\alpha \vec{H}}(\lambda), \qquad \forall \alpha >0, \quad \lambda \in T^* M.
\end{equation}
	We recall the following sub-Riemannian version of the Hopf-Rinow theorem, see e.g.\ \cite{Agrachev}.
		\begin{theorem}\label{thm:hopf-rinow}
		Let $(M,\omega, g)$ be a three-di\-men\-sional contact sub-Riemannian manifold, then $(M,d)$ is a complete metric space if and only there exists $q \in M$ such that the restriction of the Hamiltonian flow $e^{t\vec{H}}|_{T_q^* M} : T_q^* M \to T^* M$ is well-defined for all $t \in \R$.
		\end{theorem}
	As a consequence the following definition is well-posed.
	\begin{defi}
		Let $(M,\omega, g)$ be a complete three-di\-men\-sional contact sub-Riemannian manifold. For $q\in M$, we define the exponential map $\exp_q: T_q^* M \to M$ as 
		\begin{equation}\label{eq:exp}
\exp_q(\lambda) := \pi\circ e^{\vec{H}}(\lambda),
		\end{equation}
		where $\pi:T^* M\to M$ is the bundle projection. 
	\end{defi}

\section{Exponential coordinates}\label{sec:expcoords}
We describe the contact structure in a special coordinate neighbourhood of a Reeb orbit. For the following, let $(M,\omega, g)$ a contact sub-Riemannian manifold. 
Let $\Gamma \subset M$ be an embedded piece of Reeb orbit, which we always assume to be connected, without boundary, and non-trivial. Maximal integral lines of the Reeb field are simply called \emph{Reeb orbits}. We can have two cases: either $\Gamma$ is diffeomorphic to an open interval or to a circle.

\subsection{Cylindrical coordinates on the annihilator bundle}\label{sec:cylindcoords}
We first study the annihilator bundle of $\Gamma$.
\begin{defi}\label{def:bundles}
	The \emph{annihilator bundle} of $\Gamma$ is
	\begin{equation}
		A\Gamma:=\{\lambda\in T^{*} M \mid \pi(\lambda)\in\Gamma,\,\,\langle\lambda\,, f_0\rangle=0\}.
	\end{equation}
	Given $r>0$ we also define the following bundles:
	\begin{equation}
		A^{<r}\Gamma:=A\Gamma\cap \{\sqrt{2H}<r\},\qquad A^1\Gamma:=A\Gamma\cap \{\sqrt{2H}=1\}.
	\end{equation}
\end{defi}
\begin{remark}
	The restriction $2H|_{A \Gamma} : A \Gamma \to \R$ is a fiber-wise positive definite quadratic form. In particular $A\Gamma$ is an Euclidean bundle.
\end{remark}
The following property is crucial for the construction of coordinates on $A\Gamma$.
\begin{lemma}
	The restriction $\xi|_{\Gamma}$ is a trivial vector bundle over $\Gamma$. In other words there exist a horizontal orthonormal frame $f_1, f_2$ defined in a neighbourhood of $\Gamma$.
\end{lemma}	
\begin{proof}
	If $\Gamma$ is diffeomorphic to an interval, then there is nothing to prove. Assume then that $\Gamma$ is diffeomorphic to a circle. This means that we can find $\tau>0$ and $q \in \Gamma$ such that the curve $\gamma:[0,\tau] \to M$ given by
	\begin{equation}
		\gamma(z)=e^{zf_0}(q),
	\end{equation}
	satisfies $\gamma(0)= \gamma(\tau)$ and $\Gamma = \gamma([0,\tau])$. Let $X_0,Y_0$ be a basis for $\xi_q$. Then 
	\begin{equation}
		X_z:=e^{zf_0}_* X_0,\qquad Y_z:=e^{zf_0}_* Y_0, \qquad z \in [0,\tau],
	\end{equation}
	is a basis for $\xi_{\gamma(z)}$, since the Reeb flow preserves $\omega$. Notice that $e^{\tau f_0}_* :\xi_q\to\xi_q$ is an orientation preserving endomorphism, 
	the orientation being the one induced on $\xi_q$ by $d\omega$. In particular $e^{\tau f_0}_*$ is homotopic to the identity: there exists a smooth curve $G:[0,\tau]\to \mathrm{GL}(2)$ such that 
	\begin{equation}
		G_{0}=\mathrm{Id},\qquad
		\begin{pmatrix}
			X_{\tau}\\
			Y_{\tau}
		\end{pmatrix}
		=
		G_{\tau}
		\begin{pmatrix}
			X_0\\
			Y_0
		\end{pmatrix}.
	\end{equation}
	We obtain a smooth trivializing frame for $\bar{X},\bar{Y}$ for $\xi|_{\Gamma}$ by setting
	\begin{equation}
		\begin{pmatrix}
			\bar{X}_{z}\\
			\bar{Y}_{z}
		\end{pmatrix}
		=
		G_{z}^{-1}
		\begin{pmatrix}
			X_z\\
			Y_z
		\end{pmatrix},\qquad \forall\,z\in [0,\tau].
	\end{equation}
	We can extend $\bar{X}$ and $\bar{Y}$ to an horizontal frame in a neighbourhood of $\Gamma$, which in turn yields a orthonormal horizontal frame $f_1, f_2$ by Gram-Schmidt.
\end{proof}
Let $\Gamma \subset M$ be an embedded piece of Reeb orbit. We recall that either $\Gamma$ is diffeomorphic to an open interval or to a circle. Fix $q \in \Gamma$ and let $\gamma : I \to \Gamma$ be the diffeomorphism
\begin{equation}
	\gamma(z) = e^{z f_0}(q), \qquad \forall\, z \in I,
\end{equation}
where $I$ is either an open interval or $I = \R/{\tau\N} \simeq \mathbb{S}^1$ for some $\tau>0$ (the period of the orbit).

Since  $\xi|_{\Gamma}$ is a trivial bundle over $\Gamma$, there exists an oriented orthonormal frame $f_0,f_1,f_2$ on $\Gamma$ and a corresponding dual frame $\nu_0 = \omega,\nu_1,\nu_2$, which we fix for the next proposition.
\begin{prop}\label{prop:trivial}
	The map $\Psi : I \times \R^2 \to A\Gamma$ defined by
	\begin{equation}\label{eq:Psi}
		\Psi(z,x,y):=(x\nu_1+y\nu_2)|_{\gamma(z)},
	\end{equation}
	is an isomorphism of Euclidean bundles satisfying
	\begin{equation}\label{eq:bundle_iso}
		2H \circ \Psi = \| \cdot\|^2_{\R^2}.
	\end{equation}
	In particular, all the bundles of \cref{def:bundles} are trivial, with
	\begin{equation}
		A\Gamma\simeq I \times \R^2,\qquad A^{<r}\Gamma\simeq  I \times B_{r},\qquad A^{1}\Gamma\simeq I \times \partial B_1,
	\end{equation}
	where $B_{r}$ is the 2-dimensional Euclidean open ball of radius $r >0$.
\end{prop}
\begin{proof}
	Any covector $\lambda \in T^*M|_{\Gamma}$ on $\Gamma$ can be written as 
	\begin{equation}
		\lambda=\sum_{i=0}^2\langle\lambda\,, f_i\rangle\nu_i=\sum_{i=0}^2h_i(\lambda)\nu_i,
	\end{equation}
	where the $h_i:T^*M \to \R$ are defined in \eqref{hi2}.		
	Notice that 
	\begin{equation}
		A\Gamma=\{\lambda\in T^* M \mid \pi(\lambda)\in \Gamma,\,\,h_0(\lambda)=0\},
	\end{equation}
	thus any $\lambda\in A\Gamma$ can be written in a unique way as $\lambda=x\nu_1+y\nu_2$, with $x,y\in\mathbb R^2$. It follows that \eqref{eq:Psi} is a bundle isomorphism
	Since $2H = h_1^2 + h_2^2$, then it holds $2H\circ \Psi = \|\cdot\|^2_{\R^2}$ as claimed. 
\end{proof}

\cref{prop:trivial} yields a global set of coordinates on $A\Gamma$. The corresponding cylindrical coordinates $(z,r,\theta)\in I \times (0,+\infty)\times [0,2\pi)$ defined outside of the zero section of $A\Gamma$, given by
\begin{equation}
	x =r\cos\theta,\qquad y=r\sin\theta,\qquad z,
\end{equation}
will be instrumental in our analysis, and referred to as \emph{cylindrical coordinates} on $A\Gamma$.

\subsection{Tubular neighbourhood of a Reeb orbit}
	We define the tubular neighbourhhod map.
	\begin{defi}
	 The tubular neighbourhood map is defined as 
	\begin{equation}\label{eq:exponential_map}
		E:A\Gamma\to M,\qquad E(\lambda)=\pi\circ\exp(\lambda).
	\end{equation}
	\end{defi}	
The tubular neighbourhood map at $\Gamma$ yields coordinates in a neighbourhood of $\Gamma$ such that the sub-Riemannian distance $\delta:M\to \mathbb R$ from $\Gamma$,  given by
	\begin{equation}\label{eq:dist_from_gam}
		\delta(q):=\inf_{p\in\Gamma}d(p,q),\qquad  \forall \, q\in M,
	\end{equation}
	has a simple expression. The next theorem is a particular case of a sub-Riemannian version of the tubular neighbourhood theorem which holds for general non-characteristic embedded submanifolds, see e.g.\ \cite[Thm.\ 3.1]{BRR-Tubes}. See also \cite[Sec.\ 2.4]{subOnR3} for an earlier version for three-dimensional contact structures.
	\begin{theorem}[Tubular neighbourhood]\label{thm:exp_map}
	There exists a neighbourhood $U \subset A \Gamma$ of the zero section of $A\Gamma$ such that the restriction of the tubular neighbourhood map $E|_{U}:U\to M$ is a diffeomorphism onto its image, and for the distance $\delta$ from $\Gamma$ it holds
		\begin{equation}\label{eq:gamma_dist_smooth}
			\delta\circ E|_{U}=\sqrt{2H} |_U.
		\end{equation}
	If $\Gamma$ has compact closure, then we can choose $U = A^{<r} \Gamma$ for some $r >0$. Furthermore if $q\in M$ and $\sigma:[0,T]\to M$ is a geodesic realizing the distance between $q$ and $\Gamma$, i.e. $\sigma(0)\in\Gamma$, $\sigma(T)=q$ and $\delta(q)=\ell(\sigma)=T$,
 then there exists $\lambda\in A^1\Gamma$ such that 
 \begin{equation}
     \sigma(t)=E(t\lambda),\quad \forall\,t\in[0,T].
 \end{equation}
	\end{theorem}
	 In the following, given $U\subset A\Gamma$, we will denote the restriction of the tubular neighbourhood map to $U$ with $E:U\to M$ instead of $E|_U:U\to M$.
	\begin{defi}\label{def:injgamma}
	Let $\Gamma$ be an embedded piece of Reeb orbit. The \emph{injectivity radius} of $\Gamma$ is
	\begin{equation}
		r_{\inj}(\Gamma):=\sup\{r >0 \mid \text{$E :A^{<r}\Gamma\to M$ is a diffeomorphism onto its image}\},
	\end{equation}
	with the understanding that $r_{\inj}(\Gamma)=0$ if the above set is empty. 
	\end{defi}
\begin{remark}\label{rmk:tube}
    Let $\Gamma$ be a (whole) Reeb orbit with $r_\inj(\Gamma)>0$. \cref{thm:exp_map} implies that
    \begin{equation}
        E\left(A^{<r}\Gamma\right)=B_r(\Gamma):=\{q\in M\mid d(q,\Gamma)<r\},\qquad  \forall\,r<r_\inj(\Gamma). 
    \end{equation}
\end{remark}

\section{The tightness radius from a Reeb orbit}\label{sec:tight_radius}

In this section we introduce the tightness radius and prove our first estimates for it.

\begin{defi}
Let $(M,\omega, g)$ be a three-di\-men\-sional contact sub-Riemannian manifold and let $\Gamma$ be an embedded piece of Reeb orbit with $r_{\inj}(\Gamma)>0$. The tightness radius is
\begin{equation}
		r_{\tight}(\Gamma):=\sup\left\{0<r<r_{\inj}(\Gamma) \mid \text{$(A^{<r}\Gamma, \ker E^*\omega)$ is tight}\right\}.
\end{equation}
\end{defi}
  \begin{remark}\label{rmk:tube_tight}
		    It follows from \cref{rmk:tube} that, if $\Gamma$ is a (whole) Reeb orbit with $r_\inj(\Gamma)>0$, then 
      \begin{equation}\label{eq:sr-tight-tube}
          r_\tight(\Gamma)=\sup\left\{0<r<r_\inj(\Gamma)\mid \left(B_r(\Gamma),\ker \omega|_{B_r(\Gamma)}\right)\,\,\text{is tight}\right\},
      \end{equation}
      where $B_r(\Gamma)$ is the set of points at sub-Riemannian distance from $\Gamma$ smaller than $r$.
		\end{remark}
	\subsection{A class of models and their tightness radius}\label{sec:models}
	
	We compute the tightness radius for a class of examples. For the next statement we say that a smooth function $\alpha : \R^3 \to \R$ is \emph{radial} if, in cylindrical coordinates $(r,\theta,z)$ it holds $\partial_\theta \alpha =\partial_z \alpha = 0$.
	\begin{theorem}\label{thm:alpha_beta}
		Let $\omega$ be a contact form on $\mathbb R^3$, such that in cylindrical coordinates on $\R^3$
		\begin{equation}
			\omega=\alpha dz+\beta d\theta,
		\end{equation}
		for some smooth radial functions $\alpha,\beta$. Let $\gamma := \alpha\partial_r \beta -\beta\partial_r\alpha$. Then
		\begin{equation}
			g:=\left.\left( dr\otimes dr+\frac{\gamma^2}{\alpha^2+\beta^2}(d\theta\otimes d\theta+dz\otimes dz)\right)\right|_{\ker\omega}
		\end{equation}
		is a sub-Riemannian metric, and $\omega$ is normalized. Thus $(\R^3,\omega,g)$ is a contact sub-Riemannian structure. The curve
		\begin{equation}
		\Gamma=\{(0,0,z)\in\mathbb R^3 \mid z\in\mathbb R\}
		\end{equation}
		is the orbit of the Reeb field passing through the origin. In cylindrical coordinates on $A\Gamma$ of \cref{sec:cylindcoords} the map $E: A\Gamma \to \R^3$ is the identity:
		\begin{equation}
		E(z,r\cos\theta,r\sin\theta) = (r \cos\theta,r \sin \theta, z),
		\end{equation}
		and in particular the distance $\delta:\R^3\to \R$ from $\Gamma$ is
		\begin{equation}
		\delta(r \cos\theta,r \sin \theta, z) = r.
		\end{equation}
		Moreover, the frame $\{f_0,N,JN\}$ is orthonormal and oriented, where
		\begin{equation}\label{eq:explicit_normal_frame}
				f_0=\frac{\partial_r\beta}{\gamma}\partial_z-\frac{\partial_r\alpha}{\gamma}\partial_\theta,\qquad N=\partial_r,\qquad JN =\frac{\alpha}{\gamma}\partial_\theta-\frac{\beta}{\gamma}\partial_z.
		\end{equation}
		Finally, the tightness and injectivity radii are:
		\begin{equation}
			r_{\inj}(\Gamma)=+\infty,\qquad r_{\tight}(\Gamma)=\inf\{r>0 \mid \beta(r)=0\}.
		\end{equation}
	\end{theorem}
	\begin{proof}
	Since $\alpha,\beta \in C^\infty(\R^3)$ are smooth radial functions, the maps 
	\begin{equation}
	\R \ni r \mapsto \alpha(r\cos\theta,r\sin\theta,z),\qquad \beta(r\cos\theta,r\sin\theta,z),
	\end{equation}
	are smooth as functions of one variable and furthermore $\partial_r^{2j+1} \alpha(0,z)=\partial_r^{2j+1} \beta(0,z)= 0$ for all $j \in \N$ and $z \in \R$, where here and in the following we employ the shorthand $(0,z)=(0,0,z) \in \R^3$.
		The fact that $\omega$ is a smooth contact form poses some additional constraints. First, notice that
		\begin{equation}
			\omega=\alpha dz+\frac{\beta}{r^2}(xdy-ydx).
		\end{equation}
		Therefore $\beta(0,z) =0$ and $\beta/r^2$ must be a smooth function. Furthermore, $\omega$ is a contact form and thus it cannot vanish: from $\omega|_{r=0}=\alpha(0,z)dz$, we deduce that the radial function $\alpha$ satisfies $\alpha_{0}:=\alpha(0,z)=\alpha(0,0)\neq 0$. Additionally by the contact condition
		\begin{equation}\label{eq:gamma/r_pos}
		\omega\wedge d\omega=\frac{\gamma}{r} dx\wedge dy\wedge dz \qquad \Rightarrow \qquad \frac{\gamma}{r} \text{ is smooth and never vanishing}.
		\end{equation}
		Combining \eqref{eq:gamma/r_pos} with the fact that $\beta/r^2$ is smooth we deduce
		\begin{equation}\label{eq:limgamma/r}
			\lim_{r\to 0}\frac{\gamma}{r}=\lim_{r\to 0}\frac{\alpha\partial_r \beta -\beta\partial_r\alpha}{r}=\lim_{r\to 0}\frac{\alpha\partial_r \beta}{r}=\alpha_0 \lim_{r\to 0}\frac{\partial_r \beta}{r} \neq 0.
		\end{equation}
		One can check that the vector fields $N,JN$ of \eqref{eq:explicit_normal_frame} are orthonormal for $g$. Moreover
		\begin{equation}
			\omega(f_0)=1,\qquad d\omega(f_0,\cdot)=0,\qquad d\omega(N,JN)=1.
		\end{equation}
		Thus $g$ is a sub-Riemannian metric, $\omega$ is normalized, and $f_0$ is the Reeb field of $\omega$.
		
		Using \eqref{eq:gamma/r_pos} and \eqref{eq:limgamma/r} we obtain that
		\begin{align}
			\lim_{r\to 0}f_0 & =\lim_{r\to 0}\frac{\partial_r\beta /r}{\gamma/r}\partial_z-\lim_{r\to 0}\frac{\partial_r\alpha /r}{\gamma/r}\partial_\theta= \frac{1}{\alpha_0}\partial_z.
		\end{align}
		In particular $\Gamma=\{(0,0,z) \mid z \in \R \}$ is the orbit of $f_0$ passing through the origin.
		
		The Hamiltonian functions of the vector fields \eqref{eq:explicit_normal_frame} are   
		\begin{equation}
			h_{N}=p_r,\qquad h_{JN}=\frac{\alpha}{\gamma}p_\theta-\frac{\beta}{\gamma}p_z,\qquad h_0=\frac{\partial_r\beta}{\gamma}p_z-\frac{\partial_r\alpha}{\gamma}p_\theta,
		\end{equation}
		where the $p$'s are the canonical momenta associated to $(r,\theta, z)$.
		Observe that 
		\begin{equation}
			A\Gamma=\{\lambda\in T^* \mathbb R^3 \mid \pi(\lambda)\in \Gamma,\,\,p_z(\lambda)=0\}=\{\lambda\in T^* \mathbb R^3 \mid \lambda=\left(p_x dx+p_y dy\right)|_{(0,0,z)}\}.
		\end{equation}
		Therefore, the sub-Riemannian Hamiltonian satisfies
		\begin{equation}
			2H=p_r^2+h_{JN}^2=p_r^2+\left(\frac{\alpha}{\gamma}p_\theta-\frac{\beta}{\gamma}p_z\right)^2.
		\end{equation}
		Hamilton's equation in coordinates are then 
		\begin{equation}
			\begin{cases}
				\dot p_r=-h_{JN}\partial_r h_{JN},\\
				\dot p_\theta=0,\\
				\dot p_z=0,
			\end{cases}
			\quad
			\begin{cases}
				\dot r=p_r,\\
				\dot \theta=h_{JN}\partial_{p_\theta}h_{JN},\\
				\dot z=h_{JN}\partial_{p_z}h_{JN}.
			\end{cases}
		\end{equation}
		Since $h_{JN}|_{p_{\theta},p_{z}=0}=0$, then for each $p_x,p_y,z\in \mathbb R$ the curve 
		\begin{equation}
				\R \ni t\mapsto 
				(tp_{x},tp_{y},z, p_x,p_y,0) \in T^*\R^3,
		\end{equation}
		is an integral curve of $\vec{H}$ with initial condition $(0,0,z,p_x,p_y,0)\in A\Gamma$. Recall that $E A\Gamma \to \R^3$ is given by $E = \pi\circ e^{\vec{H}}|_{A\Gamma}$. Therefore in cylindrical coordinates on $A\Gamma$ (see \cref{sec:cylindcoords}) we have
		\begin{equation}
		E(z,r\cos\theta,r\sin\theta) = (r \cos\theta,r \sin \theta, z),
		\end{equation}
		as claimed.	It follows that $r_{\inj}(\Gamma)=+\infty$. 
		
		To conclude, we compute the tightness radius. For $R>0$ we denote with $C_{R}=\{(x,y,x)\in\mathbb R^3 \mid x^2+y^2<R^2\}$ the sub-Riemannian cylinder around $\Gamma$ of radius $R$. Let us denote 
		\begin{equation}
			R_0 = \inf\{r>0 \mid \beta(r)=0\}.
		\end{equation}
		If $R>R_0$ then, according to \cref{def:ot_disk}, for each $z\in\mathbb R$ the disk 
		\begin{equation}
			D_{0}=\{(r \cos\theta, r\sin\theta, z)\in\mathbb R^3 \mid r \leq R_0,\,\,\theta\in [0,2\pi)\},
		\end{equation}
		is diffeomorphic to the standard overtwisted disk, thus $r_{\tight}(\Gamma)\leq R_0$. To prove the opposite inequality, we will build a diffeomorphism $\varphi: C_{R_0} \to C_{R_0}$ such that $\varphi_*\xi_{\mathrm{st}} = \xi$, where $\xi_{\mathrm{st}}$ is the standard contact structure of \cref{ex:standard}, which is tight. 
		
		To motivate the construction, note that for any such diffeomorphism there is a non-vanishing function $f\in C^{\infty}(C_{R_0})$ such that 
		\begin{equation}
			\varphi^* (f\omega)=\omega_{\mathrm{st}},\qquad \text{on $C_{R_0}$},
		\end{equation}
		where $\omega_{\mathrm{st}}=dz+\tfrac{r^2}{2} d\theta$. In particular, $\varphi$ maps orbits of the Reeb field of $\omega_{\mathrm{st}}$ to orbits of the Reeb field of $\bar\omega :=f \omega$. Since the Reeb field of $\omega_{\mathrm{st}}$, which is $\partial_z$, does not have periodic orbits, we select $f$ in such a way that the Reeb field $\bar{f}_0$ of $\bar\omega$ has none either, and use it to build the claimed diffeomorphism.

		We now proceed with the construction of $\varphi$. By definition of $R_0$, the smooth function $\beta/r^2$ is non-vanishing for $r\in[0,R_0)$, thus we set 
		\begin{equation}
		\bar\omega := \left(\frac{r^2}{2\beta}\right)\omega = \bar{\alpha} dz + \bar{\beta} d\theta, \qquad \text{where} \qquad \bar{\alpha} := \frac{r^2}{2} \frac{\alpha}{\beta},\quad \bar\beta := \frac{r^2}{2}.
		\end{equation}
		For the contact form $\bar\omega$ all previous computation holds, replacing $\alpha,\beta$ and $\gamma$ with their barred counterparts, and in particular from \eqref{eq:explicit_normal_frame} the Reeb field is
		\begin{equation}
			\bar{f}_0=\frac{\partial_r\bar\beta}{\bar\gamma}\partial_z-\frac{\partial_r\bar\alpha}{\bar\gamma}\partial_\theta.
		\end{equation}
		In particular, using \eqref{eq:gamma/r_pos}, we observe that
		\begin{equation}
			dz(\bar{f}_0)=\frac{\partial_r\bar\beta}{\bar\gamma}=\left(\frac{\bar\gamma}{r}\right)^{-1} \qquad \text{is non-vanishing},
		\end{equation}
		so that $\bar{f}_0$ has no periodic orbits. 
		Notice that $\bar f_0$ is a vector field tangent to the cylinders $C_R$ and its components are radial functions, hence it is complete.	We claim that the map $\varphi: C_{R_0} \to C_{R_0}$ 
		\begin{equation}\label{ot_embedding}
				\varphi(x,y,z):= e^{z\bar f_o}(x,y,0),
		\end{equation}
		is a diffeomorphism such that $\varphi^*\bar{\omega} = \omega_{\mathrm{st}}$. We prove that latter property first, in fact:
		\begin{align}
				\varphi^*\bar\omega&=\bar\omega(\varphi_*\partial_r)dr+\bar\omega(\varphi_*\partial_\theta)d\theta+\bar\omega(\varphi_*\partial_z)dz\\
				&=\bar\omega(e^{z\bar f_0}_*\partial_r)dr+\bar\omega(e^{z\bar f_0}_*\partial_\theta)d\theta+\bar\omega(\bar f_0)dz\\
				&=\bar\omega(\partial_r)dr+\bar\omega(\partial_\theta)d\theta+dz=dz+\frac{r^2}{2}d\theta=\omega_{\mathrm{st}}.
		\end{align}
		In particular $\varphi$ is an immersion, since it sends a contact form to a contact form. Moreover $\varphi$ is injective: since $dz(\bar f_0)>0$, then the equality $\varphi(x,y,z)=\varphi(x',y',z')$ implies $z=z'$, which in turns implies $x=x'$, $y=y'$. The map is also surjective since the integral curves of $\bar f_0$ foliate the cylinders $C_R$, $R\in(0,R_0)$. This concludes the proof of the claim.
	\end{proof}

	As a corollary of \cref{thm:alpha_beta} we compute the tightness radius for \cref{ex:st_sr,ex:st_ot_sr}.
	\begin{corollary}\label{cor:model_radii_st}
		The curve $\Gamma=\{(0,0,z) \mid z\in\mathbb R\}$ is an orbit of the Reeb field for the standard sub-Riemannian contact structure of \cref{ex:st_sr}. In cylindrical coordinates:
  \begin{equation}
E^*\omega_{\mathrm{st}}=dz+\frac{r^2}{2}d\theta.
  \end{equation}
It holds:
		\begin{equation}
				r_{\inj}(\Gamma)=+\infty,\qquad r_{\tight}(\Gamma)=+\infty.
		\end{equation}
%
	\end{corollary}
	\begin{corollary}\label{cor:model_radii_ot}
		The curve $\Gamma=\{(0,0,z) \mid z\in\mathbb R\}$ is an orbit of the Reeb field for the standard sub-Riemannian overtwisted structure of \cref{ex:st_ot_sr}.  In cylindrical coordinates: 
		\begin{equation}
       E^*\omega_{\mathrm{ot}}=\cos\left(\frac{r^2}{2}\right)dz+\sin\left(\frac{r^2}{2}\right)d\theta.
  \end{equation}
	It holds:
		\begin{equation}
				r_{\inj}(\Gamma)=+\infty,\qquad r_{\tight}(\Gamma)=\sqrt{2\pi}.
		\end{equation}
	\end{corollary}
	
\subsection{Tightness radius and singular locus}\label{sec:tightandsingular}

	The tightness radius can be estimated studying the characteristic foliation of the fibers of the annihilator bundle. 
	\begin{defi}[Singular locus of Reeb orbits]\label{def:singlocus}
	Let $(M,\omega, g)$ be a three-di\-men\-sional contact sub-Riemannian manifold and let $\Gamma$ be an embedded piece of Reeb orbit with $\rho := r_{\inj}(\Gamma)>0$. Consider, for every $q\in \Gamma$, the embedded surfaces
	\begin{equation}
	\Sigma_q := \{E(\lambda) \mid \lambda \in A^{<\rho}\Gamma,\,\, \pi(\lambda)=q,\,\,\lambda \neq 0\}.
	\end{equation}
	The singular locus of $\Gamma$ is the union, for $q \in \Gamma$, of the singularities of the characteristic foliations of $\Sigma_q$ in the contact manifold $(M,\xi=\ker\omega)$, namely
	\begin{equation}\label{eq:singular_locus}
	\mathrm{Sing}(\Gamma) := \bigcup_{q \in \Gamma} \Big\{ p \in \Sigma_q \,\Big|\, T_{p} \Sigma_q = \xi_p\Big\}.
	\end{equation}
	\end{defi}
	\begin{theorem}\label{thm:singlocusbound}
		Let $(M,\omega, g)$ be a three-di\-men\-sional contact sub-Riemannian manifold and let $\Gamma$ be an embedded piece of Reeb orbit with $r_{\inj}(\Gamma)>0$. Then $\mathrm{Sing}(\Gamma)\subset M$ is an embedded surface and 
		\begin{equation}\label{eq:tight>sing}
			r_{\tight}(\Gamma)\geq \min\{d(\Gamma, \mathrm{Sing}(\Gamma)),r_{\inj}(\Gamma)\},
		\end{equation}
	with the understanding that, if $\mathrm{Sing}(\Gamma)=\emptyset$ then $r_{\tight}(\Gamma) = r_{\inj}(\Gamma)$.
	
		The equality holds in \eqref{eq:tight>sing} for the models of \cref{thm:alpha_beta}, and in particular of the standard sub-Riemannian overtwisted structure of \cref{ex:st_ot_sr}.
	\end{theorem}
	\begin{proof}
Let $\rho = r_{\inj}(\Gamma)$, and recall that $E:A^{<\rho}\Gamma \to M$ is a diffeomorphism onto its image. On $A\Gamma$ we fix a set of cylindrical coordinates $(r\cos\theta,r\sin\theta,z)$ as defined in \cref{sec:cylindcoords}.

Since $E_*\partial_r = \pi_* \vec{H}$ is horizontal, there exist $\phi, N\in C^{\infty}(A^{<\rho}\Gamma)$, with $N>0$, such that
		\begin{equation}\label{eq:pb_form}
			E^* \omega=N(\cos\phi dz+\sin\phi d\theta).
		\end{equation}
		Combining \cref{thm:exp_map} and \cref{prop:trivial} it holds
		\begin{equation}\label{eq:delta=r}
		\delta(E(r\cos\theta,r\sin\theta,z)) = r, \qquad \forall\, r<\rho.
		\end{equation}
		Furthermore, if $q=E(0,0,z)$, then in these coordinates
		\begin{equation}
		E^{-1}(\Sigma_q) = \{ (r\cos\theta,r\sin\theta,z) \mid r \in (0,\rho),\,\,\theta \in [0,2\pi)\}.
		\end{equation}
		Hence, singularities of the characteristic foliation are characterized by $E^*\omega \propto dz$, and
		\begin{equation}\label{eq:sing=level}
		E^{-1}(\mathrm{Sing}(\Gamma)) = \phi^{-1}(\pi \N).
		\end{equation}
		Computing the contact condition $E^* (\omega\wedge d\omega)>0$ in cylindrical coordinates yields
		\begin{equation}
			E^* (\omega\wedge d\omega)=\frac{N^2\partial_r\phi}{r} r dr\wedge d\theta \wedge dz.
		\end{equation}
		It follows that $f:= \partial_r \phi/r$ is smooth and positive on $A^{<\rho} \Gamma$. In particular $\phi:A^{<\rho}\Gamma\to\mathbb R$ is monotone increasing along radial lines $r \mapsto (r\cos\theta,r\sin\theta,z)$. Therefore, from equality \eqref{eq:sing=level}, we deduce that $\mathrm{Sing}(\Gamma)$ is an embedded surface. Moreover, by \eqref{eq:delta=r}, we deduce that
		\begin{equation}\label{eq:dist_gamma_sing}
			d(\Gamma,\mathrm{Sing}(\Gamma))=\inf\{r>0 \mid (r\cos\theta,r\sin\theta,z) \in A^{<\rho}\Gamma,\,\,\phi(r\cos\theta,r\sin\theta,z)=\pi\}.
		\end{equation}

		To conclude, observe that 
				\begin{equation}
				\frac{\phi}{H}=\frac{2\phi}{r^2}=\frac{2}{r^2}\int_{0}^r sf(s\cos\theta,s\sin\theta,z)\,\mathrm{d}s
				=2\int_{0}^1 sf(rs\cos\theta,rs\sin\theta,z)\,\mathrm{d}s.
		\end{equation}
It follows that the function $\phi/H$ is well-defined on $A^{<\rho}\Gamma$, smooth and positive. Using this fact, define the smooth function $\varphi: A^{<\rho}\Gamma \to \R^3$
		\begin{equation}\label{eq:contact_embedding}
				\varphi(x,y,z)= \left( \sqrt{\frac{\phi}{H}}x, \sqrt{\frac{\phi}{H}}y,z\right).
		\end{equation}
		By the monotonicity of $\phi$ we have that $\varphi$ is injective. Furthermore, letting $\omega_{\mathrm{ot}}=\cos(r^2/2)dz+\sin(r^2/2)d\theta$ be the standard contact overtwisted form of \cref{ex:standard_ot}, we have
		\begin{equation}\label{eq:almost_isometry}
			\varphi^*\omega_{\mathrm{ot}}=\cos \phi dz+\sin\phi d\theta=\frac{E^*\omega}{N}. 
		\end{equation}
		It follows that $\varphi$ is an immersion, and hence a diffeomorphism onto its image. Denoting $R:=\min\{d(\Gamma, \mathrm{Sing}(\Gamma)),\rho\}$, by equation \eqref{eq:dist_gamma_sing}, monotonicity of the angle $\phi$ and the second equality of \eqref{eq:almost_isometry}, the image of the restricted map 
		\begin{equation}
			\varphi:A^{<R}\Gamma \to \mathbb R^3
		\end{equation}
		is contained in the cylinder $C_{\sqrt{2\pi}}=\{(x,y,z)\in\mathbb R^3 \mid x^2+y^2<2\pi\}\subset\mathbb R^3$, which, by \cref{cor:model_radii_ot} is a tight contact manifold when endowed with the restriction of $\ker\omega_{\mathrm{ot}}$. Thus in view of \eqref{eq:almost_isometry} the contact manifold $(A^{<R}\Gamma, \ker E^{*}\omega)$ contains no overtwisted disks. The inequality \eqref{eq:tight>sing} is proven in the case in which $\Gamma$ is not a periodic orbit.
		
If $\Gamma$ is periodic, let us denote with $N:=E(A^{<\rho}\Gamma)$ its tubular neighbourhood and consider the universal cover, $\tilde N$. The latter inherits the natural structure of a contact sub-Riemannian manifold locally isometric to $(A^{<\rho}\Gamma, E^*\omega, E^*g)$. According to \cref{prop:trivial}, $N\simeq\Gamma\times B_{\rho}$, therefore $\tilde N=\tilde\Gamma\times B_\rho\simeq \mathbb R\times B_{\rho}$, where $\tilde\Gamma$ is the unique lift of $\Gamma$ going through the point $(0,0,0)$, i.e. $\tilde\Gamma=\{(z,0,0)\in\tilde N\mid z\in\mathbb R\}$. Since $\tilde\Gamma$ is not periodic, the already proved part of the theorem applies, i.e.
  \begin{equation}
      r_{\tight}\left(\tilde\Gamma\right)\geq \tilde d\left(\tilde\Gamma, \mathrm{Sing}\left(\tilde\Gamma\right)\right).
  \end{equation}
  Let $p:\tilde N\to N$ be the canonical projection. Observe that, for any $q\in \tilde\Gamma$, $A^{<\rho}_q \tilde \Gamma$ has neighbourhood $\mathcal U$ such that the map $p: U\to p(U)$ is a sub-Riemannian isometry. By \cref{def:singlocus} of the singular locus, it follows that 
  \begin{equation}
      \tilde d\left(\tilde\Gamma, \mathrm{Sing}\left(\tilde\Gamma\right)\right)=d(\Gamma, \mathrm{Sing}(\Gamma)).
  \end{equation}
  Moreover, \cref{rmk:univ_tight} implies that $r_{\tight}(\Gamma)\geq r_{\tight}\left(\tilde\Gamma\right)$. Therefore
  \begin{equation}
      r_{\tight}(\Gamma)\geq r_{\tight}\left(\tilde\Gamma\right)\geq \tilde d\left(\tilde\Gamma, \mathrm{Sing}\left(\tilde\Gamma\right)\right)=d(\Gamma, \mathrm{Sing}(\Gamma)) ,
\end{equation}
concluding the proof.
\end{proof}

\section{Contact Jacobi curves and quantitative tightness}\label{sec:contactJacobi}

In (sub-)Riemannian geometry and geometric control theory, Jacobi curves are curves in the Lagrange Grassmannian whose dynamics is intertwined with the presence of conjugate points and curvature (see \cite{AG-Feeback,AZ-JacobiI,AZ-JacobiII,CurvVar} and references within). Inspired by that approach, we introduce the concept of \emph{contact} Jacobi curves, whose dynamics is related with the presence of overtwisted disks. As a preparation, in \cref{sec:Schwarzianderivative} we recall the concept of Schwarzian derivative and some comparison results adapted to our setting.

\subsection{The Schwarzian derivative}\label{sec:Schwarzianderivative}

\begin{defi}
	Let $I\subset \mathbb R$ be an open interval and let $\varphi:I\to \mathbb \RP^1$ be a smooth immersion, which in homogeneous coordinates reads
	\begin{equation}
		\varphi(t)=[ \varphi_1(t):  \varphi_2(t)].
	\end{equation}
	The \emph{Schwarzian derivative} of $\varphi$, denoted $\mathcal S(\varphi):I\to\mathbb R$, is defined as 
	\begin{equation}\label{eq:def_schwarz}
		\mathcal S(\varphi ):=\frac{\dddot v}{\dot v}-\frac{3}{2}\left(\frac{\ddot v}{\dot v}\right)^2, \qquad \text{where} \qquad v =\frac{\varphi_1}{\varphi_2},\quad \text{or} \quad \frac{\varphi_2}{\varphi_1}.
	\end{equation}
The Schwarzian derivative is independent of the choice of the homogeneous coordinate $v$. In particular if $A:\mathbb P^1\to \mathbb P^1$ is any projective transformation then $\mathcal S(A\circ \varphi)=\mathcal S(\varphi)$.
\end{defi}

\begin{prop}\label{prop:curve-and-ODE}
Consider the second order linear ODE on an open interval $I \subset \R$
\begin{equation}\label{eq:ODE}
\ddot{u}(t) + q(t) u(t) =0, \tag{$\star$}
\end{equation}
where $q \in C(I)$. There is a one-to-one correspondence between pairs of linearly independent solutions of \eqref{eq:ODE} (up to constant rescaling) and smooth immersions $\varphi: I \to \RP^1$ with Schwarzian derivative $\mathcal{S}(\varphi) = 2q$.
\end{prop}
\begin{proof}
The space of solutions of \eqref{eq:ODE} is a vector space of dimension $2$. Choose two linearly independent solutions $u_1,u_2$. Their Wronskian is
\begin{equation}
W(u_1,u_2) := \det\begin{pmatrix}
u_1 & u_2 \\
\dot{u}_1 & \dot{u}_2
\end{pmatrix} = u_1\dot{u}_2 - u_2\dot{u}_1.
\end{equation}
From \eqref{eq:ODE} it follows that $W(u_1,u_2)$ is a non-zero constant. In particular, $u_1,u_2$ cannot vanish simultaneously on $I$. Thus, the smooth curve $\varphi: I \to \RP^1$ given by
\begin{equation}\label{eq:immersionschwarz}
\varphi(t)=[u_1(t): u_2(t)],
\end{equation}
is well-defined. One can check that it is an immersion and that $2q = \mathcal{S}(\varphi)$.

We show that this correspondence is injective. Let $\tilde{u}_1,\tilde{u}_2$ be another choice of linearly independent solutions  of \eqref{eq:ODE} yielding the same curve in $\RP^1$. Then there is $a(t) \neq 0$ such that $\tilde{u}_i(t) = a(t) u_i(t)$ for all $t\in I$ and $i=1,2$. It follows that $a:I\to\R$ is smooth (since solutions cannot vanish simultaneously) and by \eqref{eq:ODE} that
\begin{equation}
\ddot{a}(t)u_i(t)+ 2 \dot{a}(t) \dot{u}_i(t)=0,\qquad i=1,2.
\end{equation}
Multiply the first equation by $u_2(t)$, the second one by $u_1(t)$, and take the difference. We get
\begin{equation}
\dot{a}(t)W(u_1,u_2) =0  \qquad \Longrightarrow \qquad \dot{a}(t) =0,
\end{equation}
so that $a(t)=a$ must be a non-zero constant. Thus $(\tilde{u}_1,\tilde{u}_2) = a(u_1,u_2)$.

We show that the correspondence is surjective. Let $\varphi:I\to \RP^1$ be an immersion, with
\begin{equation}
\varphi(t) = [\varphi_1(t):\varphi_2(t)],
\end{equation}
for some smooth $\varphi_i:I \to \R$. The immersion condition implies that the function
\begin{equation}
\beta : I\to \R, \qquad \beta(t):= \left(\frac{1}{|\dot{\varphi}_1(t) \varphi_2(t) - \dot{\varphi}_2(t)\varphi_1(t)|}\right)^{1/2},
\end{equation}
is well-defined. We note that $u_i:= \beta \varphi_i :I \to \R$ satisfy
\begin{equation}\label{eq:nonzero}
\dot{u}_1 u_2 - \dot{u}_2 u_1 = \pm 1.
\end{equation}
Therefore, the $u_i$ are linearly independent solutions of \eqref{eq:ODE}, for some $q : I \to \R$. By the first part of the proof, since $\varphi(t) = [u_1(t):u_2(t)]$ for $t \in I$, we necessarily have $2q = \mathcal{S}(\varphi)$.
\end{proof}

We will apply the characterization of \cref{prop:curve-and-ODE} to contact Jacobi curves. An important fact is that the contact Jacobi curves \emph{always} have singular Schwarzian derivative at the initial time. The corresponding ODE \eqref{eq:ODE} is singular, and so are their solutions. Suitable assumptions on the singularity are then necessary. More precisely, we assume that $2q = \mathcal{S}(f)$ satisfies
\begin{equation}
q = -\frac{3}{4t^2} + \text{regular part}, \qquad t\in (0,T).
\end{equation}
This assumption is used for the following asymptotic properties of solutions at the singularity.

\begin{lemma}\label{lem:propofsolutions}
Consider the ODE \eqref{eq:ODE}, and assume that
\begin{equation}\label{eq:potentialassumption}
q(t)= -\frac{3}{4t^2} + \tilde{q}(t), \qquad \forall\,t\in (0,T),
\end{equation}
where $\tilde{q} \in C([0,T))$. Then \eqref{eq:ODE} has linearly independent solutions $g_1,g_2 : (0,T)\to \R$ such that
\begin{align}
g_1(t) & = t^{3/2} \tilde{g}_1(t), \\
g_2(t) & = t^{-1/2} \tilde{g}_2(t),
\end{align}
where $\tilde{g}_j \in C^0([0,T))$, $\tilde{g}_j(0)=1$ for $j=1,2$. 
\end{lemma}
\begin{proof}
If $\tilde{q}(t)$ is analytic the result is well-known from the theory second order ODEs with regular singularities, see e.g. \cite[Ch.\ 5]{TheoryODE}. The case under investigation is treated (in greater generality) in \cite[Thm.\ 2.1]{regularsingular}, where it is attributed to M.\ B\^ocher.
\end{proof}

\begin{prop}\label{prop:intersVSzeroes}
Let $f: [0,T)\to \RP^1$ be a continuous curve, such that its restriction to $(0,T)$ is a smooth immersion. Consider the ODE \eqref{eq:ODE} with $2q=\mathcal{S}(f)$. Assume that 
\begin{equation}
q(t)= -\frac{3}{4t^2} + \tilde{q}(t), \qquad \forall\,t\in (0,T),
\end{equation}
where $\tilde{q} \in C([0,T))$. Let $t_*\in (0,T)$ such that the curve $f$ has a self-intersection: $f(0) = f(t_*)$. Then there exists a non-trivial solution $u:(0,T)\to \R$ of \eqref{eq:ODE} such that
\begin{equation}\label{eq:u0=ustar}
u(0) :=\lim_{t\to 0}u(t) = u(t_*) =0.
\end{equation}
\end{prop}
\begin{proof}
By \cref{prop:curve-and-ODE} there are two independent solutions $u_1,u_2$ of \eqref{eq:ODE} such that $f(t) = [u_1(t):u_2(t)]$ for $t\in (0,T)$. Since $f(0) = f(t_*)$ we can assume without loss of generality that the following limits exist and are finite
\begin{equation}\label{eq:limitexists}
\alpha:=\lim_{t\to 0} \frac{u_1(t)}{u_2(t)} = \lim_{t\to t_*} \frac{u_1(t)}{u_2(t)}.
\end{equation}
In particular since $u_1,u_2$ cannot vanish simultaneously on $(0,T)$ we must have $u_2(t_*)\neq 0$. We let then $u := u_1 -\alpha u_2$. By construction, $u(t_*) = 0$. To show that $u(0) = 0$ we use \cref{lem:propofsolutions}. There exist $a,b,c,d \in \R$ with $ad-bc \neq 0$ such that
\begin{align}
u_1 & = a g_1 -b g_2,\\
u_2 & = c g_1 - d g_2,
\end{align}
where $g_1,g_2$ are independent solutions. By \eqref{eq:limitexists} we have
\begin{equation}
\alpha = \lim_{t\to 0} \frac{ a t^{3/2} - b t^{-1/2}}{c t^{3/2} -d t^{-1/2}} = \lim_{ t\to 0} \frac{ a t^2 - b}{c t^{2} - d}.
\end{equation}
By the existence of the limit and since $b,d$ cannot vanish at the same time, we see that $d\neq 0$ and $\alpha = b/d$. It follows that
\begin{equation}
u = u_1-\alpha u_2 = a g_1 - \frac{b}{d} c g_1 = \frac{ad-bc}{d} g_1.
\end{equation}
Since $g_1(t)  = t^{3/2} \tilde{g}_1(t)$, with $\tilde{g}_1(0)=1$ we have that $\lim_{t\to 0} u(t) = 0$.
\end{proof}

We have the following version of the Sturm-Picone comparison theorem with singularities.
\begin{prop}\label{prop:sturmpicone}
Let $u,\bar{u}: (0,T) \to \R$ be non-trivial solutions of
\begin{align}
\ddot{u}(t) & + q(t) u(t)  = 0,  \label{eq:ODEcomp} \\ 
\ddot{\bar{u}}(t) & + \bar{q}(t) \bar{u}(t) = 0, \label{eq:ODEcompmod}
\end{align}
where $q: (0,T) \to \R$ satisfies $q(t) \leq \bar{q}(t)$ for $t\in (0,T)$. 
Then between any two consecutive zeroes $t_1<t_2 \in [0,T)$ of $u$ there exists at least one zero of $\bar{u}$, unless $q\equiv \bar{q}$ on $(t_1,t_2)$.
\end{prop}
\begin{proof}
Assume without loss of generality that $u>0$ on $(t_1,t_2)$. We claim that there exists a sequence $\varepsilon_n \to 0$ such that $\dot{u}(t_1+\varepsilon_n) >0$. In fact, if for all small enough $\varepsilon>0$ we had $\dot{u} \leq 0$ on $(0,\varepsilon)$, then using that $u$ is continuous on $[t_1,T)$ with $u(t_1)=0$ we have
\begin{equation}
u(t) = \int_{t_1}^t \dot{u}(\tau)\,\mathrm{d} \tau \leq 0,\qquad \forall t \in (t_1,t_1+\varepsilon),
\end{equation}
which is a contradiction. Similarly one shows that $\dot{u}(t_2)<0$.

Assume by contradiction that $\bar{u}>0$ on $(t_1,t_2)$. Therefore by multiplying \eqref{eq:ODEcomp}-\eqref{eq:ODEcompmod} by $\bar{u},u$, respectively, subtracting, and integrating on $(t_1+\varepsilon_n,t_2)$, we obtain
\begin{equation}
[\dot{u}\bar{u}-u\dot{\bar{u}}]_{t_1+\varepsilon_n}^{t_2} = \int_{t_1+\varepsilon_n}^{t_2}(\bar{q}(t)-q(t)) u(t) \bar{u}(t)\,\mathrm{d}t.
\end{equation}
Now we take the limit of the above for $\varepsilon_n \to 0$. By the properties $\dot{u}(t_1+\varepsilon_n)>0$ and $\dot{u}(t_2)<0$, the limit of the left hand side exists and is non-positive. On the other hand, observe that the limit of the right hand side exists (possibly $+\infty$) since it is a non-decreasing sequence, and it must be $\geq 0$. Furthermore the limit is zero if and only if $\bar{q} \equiv q$ on $(t_1,t_2)$. If this is not the case then we reach a contradiction, so that $\bar{u}$ must have a zero in the interval $(t_1,t_2)$.
\end{proof}

\begin{theorem}\label{thm:comparisonScwharzian}
Let $f: [0,T)\to \RP^1$ be a continuous curve, such that its restriction to $(0,T)$ is a smooth immersion. Assume that its Schwarzian derivative $2q=\mathcal{S}(f)$ satisfies
\begin{equation}\label{eq:comparisonSchwarzian-assumption1}
q(t)+\frac{3}{4t^2} \in  C([0,T)).
\end{equation}
Assume also that there exists $\bar{q} \in C((0,T))$ such that
\begin{equation}
q \leq \bar{q}, \qquad \text{on $(0,T)$}.
\end{equation}
Let $t_* \in (0,T)$ such that $f(0) =f(t_*)$. Let $\bar{u}$ be a non-trivial solution of $\ddot{\bar{u}}+\bar{q}\bar{u}=0$. Then
\begin{equation}\label{eq:tstarestimate}
t_* \geq  \sup_{\bar{u}}  \inf\{ t>0 \mid \bar{u}(t) = 0\},
\end{equation}
where the sup is over all non-trivial solutions of $\ddot{\bar{u}}+\bar{q}\bar{u}=0$.

 If all solutions of $\ddot{\bar{u}}+\bar{q}\bar{u}=0$ on $(0,T)$ have a (possibly infinite) limit at $t=0$, and there are solutions with $\bar{u}(0)=0$ then the sup is attained on this set.
\end{theorem}
\begin{remark}\label{rmk:sol0at0}
In particular, when $\bar{q}$ has the form of \cref{lem:propofsolutions}, then all solutions have (possibly infinite) limit at $t=0$, and there is only one solution (up to rescaling) with $\bar{u}(0)=0$. Then such a solution yields the optimal estimate in \eqref{eq:tstarestimate}. This is the case in which we use \cref{thm:comparisonScwharzian}.
\end{remark}
\begin{proof}
By \cref{prop:intersVSzeroes} there is a non-trivial solution $u:(0,T)\to \R$ of $\ddot{u}+qu=0$ such that
\begin{equation}\label{eq:u0=ustar}
u(0) :=\lim_{t\to 0}u(t) = u(t_*) =0.
\end{equation}
By \cref{prop:sturmpicone}, the non-trivial solution $\bar{u}:(0,T)\to \R$ of $\ddot{\bar{u}}+\bar{q}\bar{u}=0$ must have a zero in $(0,t_*)$. Optimizing over all possible solutions we obtain \eqref{eq:tstarestimate}.

To prove the final part of the statement, it is sufficient to show that under the additional assumptions the separation theorem holds: if $\bar{u},\bar{v}:(0,T)\to \R$ are linearly-independent solution, with $\bar{v}(0) = \lim_{t\to 0} v(t) \neq 0$, and $0 = t_1<t_2< T$ are consecutive zeroes of $\bar{u}$, then $\bar{v}$ has exactly one zero in $(t_1,t_2)$. To prove this claim, assume by contradiction that $\bar{v}$ has no zeroes on $(t_1,t_2)$. It follows that $\bar{v}$ is non-vanishing on $[t_1,t_2]$. In fact, if $t_1,t_2>0$ this is true since $\bar{u},\bar{v}$ are linearly independent, while $\bar{v}(t_1)= \bar{v}(0)\neq 0$ by assumption. Then the function
\begin{equation}
\psi(t):=\frac{\bar{u}(t)}{\bar{v}(t)},\qquad t\in [t_1,t_2],
\end{equation}
is continuous on $[t_1,t_2]$, $\dot{\psi}$ exists on $(t_1,t_2)$ and $\psi(t_1)=\psi(t_2)=0$. By Rolle's theorem there is $c\in (t_1,t_2)$ with $\psi'(c) =0$. Hence, the Wronskian of $\bar{u},\bar{v}$ is zero at $c$ which is a contradiction.
\end{proof}

\subsection{Contact Jacobi curves}

We introduce the contact Jacobi curve, and study its properties. Without further mention, in this section we assume that the sub-Riemannian manifold $(M,\omega,g)$ is complete, so that the Hamiltonian flow is well-defined for all times.

\begin{defi}\label{def:contactjacobi}
	Let $(M,\omega, g)$ be a three-di\-men\-sional complete contact sub-Riemannian manifold, and let $\Gamma$ be an embedded piece of Reeb orbit. Consider the following one-parameter family of one-forms on $A^1\Gamma$:
	\begin{equation}\label{eq:moving_omega}
		\omega_r:=\left.\left((\pi\circ e^{r\vec{H}})^{*}\omega\right)\right|_{A^1\Gamma},\qquad r\in\mathbb [0,+\infty).
	\end{equation}
	The \emph{contact Jacobi curve} at $\lambda\in A^1\Gamma$ is the projectivization of $\omega_r|_{\lambda}$:
	\begin{equation}\label{eq:contact_Jacobi}
			\Omega_\lambda:[0,+\infty) \to P (T_{\lambda}^* A^{1}\Gamma)\simeq\RP^1,\qquad
			\Omega_\lambda(r):=P\left(\omega_{r}|_{\lambda}\right).
	\end{equation}
	The \emph{first singular radius} of the contact Jacobi curve is
	\begin{equation}\label{eq:r_o}
		r_o(\lambda):=\inf\{ r>0 \mid \Omega_{\lambda}(r)=\Omega_{\lambda}(0)\}.
	\end{equation}
\end{defi}
\begin{remark}
The contact Jacobi curve $\Omega_\lambda$ is defined for any $\lambda\in T^*M$ such that $\langle \lambda, f_0\rangle =0$ and $2H(\lambda)=1$. In fact, the definition does not depend the choice of $\Gamma$: if $\Gamma'$ is any other piece of Reeb orbit containing $\pi(\lambda)$, then $T_{\lambda}(A^1\Gamma) = T_{\lambda}(A^1\Gamma')$, yielding the same contact Jacobi curve $\Omega_\lambda$. The same observation holds for \cref{def:focalradius}.
\end{remark}
\begin{prop}\label{prop:Estaromega}
	Let $(M,\omega, g)$ be a three-dimensional contact sub-Riemannian manifold, and let $\Gamma$ be an embedded piece of Reeb orbit. Let $E:A\Gamma\to M$ be the tubular neighbourhood map \eqref{eq:exponential_map} and let $\omega_r$ be the form defined in \eqref{eq:moving_omega}, then
	\begin{equation}
		\left(E^*\omega\right)|_{r\lambda}=\omega_r|_\lambda\qquad \forall\,\lambda\in A^1\Gamma,\quad\forall\,r\geq 0.
	\end{equation}
	In particular, denoting with $(r\cos\theta,r\sin\theta,z)$ the cylindrical coordinates on $A\Gamma$ introduced in \cref{sec:cylindcoords}, we have
	\begin{equation}\label{eq:Estaromega}
		\left(E^*\omega\right)|_{r\lambda}=\omega_r(\partial_z|_{\lambda})dz+\omega_r(\partial_\theta|_{\lambda})d\theta,\qquad \forall\,\lambda\in A^1\Gamma,\quad\forall\,r\geq 0.
	\end{equation}
	Furthermore, the contact Jacobi curve at $\lambda\in A^1\Gamma$ has the following expression 
	\begin{equation}\label{eq:coordsOmega}
		\Omega_\lambda:[0,+\infty)\to\mathbb {RP}^1,\qquad \Omega_\lambda(r)=[E^*\omega(\partial_{\theta}|_{r\lambda})\,: \,E^*\omega(\partial_{z}|_{r\lambda})].
	\end{equation}
\end{prop}
\begin{proof}
	Let $(r\cos\theta,r\sin\theta,z)$ be cylindrical coordinates on $A\Gamma$ as in \cref{sec:cylindcoords}. Note that, out of the zero section, it holds
	\begin{equation}
		T_{\lambda} (A^1\Gamma) \simeq \spn\{\partial_{\theta}|_{\lambda},\partial_z|_{\lambda} \}.
	\end{equation}
	Writing $E^*\omega$ in these coordinates, we obtain
	\begin{align}
		(E^*\omega)|_{r\lambda} & = \omega(E_*\partial_r|_{r\lambda})dr+\omega(E_*\partial_\theta|_{r\lambda})d\theta+\omega(E_*\partial_z|_{r\lambda})dz\\
		& = \omega\left(\pi_*\vec{H}|_{e^{\vec{H}}(r\lambda)}\right)dr+\omega\left((\pi \circ e^{r\vec{H}})_*\partial_\theta|_{\lambda}\right)d\theta+\omega\left((\pi \circ e^{r\vec{H}})_*\partial_z|_{\lambda}\right)dz\\
		& = \omega_r(\partial_\theta|_{\lambda})d\theta+\omega_r(\partial_z|_{\lambda})dz=\omega_r|_\lambda,
	\end{align}
	where in the second line we used that fact that $\pi_*\vec{H}|_{\lambda}$ is horizontal, and we also used the homogeneity property of the Hamiltonian \eqref{eq:rescaling}. Equation \eqref{eq:coordsOmega} follows combining equation \eqref{eq:Estaromega} with the fact that $d\theta,dz$ is a basis for $T^*_\lambda A^1\Gamma$. 
\end{proof}

\begin{defi}\label{def:focalradius}
Let $(M,\omega, g)$ be a three-di\-men\-sional contact sub-Riemannian manifold, and let $\Gamma$ be an embedded piece of Reeb orbit. A real number $r>0$ is called a \emph{focal radius} for $\lambda\in A^1\Gamma$ if 
\begin{equation}
    d_{r\lambda}E:T_{r\lambda}( A^1\Gamma)\to T_{E(r\lambda)} M\qquad \text{is not an isomorphism}.
\end{equation}
\end{defi}
\begin{remark}\label{rmk:rfoc>rinj}
Note that $r_{\foc}(\lambda)>0$ since all points on the zero section of $A\Gamma$ are regular points for $E$. Furthermore, by definition of $r_{\inj}(\Gamma)$ (see \cref{def:injgamma}), it holds $r_{\foc}(\lambda) \geq r_{\inj}(\Gamma)$.
\end{remark}

The next result is a crucial structure theorem for the contact Jacobi curve. In particular it is well-defined, is an immersion outside of focal radii (which is a discrete set of points), and its Schwarzian satisfies the assumptions for the general comparison theory developed in \cref{sec:Schwarzianderivative}. 

\begin{theorem}\label{thm:singwelldef}
Let $(M,\omega, g)$ be a three-di\-men\-sional contact sub-Riemannian manifold, and let $\Gamma$ be an embedded piece of Reeb orbit. Let $\lambda \in A^1\Gamma$.
\begin{enumerate}[label = $(\roman*)$]
\item The set of focal radii for $\lambda$ is discrete, separated from zero;
\item The contact Jacobi curve $\Omega_\lambda:[0,+\infty)\to \RP^1$ is well-defined and smooth;
\item The contact Jacobi curve is an immersion outside of zero and the focal radius;
\item\label{item:asymptoticsfoc} For any focal radius $\bar{r}>0$, the following asymptotic formula for the Schwarzian derivative holds: there exists $k=k(\bar{r}) \in \{0,1,2,3,4\}$ such that
\begin{equation}\label{eq:asymptoticsfoc}
\mathcal{S}(\Omega_\lambda)(r) = -\frac{k(k+2)}{2(r-\bar{r})^2} + k\, O\left(\frac{1}{r-\bar{r}}\right) + O(1);
\end{equation}
\item \label{item:asymptotics0} The following asymptotic formula for the Schwarzian derivative holds at $r=0$:
\begin{equation}\label{eq:asymptotics0}
\mathcal{S}(\Omega_\lambda)(r) = -\frac{3}{2r^2} + O(r).
\end{equation}
\end{enumerate}
The term $O(r)$ denotes a smooth remainder. Namely there exist $\varepsilon >0$ and a smooth function $f: (-\varepsilon,\varepsilon) \to \R$, both depending on $\lambda$, such that $ O(r) = rf(r)$ for all $r\in (-\varepsilon,\varepsilon)$.
\end{theorem}
\begin{remark}
A case-by-case analysis shows that only the values $k=0,1,2$ can occur in \eqref{eq:asymptoticsfoc}, see \cite{PhDEugenio}. This also shows that  the contact Jacobi curve can be an immersion at focal radii, in which case $k=0$. As it will be clear from the proof, it is never an immersion at $t=0$. 
\end{remark}
\begin{proof}
Let $(r\cos\theta,r\sin\theta,z)$ be cylindrical coordinates on $A\Gamma$ as in \cref{sec:cylindcoords}. Making use of \cref{prop:Estaromega}, and in particular of equation \eqref{eq:Estaromega}, we compute the form $E^*(\omega\wedge d\omega)$
\begin{equation}\label{eq:Estarcontact}
        E^*(\omega\wedge d\omega)|_{r\lambda}= \frac{\omega_r\wedge\dot\omega_r(\partial_\theta|_{\lambda},\partial_z|_{\lambda})}{r}rdr\wedge d\theta \wedge dz,
    \end{equation}
    where the dot denotes the derivative w.r.t.\ parameter $r$.

    To proceed, we need the following claims: for any $\lambda \in A^1\Gamma$ and $\bar{r}\geq 0$ the one-parameter family of one-forms $\omega_r$, evaluated at $\lambda$, has finite order at $\bar{r}$, namely there exists $m\in \{0,1,2,3\}$ and a non-zero one-form $\beta$ on $A^1\Gamma$ such that 
\begin{equation}\label{eq:claim1}
\omega_r|_{\lambda} = (r-\bar{r})^m \beta|_{\lambda} + O\left((r-\bar{r})^{m+1}\right),
\end{equation}
where $O((r-\bar{r})^{m+1})$ denotes a smooth remainder. Furthermore, the one-parameter family of two-forms has a similar property: there exist $n \in \{0,1,2,3,4\}$ (depending on $\lambda,\bar{r})$, and a non-zero two-form $\alpha$ on $A^1\Gamma$ such that
\begin{equation}\label{eq:claim2}
\omega_r\wedge\dot\omega_r|_{\lambda} = (r-\bar{r})^n \alpha|_{\lambda} + O\left((r-\bar{r})^{n+1}\right).
\end{equation}
The proof of \eqref{eq:claim1}-\eqref{eq:claim2} is postponed to \cref{app:contidots} (see \cref{lem:claimdots}).

In the following we will often omit the evaluation of forms at the fixed $\lambda \in A^1\Gamma$, for brevity.
		
\textbf{Proof of $(i)$.} By \eqref{eq:Estarcontact}, and since $\omega\wedge d\omega$ is a volume form on $M$, we note that $\bar{r}>0$ is a focal radius if and only if $\omega_{\bar{r}}\wedge\dot{\omega}_{\bar{r}} =0$. In particular it follows from \eqref{eq:claim2} that the set of focal radii is discrete. Furthermore, by definition, $r=0$ cannot be a focal time since $E$ is always a diffeomorphism in a neighbourhood of the zero section of $A\Gamma$.

\textbf{Proof of $(ii)$ and $(iii)$.} If $r>0 $ is non-focal, $\omega_{r}\wedge\dot{\omega}_{r} \neq 0$. In particular $\omega_{r} \neq 0$. It follows that the contact Jacobi curve
\begin{equation}
\Omega_\lambda(r) = P(\omega_r|_{\lambda}) \simeq [\omega_r(\partial_\theta):\omega_r(\partial_z)],
\end{equation}
is well-defined and smooth in a neighbourhood of $r$. Furthermore, assuming $\omega_r(\partial_z)\neq 0$ we have
\begin{equation}
\dot v(r)= -\frac{\omega_r\wedge\dot\omega_r(\partial_\theta,\partial_z)}{\omega_r(\partial_z)^2} \neq 0,
\end{equation} 
and similarly if $\omega_r(\partial_\theta|_{\lambda})\neq 0$. It follows that the contact Jacobi curve is an immersion at all non-focal radii $r>0$. Note that by \eqref{eq:Estarcontact}, since $rdr\wedge d\theta\wedge dz \neq 0$ and smooth at $r=0$, we must have $\omega_0\wedge\dot{\omega}_0 = 0$ so that at $r=0$ the contact Jacobi curve is not an immersion.

We now prove how the contact Jacobi curve is well-defined and smooth also in a neighbourhood of the focal radii. Let $\bar{r}>0$ be a focal radius. By \eqref{eq:claim1} it follows that in a neighbourhood of $\bar{r}$ the one-parameter family of one-forms on $A^1\Gamma$ 
\begin{equation}\label{eq:baromega}
\bar{\omega}_r:=\frac{\omega_r}{(r-\bar{r})^m},
\end{equation}
is smooth as a function of $r$ and non-zero (when evaluated at $\lambda$). Thus, for all $r\neq \bar{r}$ in a neighbourhood of $\bar{r}$, it holds
\begin{equation}
\Omega_\lambda(r)  = P(\omega_r|_{\lambda}) = P(\bar{\omega}_r|_{\lambda}) \simeq [\bar{\omega}_r(\partial_\theta): \bar{\omega}_r(\partial_z)],
\end{equation}
and so the contact Jacobi curve extends smoothly through $\bar{r}$.

\textbf{Proof of $(iv)$.} Let $\bar{r}>0$ be a focal radius. Let $\bar{\omega}_r$ be the  one-parameter family of one-form in \eqref{eq:baromega}, which is smooth and non-zero in a neighbourhood of $\bar{r}$. Assume without loss of generality that $\bar{\omega}_{\bar{r}}(\partial_z)\neq 0$. Then $\upsilon(r):=\bar{\omega}_r(\partial_\theta)/\bar\omega_r(\partial_z)$ defines a smooth coordinate for the contact Jacobi curve in a neighbourhood of $\bar{r}$, and furthermore for all $r \neq \bar{r}$ in a neighbourhood of $\bar{r}$ it holds
\begin{equation}\label{eq:upsilondot}
\dot{\upsilon}(r) = \frac{d}{dr}\frac{\bar{\omega}_r(\partial_\theta)}{\bar{\omega}_r(\partial_z)} = \frac{d}{dr}\frac{\omega_r(\partial_\theta)}{\omega_r(\partial_z)} = -\frac{\omega_r\wedge\dot\omega_r(\partial_\theta,\partial_z)}{\omega_r(\partial_z)^2}.
\end{equation}
By \eqref{eq:claim2} and the last term of \eqref{eq:upsilondot} it follows that there exists $k\in \N$ and $c\neq 0$ such that
\begin{equation}\label{eq:upsilondotTaylor}
\dot{\upsilon}(r) = c (r-\bar{r})^k + O(r-\bar{r})^{k+1}.
\end{equation}
On one hand, $k\geq 0$ since from the first equality of \eqref{eq:upsilondot} $\dot{\upsilon}$ is smooth around $\bar{r}$. On the other hand $k\leq 4$ by \eqref{eq:claim2}. By plugging \eqref{eq:upsilondotTaylor} in the formula for the Schwarzian derivative \eqref{eq:def_schwarz}, we obtain the following asymptotic expansion as $r\to \bar{r}$:
\begin{align}
\mathcal{S}(\Omega_\lambda)(r) & = \frac{\dddot{\upsilon}}{\dot{\upsilon}} - \frac{3}{2}\left(\frac{\ddot{\upsilon}}{\dot{\upsilon}}\right)^2 \\
& = \frac{k(k-1)}{(r-\bar{r})^2} -\frac{3k^2}{2(r-\bar{r})^2} + k \, O\left(\frac{1}{r-\bar{r}}\right) + O(1) \\
& = -\frac{k(k+2)}{2(r-\bar{r})^2} + k\, O\left(\frac{1}{r-\bar{r}}\right) + O(1),
\end{align}
concluding the proof of \eqref{eq:asymptoticsfoc}.

\textbf{Proof of $(v)$.} The proof of \eqref{eq:asymptotics0} is similar, but we can exploit more precise computations at $r=0$. Assume without loss of generality that $\omega_0(\partial_z)\neq 0$. Consider the smooth coordinate $v(r)=\omega_r(\partial_\theta)/\omega_r(\partial_z)$, in a neighbourhood of $r=0$. We have: 
    \begin{equation}\label{eq:vdot}
        \dot v(r)= -\frac{\omega_r\wedge\dot\omega_r(\partial_\theta,\partial_z)}{\omega_r(\partial_z)^2} = -\frac{\omega_r(\partial_\theta)\dot{\omega}_r(\partial_z)-\omega_r(\partial_z)\dot{\omega}_r(\partial_\theta)}{\omega_r(\partial_z)^2}.
    \end{equation}
Routine computations, which are presented in \cref{app:contidots} (see \cref{lem:claimdots}), show that
\begin{align}
\omega_0(\partial_\theta) & = 0, & \ddot{\omega}_0(\partial_\theta) &= 1, & \omega_0(\partial_z) & = 1, \label{eq:dots1}\\
\dot{\omega}_0(\partial_\theta) &= 0,  & \dddot{\omega}_0(\partial_\theta) &= 0,
& \dot{\omega}_0(\partial_z) & = 0. \label{eq:dots2}
\end{align}
We deduce then from \eqref{eq:vdot} that $\dot{v}(r) = r + r^3 f(r)$, for some $f\in C^\infty([0,\varepsilon))$. From the definition of Schwarzian derivative \eqref{eq:def_schwarz} and elementary computation we obtain
\begin{equation}\label{eq:S_order_1}
\mathcal{S}(\Omega_\lambda)(r) = -\frac{3}{2r^2} + 3\dot{f}(0)r + O(r^2),
\end{equation}
which corresponds to \eqref{eq:asymptotics0}.
\end{proof}

We continue the discussion on the standard contact (resp.\ overtwisted) structure of \cref{ex:st_sr,ex:st_ot_sr}, computing their contact Jacobi curve and Schwarzian derivative.

\begin{example}[Contact Jacobi curve for the standard contact structure]\label{ex:st_scwharz}
Let $\Gamma=\{(0,0,z)\mid z\in \mathbb R\}$, which is an embedded Reeb orbit of the standard sub-Riemannian contact structure on $\mathbb R^3$ (see \cref{cor:model_radii_st}). Combining \cref{cor:model_radii_st} with equation \eqref{eq:coordsOmega} we deduce that, for all $\lambda \in A^1\Gamma$ the contact Jacobi curve is
\begin{equation}
\Omega_\lambda(r)=\left[\frac{r^2}{2} : 1\right],\qquad \forall\,r\in\mathbb R.
\end{equation}
A homogeneous coordinate is $v=r^2/2$ and the Schwarzian derivative is
\begin{equation}
    \mathcal S(\Omega_\lambda)(r)=-\frac{3}{2r^2}.
\end{equation}
\end{example}
Let $f,g:\mathbb R\to \mathbb R$ be smooth functions. In the following corollary we exploit the following useful formulas for the Schwarzian derivative
\begin{equation}
    \mathcal S(f\circ g)=\left(\mathcal S(f)\circ g\right)(g')^2+\mathcal S(g).
\end{equation}
\begin{example}[Contact Jacobi curve for the standard overtwisted structure]\label{ex:ot_scwharz}
Let $\Gamma=\{(0,0,z)\mid z\in \mathbb R\}$, which is an embedded Reeb orbit of the standard sub-Riemannian contact structure on $\mathbb R^3$ (see \cref{cor:model_radii_ot}). Combining \cref{cor:model_radii_ot} with equation \eqref{eq:coordsOmega} we deduce that, for all $\lambda \in A^1\Gamma$ the contact Jacobi curve is
\begin{equation}
		\Omega_\lambda(r)=\left[\sin\left(\frac{r^2}{2}\right): \cos\left(\frac{r^2}{2}\right)\right],\qquad \forall\,r\in\mathbb R.
\end{equation}
A homogeneous coordinate is $v=\tan(r^2/2)$. Since $\mathcal S(\tan(r))=2$, the Schwarzian derivative is
\begin{equation}
    \mathcal S(\Omega_\lambda)(r)=-\frac{3}{2r^2}+2r^2.
\end{equation}
\end{example}

\subsection{A formula for the Schwarzian derivative of a contact Jacobi curve}
	We provide a formula relating the Schwarzian derivative of a contact Jacobi curve to the sub-Riemannian distance from $\Gamma$. For the proof we refer to \cref{app:formula_for_S}.
	
	Let $(M,\omega,g)$ be a contact sub-Riemannian manifold and let $\Gamma \subset M$ be an embedded piece of Reeb orbit. We assume here that $\rho=r_{\inj}(\Gamma)>0$, which is always the case if $\Gamma$ has compact closure by \cref{thm:exp_map}, and $\delta$ (the distance function from $\Gamma$) is smooth on $E(A^{<\rho}\Gamma)\setminus \Gamma$. There, its horizontal gradient is well-defined, as the unique horizontal vector field $\nabla \delta$ satisfying
	\begin{equation}
		d\delta(Y)=g(\nabla\delta, Y),\qquad \forall \, Y\in \xi,
	\end{equation}
	which is also called the \emph{horizontal normal} to $\Gamma$, and hence denoted by $N = \nabla \delta$.
	
	\begin{defi}\label{def:adaptedframe_main}
		Let $N=\nabla \delta$ be the horizontal normal, and let $J: \xi \to \xi$ be the almost-complex structure \eqref{eq:complex_strctr}. The ordered triple $\{f_0,N,JN\}$, which is an oriented orthonormal frame defined on $E(A^{<\rho}\Gamma)\setminus \Gamma$, is called \emph{adapted frame} associated to $\Gamma$.
	\end{defi}	
	\begin{prop}\label{prop:schw_expr_main}
		Given $\lambda\in A^1\Gamma$ and $r\in (0,r_{\inj}(\Gamma))$ the Schwarzian derivative of the contact Jacobi curve \eqref{eq:contact_Jacobi} can be expressed in terms of the adapted frame of \cref{def:adaptedframe_main} as 
		\begin{equation}\label{eq:schw_expr_main}
			\frac{1}{2}\mathcal S(\Omega_{\lambda})(r)=g([N,f_0],JN)-\frac{1}{2}Ng([N,JN],JN)-\frac{1}{4}g([N,JN],JN)^2,
		\end{equation}
		where the right hand side is evaluated at $E(r\lambda)$.
	\end{prop}

\subsection{A sharp tightness radius estimate via contact Jacobi curves}\label{sec:tigthnessJacobi}

The first singular radius of a contact Jacobi curve detects the presence of overtwisted disks.

\begin{theorem}
\label{thm:dist=conj}
	Let $(M,\omega, g)$ be a complete three-di\-men\-sional contact sub-Riemannian manifold and let $\Gamma$ be an embedded piece of Reeb orbit with $r_{\inj}(\Gamma)>0$. Then, letting
	\begin{equation}
r_o^-(\Gamma) := \inf \{ r_o(\lambda)\mid \lambda \in A^1\Gamma \}, \qquad 	r_o^+(\Gamma) := \sup \{ r_o(\lambda)\mid \lambda \in A^1\Gamma \},
	\end{equation}	
	the following estimates hold:
	\begin{equation}\label{eq:upperandlower}
	\min\left\{r_{\inj}(\Gamma), r_o^-(\Gamma)\right\}\leq r_{\tight}(\Gamma) \leq \min\left\{r_{\inj}(\Gamma),r_o^+(\Gamma)\right\}.
	\end{equation}
	Moreover, if $r_o^+(\Gamma) < r_{\inj}(\Gamma)$, then for any $q\in \Gamma$, the set
	\begin{equation}\label{eq:Dqdisk}
	D_q:=\{E(r\lambda) \mid \lambda \in A_q^1 \Gamma,\,\, r\leq r_o(\lambda)\}
	\end{equation}
	is an overtwisted disk, and thus $(M,\omega)$ is an overtwisted contact manifold.
\end{theorem}
\begin{proof}
We first prove the lower bound in \eqref{eq:upperandlower}. We characterize the singular locus of the Reeb orbit and employ \cref{thm:singlocusbound}. Let $\rho = r_{\inj}(\Gamma)$, and recall that $E:A^{<\rho}\Gamma \to M$ is a diffeomorphism onto its image by \cref{thm:exp_map}. Let $(r\cos\theta,r\sin\theta,z)$ be cylindrical coordinates on $A\Gamma$ as explained in \cref{sec:cylindcoords}. By the same computation in the proof of \cref{thm:singwelldef} we have
\begin{equation}\label{eq:Epullbackomegawithomegar}
E^*\omega|_{r\lambda} = \omega_r(\partial_\theta|_{\lambda})d\theta + \omega_r(\partial_z|_{\lambda}) dz,
\end{equation}
for all $(r,\lambda)\in (0,+\infty)\times  A^1\Gamma$. By definition of singular locus (see \cref{def:singlocus}) we see that
\begin{equation}
\mathrm{Sing}(\mathrm{\Gamma}) = E\Big(\left\{ r\lambda \in A\Gamma \mid (r,\lambda)\in (0,r_{\inj}(\Gamma))\times  A^1\Gamma,\,\, \omega_r(\partial_{\theta}|_{\lambda})=0 \right\}\Big).
\end{equation}
Note that $\omega_0(\partial_{\theta}) = \omega(\pi_*\partial_{\theta}) =0$ by construction. For all $(r,\lambda)\in (0,r_{\inj}(\Gamma))\times  A^1\Gamma$ we have then the following characterization in terms of the contact Jacobi curve $\Omega_\lambda(\cdot)$:
\begin{equation}\label{eq:char-of-singular}
E(r\lambda) \in \mathrm{Sing}(\Gamma) \quad \Longleftrightarrow\quad \Omega_{\lambda}(r) = \Omega_{\lambda}(0).
\end{equation}
Using characterization \eqref{eq:char-of-singular} we obtain that, for all $\lambda \in A^1\Gamma$, it holds
\begin{align}
r_0(\lambda) & = \inf\{ r>0 \mid \Omega_{\lambda}(r) = \Omega_{\lambda}(0)\}  \\
 & \leq \inf\{ 0<r<r_{\inj}(\Gamma) \mid \Omega_{\lambda}(r) = \Omega_{\lambda}(0)\} \label{eq:inequality} \\
 & =  \inf\{\delta(E(r\lambda)) \mid E(r\lambda) \in \mathrm{Sing}(\Gamma) \},
\end{align}
where $\delta$ is the sub-Riemannian distance from $\Gamma$.

Assume now that $\mathrm{Sing}(\Gamma)\neq \emptyset$. This means that there exists $(r,\lambda) \in (0,r_{\inj}(\Gamma))\times A^1\Gamma$ such that $\Omega_{\lambda}(r)=\Omega_{\lambda}(0)$ and thus inequality \eqref{eq:inequality} is actually an equality. Therefore, taking the infimum over all $\lambda \in A^1\Gamma$ we obtain
\begin{equation}
r_0^-(\Gamma) = d(\Gamma,\mathrm{Sing}(\Gamma)) \leq r_{\tight}(\Gamma),
\end{equation}
where the last inequality is \cref{thm:singlocusbound}. If $\mathrm{Sing}(\Gamma) = \emptyset$ then again by \cref{thm:singlocusbound} we have that $r_{\tight}(\Gamma) = r_{\inj}(\Gamma)$. Combining both cases we obtain the left inequality of \eqref{eq:upperandlower}.

We now prove the upper bound in \eqref{eq:upperandlower}. By definition $r_{\tight}(\Gamma) \leq r_{\inj}(\Gamma)$, so we proceed under the assumption $r_o^+(\Gamma)<r_{\inj}(\Gamma)$. It is sufficient to show that the sets \eqref{eq:Dqdisk} are overtwisted disks.

As explained in the proof of \cref{thm:singlocusbound} there exist smooth functions $\phi, N\in C^{\infty}(A^{<\rho}\Gamma)$, with $N>0$, such that
		\begin{equation}\label{eq:pb_form2}
			E^* \omega=N(\sin\phi d\theta+\cos\phi dz).
		\end{equation}
Comparing \eqref{eq:pb_form2} with \eqref{eq:Epullbackomegawithomegar} yields, for all $(r,\lambda)\in (0,r_{\inj}(\Gamma))\times A^1\Gamma$
\begin{equation}
N(r\lambda)\sin\phi(r\lambda)  = \omega_r(\partial_\theta|_{\lambda}), \qquad N(r\lambda)\cos\phi(r\lambda)  = \omega_r(\partial_z|_{\lambda}).
\end{equation}
Since $\omega_0(\partial_\theta|_{\lambda})=0$, it follows that $r_o(\lambda)$ is solution to the equation 
    \begin{equation}
        \phi(r_o(\lambda)\lambda)=\pi.
    \end{equation}
Since any positive value of $\phi$ is a regular one (see the proof of \cref{thm:singlocusbound}), $r_o:A^1\Gamma\to \R$ is smooth. It follows that, for $q\in \Gamma$, the set $D_q \subset M$ of \eqref{eq:Dqdisk} is a smooth embedded disk. Recall the map $\varphi: A^{<\rho}\Gamma \to \R^3$
		\begin{equation}\label{eq:contact_embedding2}
				\varphi(x,y,z)= \left( \sqrt{\frac{\phi}{H}}x, \sqrt{\frac{\phi}{H}}y,z\right).
		\end{equation}
As we already observed in the proof of \cref{thm:singlocusbound}, $\varphi$ is a diffeomorphism on the image, and
\begin{equation}
\varphi^*\omega_{\mathrm{ot}} =  \sin\phi d\theta +\cos\phi dz= \frac{1}{N} E^*\omega.
\end{equation}
It follows that on the tubular neighbourhood $E(A^{<\rho}\Gamma)\subset M$ it holds $(\varphi \circ E^{-1})_* \xi = \xi_{\mathrm{ot}}$. Furthermore it is easy to check (recalling that $r\mapsto \phi(r\lambda)$ is monotone increasing for $r\in (0,r_{\inj}(\Gamma))$ and $\lambda \in A^1\Gamma$) that, if $q=E(0,0,z)$, then
\begin{equation}
\varphi\circ E^{-1}(D_q) = \{(x,y,z)\in \R^3\mid x^2+y^2\leq 2\pi\},
\end{equation}
which is (up to a translation along the $z$ axis) the standard overtwisted disk of $(\R^3,\xi_{\mathrm{ot}})$. It follows that $D_q$ is a overtwisted disk.	
\end{proof}

\subsection{Comparison theorems under Schwarzian derivative bounds}

We can apply \cref{thm:comparisonScwharzian} to contact Jacobi curves, to estimate their singular radius. Combined with \cref{thm:dist=conj}, we obtain a comparison theorem for the tightness radius.

\begin{theorem}\label{thm:comparisonschwarzian}
		Let $(M,\omega, g)$ be a complete three-di\-men\-sional contact sub-Riemannian manifold and let $\Gamma$ be an embedded piece of Reeb orbit with $r_{\inj}(\Gamma)>0$. Assume that there exist $k_1,k_2\in \R$ such that the Schwarzian derivative of the contact Jacobi curves are bounded above by:
		\begin{equation}\label{eq:schwarzianestimate}
			\frac{1}{2}\mathcal S(\Omega_{\lambda})(r)\leq -\frac{3}{4r^2}+k_1r+k_2r^2,\qquad \forall\, r\in (0,r_{\inj}(\Gamma)],\quad \forall\, \lambda\in A^1\Gamma.
		\end{equation}
Then for the tightness radius of $\Gamma$ it holds
\begin{equation}\label{eq:rtightestimate}
r_\tight(\Gamma) \geq \min\left\lbrace r_*(k_1,k_2),r_{\inj}(\Gamma) \right\rbrace,
\end{equation}
where
\begin{equation}\label{eq:k12}
r_*(k_1,k_2):= \begin{cases}
\frac{-k_1+\sqrt{8 \pi  k_2^{3/2}+k_1^2}}{2 k_2}, & \text{if } k_1,k_2>0, \\
\frac{\sqrt{2 \pi} }{k_2^{1/4}}, & \text{if }  k_1 \leq 0, \, k_2>0,\\
			\left(\frac{3 j_{2/3}}{2}\right)^{2/3}\frac{1}{k_1^{1/3}}, & \text{if }  k_1>0, \,k_2 \leq 0 ,\\
			+\infty & \text{if } k_1,k_2\leq 0,
			\end{cases}
\end{equation}
and $j_{2/3}\sim 3.37$ is the first positive root of the Bessel function of first kind $J_{2/3}$.
\end{theorem}
\begin{remark}\label{rmk:asymptotics}
Fix $\lambda \in A^1\Gamma$. At $r=0$ and at focal radii $\bar{r}$, the following asymptotics hold
	\begin{equation}\label{eq:asymptoticsremark}
	\frac{1}{2}\mathcal{S}(\Omega_\lambda)(r) = -\frac{3}{4r^2} +O(r),\qquad 	\frac{1}{2}\mathcal{S}(\Omega_\lambda)(r) = -\frac{k(k+2)}{4(r-\bar{r})^2} + k\, O\left(\frac{1}{r-\bar{r}}\right) + O(1),
	\end{equation}
	with smooth remainders, and $k\in \{0,1,2,3,4\}$ (see \cref{item:asymptotics0,item:asymptoticsfoc} of \cref{thm:singwelldef}). In particular the Schwarzian derivative remains bounded \emph{from above} in a neighbourhood of focal radii. Since focal radii are discrete, one can see that for any $\lambda \in A^1\Gamma$ there is a smooth function $R_{\lambda}:[0,+\infty) \to \R$  such that
	\begin{equation}
	\frac{1}{2}\mathcal{S}(\Omega_\lambda)(r) \leq -\frac{3}{4r^2} + R_{\lambda}(r), \qquad \forall\, r\in (0,+\infty).
	\end{equation}
	It follows that for any bounded interval $I$ the bound \eqref{eq:schwarzianestimate} holds on that interval, for some $k_1,k_2\in \R$, depending on $\lambda,I$. In particular the upper bound in \eqref{eq:schwarzianestimate} is always verified for a given $\lambda$ if $r_{\inj}(\Gamma)<+\infty$ (e.g.\ when $M$ is compact). The \emph{uniformity} of this bound w.r.t. $\lambda \in A^1\Gamma$ is more delicate because the remainder terms at focal radii in \eqref{eq:asymptoticsremark}, albeit smooth w.r.t.\ $r$, may not be upper semi-continuous w.r.t.\ $\lambda$.
\end{remark}
\begin{remark}[Sharpness]\label{rmk:sharp}
		The tightness radius estimate is sharp for the standard sub-Rieman\-nian contact structure (resp.\ overtwisted structure) of \cref{ex:st_sr} (resp.\ \cref{ex:st_ot_sr}). More precisely, for any embedded piece of Reeb orbit $\Gamma$ we have $r_{\inj}(\Gamma)=+\infty$, (see \cref{cor:model_radii_st,cor:model_radii_ot}) and as computed in \cref{ex:st_scwharz,ex:ot_scwharz}, for all $\lambda\in A^1\Gamma$ it holds
		\begin{equation}
				\frac{1}{2}\mathcal S(\Omega_{\lambda})(r)=-\frac{3}{4r^2}+ k_2 r^2,\qquad \forall\, r\in(0,+\infty),
		\end{equation}
		with $k_2=0$ (resp.\ $k_2=1$). For these structures, $r_{\tight}(\Gamma)=+\infty$ (resp.\ $r_{\tight}(\Gamma)=\sqrt{2\pi}$). \cref{thm:comparisonschwarzian} is also sharp for left-invariant K-contact structures, see \cref{sec:K-left}, \cref{prop:rho1-rho2-K-left}.
	\end{remark}
	\begin{remark}[On the necessity of $k_1$]\label{rmk:kappa1} The fact that $k_1=0$ in the models of \cref{rmk:sharp} is no coincidence. Indeed, for any radial model of \cref{thm:alpha_beta} the Schwarzian derivative satisfies 
        \begin{equation}\label{eq:order1=0}
 \frac{1}{2}           \mathcal S(\Omega_\lambda(r))=-\frac{3}{4 r^2}+O(r^2).
        \end{equation}
        This follows from \eqref{eq:S_order_1} and the fact that, for smoothness reasons, the homogeneous coordinate $\dot v(r)$ is always an odd function of $r$. However \eqref{eq:order1=0} does not necessarily hold in less symmetric structures, as we illustrate with an example.
        
        Let $\varepsilon>0$ and let $f_\varepsilon:\mathbb R^3\to\mathbb R^3$ be the following smooth function 
    \begin{equation}
f_\varepsilon(r\cos\theta,r\sin\theta,z)= f_{\varepsilon,\theta}(r):=\frac{r^2}{2}+\varepsilon\frac{r^5\cos\theta}{1+r^2},
    \end{equation}
    where $r^2=x^2+y^2$. For $\varepsilon>0$ small enough the form 
    \begin{equation}
    \omega=\sin(f_{\varepsilon,\theta}(r))d\theta+\cos(f_{\varepsilon,\theta}(r))dz,
    \end{equation}
    defines a smooth contact structure on $\mathbb R^3$, and the symmetric form 
    \begin{equation}
        g=\left.\left(dr\otimes dr+\left(\frac{\partial f_{\varepsilon,\theta}}{\partial r}\right)^2\left(d\theta\otimes d\theta +dz\otimes dz\right)\right)\right|_{\ker\omega},
    \end{equation}
    defines a smooth sub-Riemannian metric, and $\omega$ is normalized with respect to it. Arguing as in \cref{thm:alpha_beta} one can show that $\Gamma=\{(0,0,z)\mid z\in\mathbb R\}$ is the orbit of the Reeb field through the origin, and that, in the sense of \cref{thm:alpha_beta},  $E:A\Gamma\cong\mathbb R^3\to\mathbb R^3$ is the identity. It follows that for all $\lambda \in A^1\Gamma$
    \begin{equation}
        \Omega_\lambda(r)=[\sin(f_{\varepsilon,\theta}(r)):\cos(f_{\varepsilon,\theta}(r))],\qquad  \forall\,r\in \mathbb R.
    \end{equation}
    Therefore, setting $\upsilon_{\varepsilon,\theta}(r)=\tan{\left(f_{\varepsilon,\theta}(r)\right)}$, elementary computations show that
    \begin{equation}
        \mathcal S(\Omega_\lambda)(r)= \frac{\dddot{\upsilon}_{\varepsilon,\theta}}{\dot{\upsilon}_{\varepsilon,\theta}} - \frac{3}{2}\left(\frac{\ddot{\upsilon}_{\varepsilon,\theta}}{\dot{\upsilon}_{\varepsilon,\theta}}\right)^2 = - \frac{3}{2r^{2}} + 15\varepsilon \cos{\left(\theta\right)}r+ O\left(r^{2}\right).
    \end{equation}
    In this case, $k_1 \neq 0$ is necessary in order to apply \cref{thm:comparisonschwarzian}.
 \end{remark}
\begin{remark}
If assumption \eqref{eq:schwarzianestimate} is verified only for $r \in (0,T]$, then the estimates hold only up to distance $T$. More precisely \eqref{eq:rtightestimate} holds by adding $T$ to the arguments of the minimum.
\end{remark}
\begin{proof}
By \cref{thm:dist=conj}, it is sufficient to prove that, for any $\lambda \in A^1\Gamma$ with $r_o(\lambda)<r_{\inj}(\Gamma)$, it holds $r_o(\lambda)\geq r_*(k_1,k_2)$. Fix then such a $\lambda$.

By \cref{def:contactjacobi}, $r_o(\lambda)$ is the first self-intersection time of the contact Jacobi curve $\Omega_\lambda : [0,+\infty)\to \RP^1$, which we can restrict under our working assumption to the interval $[0,r_{\inj}(\Gamma))$, where \eqref{eq:schwarzianestimate} holds. We intend to apply \cref{thm:comparisonScwharzian} with $q=2\mathcal{S}(\Omega_\lambda)$ and a suitable model $\bar{q}$ for comparison, given by the right hand side of \eqref{eq:schwarzianestimate}. We note that the regularity assumption \eqref{eq:comparisonSchwarzian-assumption1} is verified thanks to \cref{item:asymptotics0} of \cref{thm:singwelldef}.

\noindent\textbf{1. Case $k_1,k_2>0$.} In this general case, the solutions of the model ODE
\begin{equation}
\ddot{\bar{u}}(r)+\left(-\frac{3}{4r^2}+k_1 r + k_2 r^2\right)\bar{u}(r)=0
\end{equation}
cannot be expressed in terms of elementary functions. We choose then a larger $\bar{q}$, which has the advantage of having elementary solution. Let then
\begin{equation}
\bar{q}(r):= -\frac{3}{4\left(r+\frac{k_1}{2k_2}\right)^2 } +k_2\left(r+\frac{k_1}{2k_2}\right)^2 ,\qquad r\in (0,+\infty).
\end{equation}
Observe that \eqref{eq:schwarzianestimate} and elementary estimates (with $k_1,k_2>0$) imply
\begin{equation}
\frac{1}{2}\mathcal{S}(\Omega_\lambda)(r) \leq \bar{q}(r),\qquad \forall\, r\in(0,r_{\inj}(\Gamma)).
\end{equation}
Furthermore, the solution of $\ddot{\bar{u}}+\bar{q}\bar{u}=0$, such that $\bar{u}(0)=0$ (see \cref{rmk:sol0at0}), is
\begin{equation}
\bar{u}(r) =  \frac{\sin \left(\frac{\sqrt{k_2}}{2} r(r+2 w)  \right)}{\sqrt{r+w}} ,\qquad w:=\frac{k_1}{2k_2}, \qquad r\in (0,+\infty).
\end{equation}
Its first positive root is the non-zero solution of $\sqrt{k_2} r (r+2 w) = 2\pi$. By \cref{thm:comparisonScwharzian} we obtain
\begin{equation}
r_o(\lambda)\geq \inf\{r>0\mid \bar{u}(r) =0\}= \frac{-k_1+\sqrt{8 \pi  k_2^{3/2}+k_1^2}}{2 k_2}.
\end{equation}

\noindent\textbf{2. Case $k_1\leq 0$, $k_2>0$.} By \eqref{eq:schwarzianestimate} we have
\begin{equation}
\frac{1}{2}\mathcal{S}(\Omega_\lambda)(r) \leq -\frac{3}{4r^2} + k_1 r+ k_2 r^2 \leq -\frac{3}{4r^2} + k_2 r^2,\qquad \forall\, r\in(0,r_{\inj}(\Gamma)).
\end{equation}
Hence this case can be seen as a special case of the above case, by setting $k_1=0$.

\noindent\textbf{3. Case $k_1>0$ and $k_2\leq 0$.} By \eqref{eq:schwarzianestimate} we have
\begin{equation}
\frac{1}{2}\mathcal{S}(\Omega_\lambda)(r) \leq -\frac{3}{4r^2} + k_1 r :=\bar{q}(r), \qquad \forall\, r\in(0,r_{\inj}(\Gamma)).
\end{equation}
In this case the non-trivial solutions of $\ddot{\bar{u}}+\bar{q}\bar{u}=0$ is most easily obtained by the replacement $\bar{u}(t) = r^{1/2} \zeta( 2\sqrt{k_1}/3 r^{3/2})$, where $\zeta$ is a solution of the Bessel equation:
\begin{equation}
z^2\ddot{\zeta}(z)+z \dot{\zeta}(z)+(z^2-4/9)\zeta(z)=0.
\end{equation}
Therefore we get that a solution of  $\ddot{\bar{u}}+\bar{q}\bar{u}=0$ with $\bar{u}(0)=0$ (see \cref{rmk:sol0at0}) is
\begin{equation}
\bar{u}(r) = r^{1/2} J_{2/3}\left(\frac{2\sqrt{\kappa}_1 r^{3/2}}{3} \right), \qquad r\in (0,+\infty),
\end{equation}
where $J_n$ is the Bessel function of first kind. Therefore applying \cref{thm:comparisonScwharzian} we obtain
\begin{equation}
r_{o}(\lambda) \geq \inf\{r>0\mid \bar{u}(r) =0\} = \left(\frac{3 j_{2/3}}{2}\right)^{2/3} \frac{1}{k_1^{1/3}},
\end{equation}
where $j_{2/3}\sim 3.37$ is the first positive zero of the Bessel function $J_{2/3}$.

\noindent\textbf{3. Case $k_1,k_2\leq 0$.} By \eqref{eq:schwarzianestimate} we have
\begin{equation}
\frac{1}{2}\mathcal{S}(\Omega_\lambda)(r) \leq -\frac{3}{4r^2} :=\bar{q}(r),\qquad \forall\, r\in(0,r_{\inj}(\Gamma)).
\end{equation}
A non-trivial solution of  $\ddot{\bar{u}}+\bar{q}\bar{u}=0$ with $\bar{u}(0)=0$ (see \cref{rmk:sol0at0}) is
\begin{equation}
\bar{u}(r) = t^{3/2} , \qquad r\in (0,+\infty).
\end{equation}
Therefore by \cref{thm:comparisonScwharzian} we obtain $r_{o}(\lambda) \geq \inf\{r>0\mid \bar{u}(r) =0\} =+\infty$.
\end{proof}

\subsection{Comparison theorems under canonical curvature bounds}\label{sec:canonical}

In this section we develop a different type of tightness radius estimate, assuming a control on the so-called \emph{canonical curvature}, rather than on the Schwarzian derivative of the contact Jacobi curve. The canonical curvature plays the role of sectional curvature in sub-Riemannian comparison geometry. Historically, it was introduced in \cite{AG-Feeback,AZ-JacobiI} for the study of curves in the Lagrange Grassmannians. The concept of canonical curvature (and the related canonical frame) was then formalized in full generality in \cite{ZeLi}, and was further studied and developed in \cite{ZeLi2,LLZ,AAPL,AL-Bishop,CurvVar,ABR-contact,conj,BR-BakryEmery,BR-Jacobi} in sub-Riemannian geometry and its comparison theory. We follow here the presentation that can be found in \cite{CurvVar}.

\begin{defi}\label{def:canframe}
Let $(M,\omega,g)$ be a three-di\-men\-sional contact sub-Riemannian manifold. Let $\lambda \in T^*M \setminus \{H=0\}$. The \emph{canonical frame} is a one-parameter family of moving frames for $T_{\lambda}(T^*M)$, denoted by $\{E_i(t),F_i(t)\}_{i=a,b,c}$, which is determined up to sign by the following properties:
\begin{itemize}
	\item It is a one-parameter family of Darboux frames:
	\begin{equation}
			\sigma(E_i(t),E_j(t))=\sigma(F_i(t),F_j(t))= \sigma(E_i(t),F_j(t))-\delta_{ij} =0,\qquad \forall\ t \in \R;
	\end{equation}
	 \item the ``$E$'' part of the frame satisfies:
	\begin{equation}
	\ker \pi_* = e^{t\vec{H}}_* \spn\{ E_a(t), E_b(t), E_c(t)\}, \qquad \forall\, t\in \R;
	\end{equation}
	\item It satisfies the \emph{structural equations}:
	\begin{align}
			\dot  E_a(t) & =-F_a(t),& \dot F_a(t)& =R_a^{\lambda}(t) E_a(t)-F_c(t), \label{evolution_can_a}\\
			\dot E_b(t) & =-F_b(t),& \dot F_b(t) & =0, \label{evolution_can_b}\\ 
			\dot E_c(t) & =E_a(t), &  \dot F_c(t) & =R_c^{\lambda}(t) E_c(t), \label{evolution_can_c}
		\end{align}
		where the dot denotes the derivative with respect to $t$, and $R_a^{\lambda}(t),R_c^{\lambda}(t)$ are smooth functions called \emph{canonical curvatures} at $\lambda$.
	\end{itemize}

\end{defi}
On three-dimensional contact sub-Riemannian manifolds, the canonical frame can be computed explicitly. Let us recall some notation. Let $f_0,f_1,f_2$ be a local oriented orthonormal frame and $\nu_0=\omega, \nu_1,\nu_2$ be the corresponding dual frame. Let $ h_0,h_1,h_2:T^* M\to \mathbb R$ be the linear-on-fibers functions defined in \eqref{hi2}.
Note that, for any $q \in M$ such that $f_1,f_2$ are defined, the map $(h_0,h_1,h_2):T_q^* M \to \R^3$ induces coordinates on the fiber $T^*_q M$ and, in turn, vertical\footnote{We call a vector (or a vector field) on $T^*M$ \emph{vertical} if $\pi_* V = 0$.} vector fields $ \partial_{h_0},\partial_{h_1},\partial_{h_2}$ locally defined on $T^* M$. They are characterized by the following property in terms of the symplectic form
	\begin{equation}\label{eq:propertypartialh}
		\sigma(\partial_{h_i},\cdot)=(\pi^*\nu_i)(\cdot), \qquad i=0,1,2.
	\end{equation}
Correspondingly, we define the vertical vector fields
		\begin{equation}
		\partial_\theta=h_1\partial_{h_2}-h_{2}\partial_{h_1}, \qquad \mathfrak e=h_1\partial_{h_1}+h_2\partial_{h_2}+h_0\partial_{h_0}.
	\end{equation}
	Both vector fields are globally well-defined on $T^* M$ independently on the choice of the $f_0,f_1,f_2$. Furthermore $\mathfrak{e}$ is the \emph{Euler vector field}, the generator of fiber-wise dilations, namely
	\begin{equation}
	e^{(\ln \alpha) \mathfrak{e}}(\lambda) = \alpha \lambda, \qquad \forall \alpha >0,\quad \lambda \in T^*M.
	\end{equation}
Finally, let also introduce the following globally-defined smooth vector field:
\begin{equation}\label{eq:Hprime}
\vec{H}':= [\partial_\theta,\vec{H}] = h_2\vec{h}_1-h_1\vec{h}_2-\left(\sum_{j=1}^2 c_{12}^j\partial_\theta h_j\right)\partial_\theta +\left(\sum_{i,j=1}^2 h_ic_{i0}^j\partial_\theta h_j\right)\partial_{h_0}.
\end{equation}

In terms of the above ingredients, the canonical frame has the following expression (see \cite[Sec.\ 7.5]{CurvVar}, to which we refer to for a proof and more details):
\begin{align}
E_c(t) & = \frac{1}{\sqrt{2H}}e^{-t\vec{H}}_*\partial_{h_0}, & 
F_c(t) & = \frac{1}{\sqrt{2H}}e^{-t\vec{H}}_*\left(-[\vec{H},\vec{H}'] + R_{a}^{\lambda}(t) \partial_\theta\right), \label{eq_can_frame1}\\
E_a(t) & = \frac{1}{\sqrt{2H}}e^{-t\vec{H}}_*\partial_\theta, &
F_a(t) & = \frac{1}{\sqrt{2H}}e^{-t\vec{H}}_* \vec{H}', \label{eq_can_frame2} \\ 
E_b(t) & = \frac{1}{\sqrt{2H}}e^{-t\vec{H}}_* \mathfrak{e}, &
F_b(t) & = \frac{1}{\sqrt{2H}}e^{-t\vec{H}}_* \vec{H}. \label{eq_can_frame3}
\end{align}
{We have the following expressions for the canonical curvatures:
	\begin{align}
			R_a^{\lambda}(t) & =\frac{1}{2H}\sigma_{\lambda(t)}([\vec{H},\vec{H}'],\vec{H}'),  \label{Ricci_Curv1}  \\
			R_c^{\lambda}(t) & =\frac{1}{2H}\sigma_{\lambda(t)}([\vec{H},[\vec{H},\vec{H}']],[\vec{H},\vec{H}'])-\frac{1}{(2H)^2}\sigma_{\lambda(t)}([\vec{H},\vec{H}'],\vec{H}')^2, \label{Ricci_Curv2}
	\end{align}
	where the right hand sides of \eqref{Ricci_Curv1}-\eqref{Ricci_Curv2} are evaluated at $\lambda(t)=e^{t\vec{H}}(\lambda)$.} 
	(In \cite[Prop.\ 7.13]{CurvVar}, from which the expressions \eqref{Ricci_Curv1}-\eqref{Ricci_Curv2} are taken, there is a typo in the formulas for the curvatures, where $R_{bb}$ should be instead $R_{cc}$.)
	
	\begin{remark}\label{rmk:canonic_curv}
	There is a relation between $R_a^\lambda(t), R_c^\lambda(t)$ and the invariants $\chi,\kappa$ of \cref{eq:strcoeff}. The relation is quite involved as also the derivatives of $\chi$ appear. Such relation can be found in \cite[Thm.\,6.3]{AAPL} with a different notation, and we do not report it here.
	\end{remark}
	
The next statement links the canonical frame with the contact Jacobi curve.

		\begin{lemma}\label{lem:omega=F^c}
		Let $(M,\omega,g)$ be a three-di\-men\-sional contact sub-Riemannian manifold, and $\lambda \in T^*M\setminus \{H(0)\}$. Let $\{E_i(r),F_i(r)\}_{i=a,b,c}$ be the canonical frame and let $\{E^i(r),F^i(r)\}_{i=a,b,c}$ be its dual. Then
		\begin{equation}
			F^c(r) = -\sqrt{2H}(\pi\circ e^{r\vec{H}})^*\omega, \qquad \forall \, r \in \R,
		\end{equation}
		where both sides are evaluated at the given $\lambda$.
	\end{lemma}
	\begin{proof}
		For simplicity, assume $2H(\lambda)=1$ It is sufficient to show that for $i=a,b,c$ it holds
		\begin{equation}
			\langle(\pi\circ e^{r\vec{H}})^*\omega,\,F_i(r)\rangle=-\delta_{ic},\qquad \langle(\pi\circ e^{r\vec{H}})^*\omega,\,E_i(r)\rangle=0.
		\end{equation}
		We begin with $E_i$. It holds:
		\begin{equation}
			\langle(\pi\circ e^{r\vec{H}})^*\omega,E_i(r)\rangle=\langle\pi^*\omega,\,e^{r\vec{H}*}E_i(r)\rangle  =0.
		\end{equation}
		Note that $\vec{H}$ pointwise projects to an horizontal vector, namely
		\begin{equation}\label{eq:pistarHfreccia}
		\pi_*  \vec{H}|_{\lambda}  = h_1(\lambda) f_1|_{\pi(\lambda)}+h_2(\lambda) f_2|_{\pi(\lambda)}.
		\end{equation}
		Therefore, using \eqref{eq_can_frame3} and \eqref{eq:pistarHfreccia} it holds
		\begin{equation}
			\langle(\pi\circ e^{r\vec{H}})^*\omega,F_b(r)\rangle=  \langle\pi^*\omega, \vec{H}\rangle=0.
		\end{equation}		
		Similarly, for $\vec{H}'$ it follows from equation \eqref{eq:Hprime} that 
		\begin{equation}\label{eq:pistarHprimefreccia}
			\pi_* \vec{H}'|_\lambda =h_2(\lambda)f_{1}|_{\pi(\lambda)}-h_1(\lambda)f_{2}|_{\pi(\lambda)}.
		\end{equation}
	Therefore
		\begin{equation}
				\langle(\pi\circ e^{r\vec{H}})^*\omega,F_a(r)\rangle  =
																\langle\pi^*\omega, \vec{H}'\rangle=0.
		\end{equation}
		Finally, using \eqref{eq_can_frame1} we obtain
		\begin{align}
				\langle(\pi\circ e^{r\vec{H}})^*\omega,F_c(r)\rangle& = \langle\pi^*\omega, [\vec{H}', \vec{H}]\rangle
				\\
				&= \vec{H}'(\pi^*\omega(\vec{H}))-\vec{H}(\pi^*\omega(\vec{H}'))- d\omega(\pi_*\vec{H}',\pi_*\vec{H}) \\
				&= \nu_1\wedge\nu_2( h_1 f_1+h_2 f_2,h_2 f_1-h_1 f_2)=-1,
		\end{align}
		where we used $\pi^*\omega(\vec{H})=\pi^*\omega(\vec H')=0$,  $d\omega=\nu_1\wedge\nu_2$ and \eqref{eq:pistarHfreccia}--\eqref{eq:pistarHprimefreccia}.
	\end{proof}

We now prove our comparison theorem for tightness radius under canonical curvature bounds. As it will be clear from the techniques in the proof, the result is non-sharp  (see \cref{rmk:not_sharp}). However, it provides a connection between tightness and sub-Riemannian curvature invariants.

\begin{theorem}\label{thm:comparisoncurvature}
		Let $(M,\omega, g)$ be a complete three-di\-men\-sional contact sub-Riemannian manifold and let $\Gamma$ be an embedded piece of Reeb orbit with $r_{\inj}(\Gamma)>0$. Assume that there exist $A,C>0$ such that 
\begin{equation}\label{eq:bound_AC}
 \sqrt{1+R_a^\lambda(r)^2}\leq A,\quad \sqrt{1+R_c^\lambda(r)^2}\leq C,\qquad \forall\, r\in[0,r_{\inj}(\Gamma)],\quad \forall\,\lambda\in A^1\Gamma,
\end{equation}
where $R_a^\lambda(r),R_c^\lambda(r)$ are the canonical curvatures at $\lambda$.  Then for the tightness radius of $\Gamma$ it holds
	\begin{equation}
		r_{\tight}(\Gamma)\geq\min\left\{\tau(A,C),r_{\inj}(\Gamma)\right\},
	\end{equation}
	where 
	\begin{equation}\label{eq:tauAC}
		\tau(A,C):=\int_{0}^\infty\frac{1}{A u^2+C u+1}\,\mathrm{d}u.
	\end{equation}
\end{theorem}
\begin{proof}
By \cref{thm:dist=conj}, we have to prove that, for all $\lambda \in A^1\Gamma$ with $r_o(\lambda)<r_{\inj}(\Gamma)$, it holds
	\begin{equation}\label{eq:r_o_ineq}
		r_o(\lambda) >  \int_{0}^\infty\frac{1}{A u^2+C u+1}\,\mathrm{d} u.
	\end{equation}
Fix then such a $\lambda$. 	Recall that, by \cref{def:contactjacobi}, the first singular radius is
	\begin{align}
		r_o(\lambda)& =\inf\{r>0 \mid \Omega_{\lambda}(r)=\Omega_{\lambda}(0)\}\\
		& =\inf\left\{r>0 \mid \ker\omega_{r}|_\lambda=\ker\omega_{0}|_\lambda\right\},
	\end{align}
 where $\omega_r=(\pi\circ e^{r\vec{H}})^{*}\omega|_{A^1\Gamma}$ is the one-form \eqref{eq:moving_omega}.
	Observe that
	\begin{equation}
			\ker\omega_{0}|_\lambda=\ker\pi^*\omega\cap T_\lambda( A^1\Gamma)= \spn\{\partial_\theta\} = \{E_a(0)\},
	\end{equation}
	where we used \eqref{eq_can_frame2}.	Using \cref{lem:omega=F^c} and  we can write $r_o(\lambda)$ as
	\begin{equation}
		r_o(\lambda)=\inf\{r>0\mid\langle F^c(r), E_a(0)\rangle=0\}.
	\end{equation}
	From \cref{def:canframe} we deduce the structural equation for the dual of the canonical frame:
	\begin{align}
				\dot F^a& =E^a, & \dot E^a & =-R_a F^a-E^c,\\
			\dot F^c & =F^a, & \dot E^c & =-R_c F^c,
		\end{align}
	where we omit the dependence on $r$ and $\lambda$. It follows that $x_0(r):=\langle F^c(r), E_a(0)\rangle$ satisfies the following Cauchy problem:
\begin{equation}\label{eq:cauchy1}
			\begin{cases}
				\frac{d}{dr}\left(\dddot x_0 +R_a\dot x_0 \right)-R_c x_0=0, \\
				x_0(0)=0,\,\dot x_0(0)=0,\,\ddot x_0(0)=1,\,\dddot x_0(0)=0.
			\end{cases}
	\end{equation}		
We rewrite it as a first-order Cauchy problem by defining $x: = (x_0,x_1,x_2,x_3)$ with
	\begin{equation}
		x_1:=x^{(1)},\quad x_2:=x^{(2)},\quad x_3:=x^{(3)}+R_ax^{(1)},
	\end{equation}
	so that \eqref{eq:cauchy1} rewrites as 
	\begin{equation}\label{eq:cauchy2}
		\begin{cases}
			\dot x_0=x_1,\\
			\dot x_1=x_2,\\
			\dot x_2=x_3-R_ax_1,\\
			\dot x_3=R_c x_0,\\
			 x(0)=(0,0,1,0).
		\end{cases}
	\end{equation}
	Therefore we have $r_o(\lambda)=\inf\{r>0\mid x_0(r)=0\}$. Notice that, since 
	\begin{equation}
		x_0(r)=\int_{0}^r\left(\int_{0}^t x_2(s)\, \mathrm{d}s\right)\,\mathrm{d}t,\qquad x_2(0)=1,
	\end{equation}
	then $r_o(\lambda)>\tau:=\inf\{r>0 \mid x_2(r)=0\}.$ We define then
	\begin{equation}
		v:=\begin{pmatrix} x_0/x_2 \\ x_1/x_2 \\ x_3/x_2 \end{pmatrix}, \qquad M = (0,R_a,-1)\cdot v\, \mathbb{1} +\begin{pmatrix}
				0 & 1 & 0\\
				0 & 0 & 0 \\
				R_c & 0 & 0
			\end{pmatrix}, \qquad w := \begin{pmatrix}
				0 \\
				1\\  
				0
			\end{pmatrix},
	\end{equation}
	where $\mathbb{1}$ is the identity matrix and $\cdot$ denotes the scalar product. We find that $v$ is the solution of the following Cauchy problem 
	\begin{equation}
\begin{cases}
			\dot v = M v+ w,\\
			v(0)=0.
\end{cases}
	\end{equation}
	Thus, letting $u:=|v|$ we have the inequality 
		\begin{equation}
	\dot{u}(r)\leq \sqrt{1+R_a^2(t)}u(r)^2+\sqrt{1+R_c^2(t)}u(r)+1\leq Au(t)^2+Cu(r)+1, \qquad \forall \, r \in [0,\tau).
	\end{equation}
	We deduce that $r_o(\lambda)$ must be greater than the first blow-up time of the Cauchy problem
	\begin{equation}
		\begin{cases}
			\dot u=Au^2+Cu+1,\\
			u(0)=0,
		\end{cases}
	\end{equation}
	which is given by \eqref{eq:tauAC}. It follows that $r_o(\lambda)>\tau\geq \tau(A,C)$.
\end{proof}
\begin{remark}\label{rmk:not_sharp}
	In the proof of \cref{thm:comparisoncurvature} it is shown that, for any $\lambda\in A^1\Gamma$ it holds $r_o(\lambda)> \tau(A,C)$, where $x_0(\cdot)$ is the solution to the ODE \eqref{eq:cauchy1}. According to \cref{thm:dist=conj} we have
	\begin{equation}
		r_{\tight}(\Gamma)\geq \min\left\{\inf_{\lambda\in A^1\Gamma}  r_o(\lambda),r_{\inj}(\Gamma)\right\}.
	\end{equation}
Therefore, if $r_{\tight}(\Gamma)\leq r_{\inj}(\Gamma)$ then the estimate in  \cref{thm:comparisoncurvature} is not sharp.
\end{remark}

\section{K-contact sub-Riemannian manifolds}\label{sec:Kcontact}

A three-di\-men\-sional contact sub-Riemannian manifold $(M,\omega, g)$ is called K-contact if the associated Reeb field $f_0$ is a Killing vector field for the Riemannian extension $g$, or, equivalently, if the Reeb flow acts on $M$ by sub-Riemannian isometries.

Remember the two metric invariants $\chi,\kappa: M\to \R$ introduced in \cref{sec:contactSR}. The first one can be given in terms of an orthonormal frame $f_0,f_1,f_2$ and its structural coefficients by:
\begin{equation}\label{eq:chi(1)}
    \chi=\sqrt{\left(c_{01}^1\right)^2+\frac{1}{4}\left(c_{01}^2+c_{02}^1\right)^2}.
\end{equation}
The next result is well-known, see \cite[Cor.\ 17.10]{Agrachev}.
\begin{prop}
  A three-dimensional contact sub-Riemannian manifold $(M,\omega,g)$ is K-contact if and only $\chi=0$.
\end{prop}

The second metric invariant is given by the formula:
\begin{equation}\label{eq:kappa(1)}
    \kappa =f_1\left(c_{12}^2\right)-f_2\left(c_{12}^1\right)-\left(c_{12}^1\right)^2-\left(c_{12}^2\right)^2 +
				\frac{c_{02}^1-c_{01}^2}{2}.
    \end{equation}
We recall that, in K-contact case, $\kappa$ is the Gaussian curvature of the surface obtained by locally quotienting $M$ under the action of the Reeb flow, see \cref{rmk:chikappa}.

\begin{remark}\label{rmk:equal_tubmaps}
    Let $(M,\omega,g)$ be a contact sub-Riemannian manifold (not necessarily K-contact). Let $H_R:T^*M\to \mathbb R$ be the Hamiltonian of the Riemannian extension, i.e., the Riemannian metric obtained from the sub-Riemannian one, promoting the Reeb field to a unit normal to $\ker\omega$. Then the Riemannian tubular neighbourhood map reads 
    \begin{equation}
        E_R:A\Gamma\to M,\qquad E_R(\lambda)=\pi\circ e^{\vec H_R}(\lambda).
    \end{equation}
    In general $E_R\neq E$.
    However in the K-contact case the two maps coincide. Indeed 
    \begin{equation}
        H_R = H + \frac{1}{2}h_0^2 \qquad \Rightarrow \qquad \vec H_R=\vec H +h_0\vec h_0,
    \end{equation}
    where $H:T^*M\to \mathbb R$ is the sub-Riemannian Hamiltonian.
Since $[\vec H,\vec h_0]=0$ and $h_0|_{A\Gamma}=0$, then $\vec H_R|_{A\Gamma} = \vec H|_{A\Gamma}$ and both are tangent to $A\Gamma$. It follows that 
\begin{equation}
    E_R(\lambda)=\pi\circ e^{\vec H_R}(\lambda)=\pi\circ e^{\vec H}(\lambda)=E(\lambda),\qquad \forall\,\lambda\in A\Gamma.
\end{equation}
\end{remark}

 \begin{lemma}\label{lem:E*star_Kcont}
		Let $(M,\omega,g)$ be a complete three-dimensional K-contact sub-Riemannian manifold. Let $\Gamma$ be an embedded piece of a Reeb orbit diffeomorphic to an interval, then there exists a set of cylindrical coordinates $(r,\theta,z)$ on $A\Gamma$ (see \cref{sec:cylindcoords}) such that 
		\begin{equation}
			E^*\omega=dz+\omega_r(\partial_\theta)d\theta,
		\end{equation}
	where $\omega_r$ is defined in \cref{def:contactjacobi}.
\end{lemma}
\begin{proof}
Let $q \in \Gamma$ and let $f_1(q),f_2(q)\in T_q M$ be an orthonormal basis for $\xi_q$. Since $\Gamma$ is diffeomorphic to an interval, there exist a point $q\in \Gamma$ and an interval $I$ such that the map $\gamma:I\to M$ defined by $\gamma(z)=e^{zf_0}(q)$ is an embedding. Since $M$ is K-contact, the frame 
\begin{equation}
	f_{i}(\gamma(z))=e^{zf_0}_{*}f_i(q), \qquad i=1,2,
\end{equation}
is an orthonormal frame for the sub-Riemannian metric, along $\Gamma=\gamma(I)$.
Consequently the dual frame $\nu_1,\nu_2$ satisfies 
\begin{equation}\label{eq:inv_dual_frame}
	e^{z\vec{h}_0}(\nu_{i}|_q)=\left(e^{-zf_0}\right)^*(\nu_{i}|_q)=\nu_{i}|_{\gamma(z)}.
\end{equation}
Let $(r\cos\theta,r\sin\theta,z)$ be cylindrical coordinates on $A\Gamma$ as in \cref{sec:cylindcoords} built with the frame $f_0,f_1,f_2$, so that any $\lambda \in A^1\Gamma$ can be written as
\begin{equation}
	\lambda = (\cos\theta  \nu_1 + \sin\theta \nu_2)|_{\gamma(z)} = :\nu_{\theta,z}.
\end{equation}
Since $M$ is K-contact, then $\{H,h_0\}=0$, or equivalently $[\vec{H},\vec{h}_0]=0$. It follows that
\begin{equation}
	\pi\circ e^{r\vec{H}}(\nu_{\theta,z}) = \pi\circ e^{r\vec{H}} \circ e^{z \vec{h}_0}(\nu_{\theta,0}) =\pi \circ e^{z \vec{h}_0} \circ e^{r\vec{H}}(\nu_{\theta,0}),
\end{equation}
so that for $\partial_z \in T A^1\Gamma$ it holds $\omega_r(\partial_z) =\omega(\pi_*\circ e^{r\vec{H}}_*\partial_z) = \omega(f_0) = 1$ (cf. \cref{def:contactjacobi}). The result then follows from equality \eqref{eq:Estaromega} of \cref{prop:Estaromega}.
\end{proof}

For K-contact structures, the next results shows that there cannot be overtwisted disks within the injectivity radius from a Reeb orbit.

\begin{theorem}\label{thm_conj_time}
Let $(M,\omega, g)$ be a complete three-di\-men\-sional K-contact sub-Riemannian manifold and let $\Gamma$ be an embedded piece of Reeb orbit with $r_{\inj}(\Gamma)>0$. Then 
 \begin{equation}\label{eq:tight=inj}
     r_{\tight}(\Gamma)=r_{\inj}(\Gamma).
 \end{equation}
\end{theorem}
\begin{proof}
	Let $\lambda \in A^1\Gamma$. Combining \cref{prop:Estaromega} and \cref{lem:E*star_Kcont} we deduce that the contact Jacobi curve has the following expression 
\begin{equation}
	\Omega_\lambda:[0,+\infty)\to \mathbb{RP}^1,\qquad \Omega_\lambda(r)=[\omega_r(\partial_\theta|_{\lambda})\,:\,1].
\end{equation}
We see that the contact Jacobi curve is not surjective. By \cref{thm:singwelldef}, the contact Jacobi curve is an immersion when restricted to the interval $(0,r_{\foc}(\lambda))$. A non-surjective immersion of an interval in $\mathbb {RP}^1$ is necessarily injective, therefore $r_o(\lambda)> r_{\foc}(\lambda)$. Since $r_{\foc}(\lambda)\geq r_{\inj}(\Gamma)$, using \cref{thm:dist=conj} we obtain \eqref{eq:tight=inj}.
\end{proof}

We need the following comparison theorem for focal radii from the Reeb orbit.
\begin{lemma}\label{lem:focalcomparison}
Let $(M,\omega, g)$ be a complete three-di\-men\-sional K-contact sub-Riemannian manifold and let $\Gamma$ be an embedded piece of Reeb orbit with $r_{\inj}(\Gamma)>0$. Assume that the sub-Riemannian metric invariant $\kappa \leq \kappa_+$ for some $\kappa_+ \in \R$. Then for all $\lambda \in A^1\Gamma$ it holds 
	\begin{equation}\label{eq:rfocestimate}
		r_{\foc}(\lambda)\geq \begin{cases}
		\frac{\pi}{\sqrt{\kappa_+}} & \kappa_+ >0, \\
		+\infty & \kappa_+ \leq 0.
		\end{cases}
	\end{equation}
Similarly, with reversed inequality, if $\kappa \geq \kappa_-$ for some $\kappa_-\in \R$.

	In particular, if $\kappa\leq 0$, then the tubular neighbourhood map $E:A\Gamma\to M$ is an immersion.
\end{lemma}
\begin{proof}
By \cref{rmk:equal_tubmaps}, the sub-Riemannian Hamiltonian flow from $A\Gamma$ coincides with the one of the Riemannian extension. In other words, $E:A\Gamma\to M$ coincides with the Riemannian tubular neighbourhood map from $\Gamma$, by identifying the normal bundle $N\Gamma$ with $A\Gamma$ through the Riemannian metric. With these identifications, $E_* \partial_z = f_0$. Consider then the geodesic $t\mapsto E(t\lambda)$. It follows that $r$ is a focal radius for $\lambda \in A^1\Gamma$ if and only if $r$ is a conjugate radius for the Riemannian exponential map for the Riemannian manifold obtained by taking the quotient by the Reeb flow of a small neighbourhood of the geodesic $t\mapsto E(t\lambda)$. Such a quotient has Gauss curvature equal to the sub-Riemannian invariant $\kappa$ (see \cref{rmk:chikappa}). By classical comparison theory for conjugate points (see e.g.\ \cite[Prop.\ 2.4]{DoCarmo}) it follows that if $\kappa_-\leq \kappa\leq \kappa_+$ we have
	\begin{equation}
	\frac{\pi}{\sqrt{\kappa_+}} \leq r_{\foc}(\lambda) \leq \frac{\pi}{\sqrt{\kappa_-}},
	\end{equation}
	with the convention that the bounds are $+\infty$ if $\kappa_{\pm} \leq 0$.
\end{proof}

Before proving the main theorem of this section, we state the following generalization of classical Riemannian Hopf-Rinow theorem.
\begin{theorem}[Hopf-Rinow from a submanifold]\label{thm:Hopf-Rinow-generalized}
	Let $M$ be a Riemannian manifold, with distance function $d_R$. Let $S \subset M$ be an embedded submanifold (without boundary) and assume that $(S,d_R|_{S})$ is a complete metric space. Then the following are equivalent:
	\begin{itemize}
		\item $(M,d_R)$ is complete;
		\item The Riemannian tubular neighbourhood map $E_R: AS \to M$ is well-defined.
	\end{itemize}
\end{theorem}
The classical Hopf-Rinow corresponds to the case $S=\{pt\}$. The proof of \cref{thm:Hopf-Rinow-generalized} is analogous to the standard one \cite[Thm.\ 2.8, Sec.\ 7.2]{DoCarmo}, replacing the normal neighbourhood of a point with a normal tubular neighbourhood around $S$.

We can now prove the main result of this section: a Hadamard-type theorem for K-structures.
\begin{theorem}\label{thm:contactHadamard}
	Let $(M,\omega,g)$ be a three-di\-men\-sional complete and simply connected K-contact sub-Riemannian manifold. If $\kappa\leq 0$ then $(M,\ker\omega)$ is contactomorphic to the standard contact structure on $\R^3$ (see \cref{ex:standard}).
\end{theorem}
\begin{proof}
Our first claim is that the Reeb field $f_0$ is complete. Let $q\in M$ and $\gamma:I\to M$, $\gamma(z)=e^{zf_0}(q)$, be the maximal integral curve of $f_0$ going through $q$. If every Reeb orbit is periodic then there is nothing to prove. Assume by contradiction that $I\subsetneq \mathbb R$ is an open interval and that $\gamma:I\to M$ is not periodic. Without loss of generality assume $I=(0,2)$. Then $z\in I$ can be written as $z=k + \delta$, with $k \in\{ 0,1\}$ and $\delta \in (0,1)$. By the triangle inequality and the fact that the Reeb flow is an isometry we have for all $z\in I$:
\begin{align}
d(q,\gamma(z)) & \leq d(q,e^{k f_0}(q)) + d(e^{k  f_0}(q),e^{(k +\delta) f_0}(q))\\
& =  d(q, e^{f_0}(q)) + d(q,e^{\delta f_0}(q)) \\
& \leq 2\max_{|\tau|\leq 1} d(q,e^{\tau f_0}(q)).
\end{align}
Since $d$ is complete, the trajectory $I \ni z\mapsto e^{z f_0}(q)$ is contained in a pre-compact set and thus $\gamma$ can be extended on an interval larger than $I$, proving the completeness of $f_0$.

Fix then a maximal piece of Reeb orbit $\Gamma$ (either $\Gamma\simeq \mathbb{S}^1$ or $\Gamma \simeq \R$). Since $\kappa\leq0$, according to \cref{lem:focalcomparison}, $E:A\Gamma\to M$ is an immersion, thus the triple
\begin{equation}\label{eq_pullback}
(A\Gamma, E^*\omega, E^* g) 
\end{equation}
is a K-contact sub-Riemannian manifold.  Let $\tilde g$ denote the Riemannian extension of the sub-Riemannian metric $g$, i.e., the one obtained by promoting $f_0$ to a unit normal to $\ker\omega$. We claim that $(A\Gamma,E^*\tilde g)$, is a complete Riemannian manifold. Let $d_R:A\Gamma\times A\Gamma\to\mathbb R$ be the Riemannian distance function of $(A\Gamma,E^*\tilde g)$. Observe that the zero section $\Gamma_0 \subset A\Gamma$ is an embedded submanifold. Moreover $(\Gamma_0, d_R|_{\Gamma_0})$ has transitive isometry group, namely the restriction of the flow of the Reeb field of \eqref{eq_pullback}. Thus $(\Gamma_0, d_R|_{\Gamma_0})$ is a locally compact metric space with transitive isometry group and hence it is complete. 

Since the curves $r\mapsto E(r\lambda)$ for $\lambda \in A^1\Gamma$ are sub-Riemannian geodesics from $\Gamma \subset M$, then radial lines $r\mapsto r\lambda$ on $A\Gamma$ are sub-Riemannian geodesics from $\Gamma_0 \subset A\Gamma$. Then according to \cref{rmk:equal_tubmaps}, they are also Riemannian geodesics normal to $\Gamma_0$. It follows that the Riemannian tubular neighbourhood map of \eqref{eq_pullback} from $\Gamma_0$ is well-defined. \cref{thm:Hopf-Rinow-generalized} then implies that $(A\Gamma,d_R)$ is a complete Riemannian manifold. Moreover denoting with $d^{A\Gamma}:A\Gamma\times A\Gamma\to \mathbb R$ the sub-Riemannian distance of \eqref{eq_pullback}, we have $d_R\leq d^{A\Gamma}$. Therefore $(A\Gamma, d^{A\Gamma})$ is a complete metric space as well. Consequently, according to \cref{thm:hopf-rinow}, the sub-Riemannian exponential map $\exp^{A\Gamma}:T^*(A\Gamma)\to A\Gamma$ is well-defined.

Being a local isometry, $E:A\Gamma\to M$ sends geodesics to geodesics. Thus for a fixed $\lambda\in A\Gamma$ we have the equality 
	\begin{equation}
		E\circ \exp^{A\Gamma}_\lambda\circ E^*=\exp_{E(\lambda)}.
	\end{equation}
Since $(M,\omega,g)$ is a complete sub-Riemannian manifold and there are no non-trivial abnormal length-minimizers the exponential map is surjective \cite[Prop.\ 8.38]{Agrachev}). We deduce that $E:A\Gamma\to M$ is surjective as well. 
 We have proved that $(A\Gamma,E^*\tilde g)$ is a complete Riemannian manifold and that the map 
\begin{equation}
	E:(A\Gamma,E^*\tilde g)\to (M,\tilde g)
\end{equation}
is a surjective local isometry of Riemannian manifolds. We recall the following Riemannian result from \cite[Lemma\, 3.3, Chap.\,7]{DoCarmo}.
\begin{lemma}\label{lem:covering-R}
		Let $M$ and $N$ be Riemannian manifolds. Let $f:M\to N$ be a surjective local diffeomorphism satisfying $\|f_* v\|\geq \|v\|$, $\forall\,v\in TM$. If $M$ is complete, then $f$ is a covering map.
\end{lemma}
Applying \cref{lem:covering-R} we deduce that $E:A\Gamma\to M$ is a covering map.
Since $M$ is simply connected, $E$ is a actually a diffeomorphism. 
		Thus $r_{\inj}(\Gamma)=+\infty$ and according to \cref{thm_conj_time} also $r_{\tight}(\Gamma)=+\infty$.  Notice that, since $A\Gamma\simeq \Gamma\times \mathbb R^2$ and $E:A\Gamma\to M$ is a diffeomorphism, $\Gamma$ cannot be periodic, otherwise $M$ would not be simply connected. Therefore we can apply \cref{lem:E*star_Kcont}, obtaining 
\begin{equation}\label{eq:omega-K-contact}
E^*\omega =  dz + \omega_r(\partial_\theta)d\theta.
\end{equation}
We deduce that the map $ \Phi:\mathbb R^3\to  M$, defined by
 \begin{equation}
 \Phi(z,x,y):=E\left(\frac{\sqrt{2\omega_r(\partial_\theta)}}{r}x,\frac{\sqrt{2\omega_r(\partial_\theta)}}{r}y,z\right),
 \end{equation}
 is a diffeomorphism, satisfying $\Phi^*\omega=\omega_{\mathrm{st}}$.
\end{proof}
\begin{corollary}\label{cor_reeb_orb}
	Let $(M,\omega,g)$ be a complete three-di\-men\-sional K-contact sub-Riemannian manifold. If $\kappa\leq 0$ then any periodic orbit of the Reeb field is the generator of an infinite cyclic subgroup of the fundamental group $\pi_1(M)$.
\end{corollary}
\begin{proof}
	Let $\gamma:\mathbb S^1\to M$ be a periodic orbit of the Reeb field. Then, according to the proof of \cref{thm:contactHadamard}, the map $E :A\Gamma\simeq \mathbb S^1\times \mathbb R^2\to M$ is a covering which maps the generator of $\pi_1(A\Gamma)$, which is $[z\mapsto (z,0,0)]\in\pi_1(A\Gamma)$, to $[\gamma]\in \pi_1(M)$.
\end{proof}

\subsection{Examples with prescribed non-positive curvature}\label{sec:nonposi}

Next we would like to show that there are indeed many examples of sub-Riemannian manifolds with vanishing $\chi$ and non positive curvature $\kappa$, to which \cref{thm:contactHadamard} can be applied, see \cref{cor:examples-univ-tight}.

\begin{theorem}\label{thm_examples}
	Let $(B,\eta)$ be an orientable Riemannian surface of curvature $\kappa_\eta$. Assume that the area form $\Omega$ of $g$ defines an integral cohomology class (which can always be achieved by constant rescaling of the metric). Then there exists a principal circle bundle $\pi:M\to B$, with connection $\omega \in \Omega^1(M)$, $d\omega=\pi^* \Omega$. Moreover the triple $(M,\omega, g=\pi^*\eta+\omega\otimes\omega)$ defines a sub-Riemannian structure having invariants $\chi=0, \kappa=\pi^*\kappa_\eta$.
\end{theorem}
\begin{proof}
	We make use of the following result from \cite{Kobayashi}, as stated in \cite[Thm.\ 3]{Boothby}.
	\begin{theorem}\label{thm_koba}
		Let $(B,\Omega)$ be a $2$-dimensional symplectic manifold and assume that $[\Omega]\in H^2(B,\mathbb Z)$. Then there is a principal circle bundle $\pi: M \to B$ over $B$ with connection $\omega$ satisfying $d\omega=\pi^*\Omega$. In particular $\omega$ determines a contact structure and its Reeb field $f_0$ generates the right translations of the bundle by $\mathbb S^1$.
	\end{theorem}
	Let $\pi:(M,\omega)\to B$ be the principal $\mathbb S^1$-bundle with contact connection $\omega$, obtained applying \cref{thm_koba} to the symplectic manifold $(B,\Omega)$. Observe that the quadratic form $g=\pi^*\eta+\omega\otimes\omega$ is positive definite over $M$, therefore the triple $(M,\omega, g)$ defines a sub-Riemannian structure. By construction $f_0$ generates a one-parameter group of isometries so that $\chi =0$. The corresponding orbit space is isometric to $(B,\eta)$, thus $\kappa =\pi^*\kappa_\eta$ by \cref{rmk:chikappa}.
\end{proof}

 It follows from  \cite{N-overtwisted,NP-resolution} that K-contact structures are tight, see \cite[Rmk.\ 1.3]{QDarboux}. By applying \cref{thm:contactHadamard} to the universal cover of $(M,\omega,g)$ equipped with the pull-back sub-Riemannian structure, we obtain an alternative proof for this class of structures.

\begin{corollary}\label{cor:examples-univ-tight}
The contact structures in \cref{thm_examples}, with $\kappa_\eta\leq 0$, are universally tight.
\end{corollary}

\subsection{Examples with positive curvature}\label{sec:posi}

In the proof of \cref{thm:contactHadamard}, the assumption $\kappa\leq 0$ can be replaced by asking that $E: A\Gamma \to M$ has no critical points (this fact is implied by $\kappa \leq 0$ through \cref{thm_conj_time}). Arguing as in the proof of the latter, it is sufficient that there exists a Reeb orbit $e^{z f_0}(q)$ such that the Riemannian manifold obtained by taking the quotient by the action of the Reeb flow (locally along any sub-Riemannian geodesic $E :A\Gamma \to M$) has no conjugate points along (Riemannian) geodesics starting from the corresponding point in the quotient. The curvature of the quotient, in this case, can also be positive.

We give an example. Let $(N,h)$ be a complete, orientable, non-compact Riemannian surface. There exists $\alpha \in \Lambda^1 N$ such that $d\alpha = \mathrm{vol}_h$ is the volume form of $N$. Let then $M=N \times \R$, with projection $p(q,z) = q$. Define
\begin{equation}
\omega := dz + p^*\alpha, \qquad g = p^*h + \omega\otimes \omega.
\end{equation}
One can check that  $\omega\wedge d\omega = \mathrm{vol}_g$ so that $(M,\omega,g)$ is a contact sub-Riemannian manifold. The Reeb vector field is $f_0 = \partial_z$, which is clearly a generator of isometries for $(M,\omega,g)$. Thus, the sub-Riemannian invariants are
\begin{equation}
\chi = 0,\qquad \kappa =p^*\kappa_h,
\end{equation}
where $\kappa_h$ is the Riemann curvature of $(N,\eta)$. Assume now that $(N,\eta)$ has a \emph{pole}, namely, there exists a point $q$ such that any geodesic emanating from $q$ has no conjugate points along it. Let $\Gamma$ be the Reeb orbit passing through from $(q,0)\in M$. Then the construction in the proof of \cref{thm_conj_time} shows that the map $E: A\Gamma \to M$ is an immersion and thus also the proof of \cref{thm:contactHadamard} proceeds unchanged. As an example, one can choose as $(N,h)$ the Riemannian surface given by $N= \{(x,y,w)\in \R^3 \mid w = x^2+y^2\}$ with the induced metric from $\R^3$. The origin is a pole in the above sense, and the curvature $\kappa_h$ is positive.

\subsection{Left-invariant K-contact structures}\label{sec:K-left}

A K-contact sub-Riemannian structure $(M,\omega,g)$ is called \emph{left-invariant} if $M$ is a Lie group, and both the contact form $\omega$ and the sub-Riemannian metric $g$ are left-invariant.

\begin{example}{(Left-invariant sub-Riemannian structure on the Heisenberg group.)}\label{ex:left_R3} The sub-Riemannian structure $(\R^3,\omega_{\mathrm{st}},g_{\mathrm{st}})$ defined in \cref{ex:st_sr} is left-invariant with respect to the Heisenberg group multiplication, defined by
\begin{equation}
    (x, y, z) \star (x', y', z') :=
\left( x + x', y + y', z + z' + \frac{1}{2}(xy' - x'y) \right).
\end{equation}
\end{example}
\begin{example}{(Left-invariant sub-Riemannian structure on $\mathrm{SU}(2)$.)}\label{ex:left_SU(2)}
The Lie algebra of \( \text{SU}(2) \) is the algebra of anti-hermitian traceless \( 2 \times 2 \) complex matrices:
\begin{equation}
\mathfrak{su}(2): = \left\{
\begin{pmatrix}
i\alpha & \beta \\
-\bar\beta & -i\bar\alpha
\end{pmatrix}
\in \text{Mat}(2, \mathbb{C}) \ \middle| \ \alpha \in \mathbb{R},\,\, \beta \in \mathbb{C}
\right\}.
\end{equation}
A basis of \( \mathfrak{su}(2) \) is \( \{f_1, f_2, f_0\} \), where
\begin{equation}
f_1 = \frac{1}{2}
\begin{pmatrix}
0 & i \\
i & 0
\end{pmatrix}, \quad
f_2 = \frac{1}{2}
\begin{pmatrix}
0 & 1 \\
-1 & 0
\end{pmatrix}, \quad
f_0 = \frac{1}{2}
\begin{pmatrix}
i & 0 \\
0 & -i
\end{pmatrix}.
\end{equation}
Setting $\xi=\spn\{f_1,f_2\}$, and declaring $f_1,f_2$ an oriented orthonormal frame, we obtain a left-invariant sub-Riemannian structure on $\mathrm{SU}(2)$ with Reeb field $f_0$. See \cite[Sec.\ 13.6]{Agrachev}.
\end{example}
\begin{example}{(Left-invariant sub-Riemannian structure on $\widetilde{\mathrm{SL}}(2)$.)}\label{ex:left_SL(2)}
The Lie algebra of $\widetilde{\mathrm{SL}}(2)$, i.e. the universal cover of $\mathrm{SL}(2)$, is the algebra of $2\times 2$ traceless real matrices:
\begin{equation}
 \mathfrak{sl}(2) := \{A \in \mathrm{Mat}(2,\mathbb{R}) \mid \mathrm{trace}(A) = 0 \}.
\end{equation}
 A basis for $\mathfrak{sl}(2)$ is $\{f_1,f_2,f_0\}$, where
\begin{equation}
f_1 = \frac{1}{2}
\begin{pmatrix}
1 & 0 \\
0 & -1
\end{pmatrix}, \quad
f_2 = \frac{1}{2}
\begin{pmatrix}
0 & 1 \\
1 & 0
\end{pmatrix}, \quad
f_0 = \frac{1}{2}
\begin{pmatrix}
0 & -1 \\
1 & 0
\end{pmatrix}.
\end{equation}
Setting $\xi=\spn\{f_1,f_2\}$, and declaring $f_1,f_2$ an oriented orthonormal frame, we obtain a left-invariant sub-Riemannian structure on $\widetilde{\mathrm{SL}}(2)$ with Reeb field $f_0$. See \cite[Sec.\ 17.5.4]{Agrachev}.
\end{example}
The three examples above actually compose an exhaustive list of the simply connected $K$-contact left-invariant sub-Riemannian structures.
\begin{theorem}{\cite[Corollary 17.31]{Agrachev}}\label{thm:K-left}
    Let $(M,\omega, g)$ be a three-di\-men\-sional simply connected, K-contact, left-invariant sub-Riemannian manifold. The metric invariant $\kappa:M\to \mathbb R$ defined in \eqref{eq:kappa(1)} is constant. Up to a constant rescaling of the contact form and the sub-Riemannian metric, we have the following three cases:
    \begin{enumerate}[label = $(\roman*)$]
        \item $\kappa=0$ and $(M,\omega, g)$ is isometric to the left-invariant sub-Riemannian structure on the Heisenberg group, i.e., \cref{ex:left_R3};
        \item $\kappa=1$ and $(M,\omega, g)$ is isometric to the left-invariant sub-Riemannian structure on $\mathrm{SU}(2)$, i.e, \cref{ex:left_SU(2)};
        \item $\kappa=-1$ and $(M,\omega, g)$ is isometric to the left-invariant sub-Riemannian structure on $\widetilde{\mathrm{SL}}(2)$, i.e, \cref{ex:left_SL(2)}.
    \end{enumerate}
\end{theorem}

We compute the tightness radius estimate of \cref{thm:comparisonschwarzian-intro,thm:comparisoncurvature-intro} for \cref{ex:left_R3,,ex:left_SU(2),,ex:left_SL(2)}. For a Reeb orbit $\Gamma$, we denote with $\rho_{\mathrm{Thm.B}}(\Gamma)$ and $\rho_{\mathrm{Thm.C}}(\Gamma)$ the estimates produced by \cref{thm:comparisonschwarzian-intro} and \cref{thm:comparisoncurvature-intro}, respectively:
\begin{align}
\rho_{\mathrm{Thm.B}}(\Gamma)& :=\min\{r_\inj(\Gamma), r_*(k_1,k_2)\},\\
    \rho_{\mathrm{Thm.C}}(\Gamma) & :=\min\{r_\inj(\Gamma), \tau(A,C)\},
\end{align}
where $r_*(k_1,k_2)$ is defined in \eqref{eq:k12-intro} and $\tau(A,C)$ is defined in \eqref{eq:tauAC-intro}.
\begin{prop}\label{prop:rho1-rho2-K-left}
    Let $(M,\omega, g)$ be one of the K-contact left-invariant structures of \cref{ex:left_R3,,ex:left_SU(2),,ex:left_SL(2)}, and let $\Gamma$ be a Reeb orbit. We have the following cases:
    \begin{enumerate}[label = $(\roman*)$]
        \item $(M,\omega,g)$ is the left-invariant structure on the Heisenberg group, and
        \begin{equation}
            \rho_{\mathrm{Thm.B}}(\Gamma)=r_\tight(\Gamma)=+\infty,\qquad \rho_{\mathrm{Thm.C}}(\Gamma)=\frac{2\pi}{3\sqrt{3}}\approx 1.02;
        \end{equation}
        \item $(M,\omega,g)$ is the left-invariant structure on $\mathrm{SU}(2)$, and 
        \begin{equation}
            \rho_{\mathrm{Thm.B}}(\Gamma)=r_\tight(\Gamma)=\pi,\qquad \rho_{\mathrm{Thm.C}}(\Gamma)\approx 1.05,
        \end{equation}
        \item $(M,\omega,g)$ is the left-invariant structure on $\mathrm{SL}(2)$, and 
        \begin{equation}
            \rho_{\mathrm{Thm.B}}(\Gamma)=r_\tight(\Gamma)=+\infty,\qquad \rho_{\mathrm{Thm.C}}(\Gamma)\approx 1.05.
        \end{equation}

    \end{enumerate}
\end{prop}
\begin{proof}
    We begin with the computation of $\rho_{\mathrm{Thm.B}}(\Gamma)$.
    We need the following lemma.
    \begin{lemma}\label{lem:K_cont_schw}
    Let $(M,\omega, g)$ be a three-di\-men\-sional K-contact, left-invariant sub-Riemannian structure with metric invariant $\kappa\in\mathbb R$. Let $\Gamma$ be an embedded piece of Reeb orbit. Then, for any $\lambda\in A^1\Gamma$ and $r\in(0,+\infty)$, the Schwarzian derivative of the contact Jacobi curve satisfies
        \begin{equation}
\frac{1}{2}    \mathcal S(\Omega_\lambda)(r)
            =\begin{cases} \frac{\kappa}{4} \left(1 - \frac{3}{\sin^{2}{\left(\sqrt{ \kappa} r\right)}}\right) & \kappa >0 ,\\
-\frac{3}{4r^2} & \kappa =0,\\
\frac{\kappa}{4} \left(1 + \frac{3}{\sinh^{2}{\left(\sqrt{ -\kappa} r\right)}}\right) & \kappa <0 .
\end{cases}
        \end{equation}
\end{lemma}
\begin{proof} 
Recall that the contact Jacobi curve at $\lambda\in A^1\Gamma$ is the projectivization
\begin{equation}
\Omega_\lambda(r)=P\left(\omega_r|_\lambda\right),\qquad \forall\, r \in [0,+\infty).
\end{equation}
where $\omega_r$ is defined in \eqref{eq:moving_omega}, see \cref{def:contactjacobi}. We explicitly compute $\omega_r|_\lambda$ which, we recall, is a one-parameter family of one-forms in the fixed space $T_\lambda^*(A^1\Gamma)$. By  \cite[Prop.\ 17.17]{Agrachev}, there exists an orthonormal frame $\{f_0,f_1,f_2\}$ satisfying 
\begin{equation}\label{eq:k-frame}
    \begin{aligned}
        [f_2,f_1]=f_0,\qquad
        [f_1,f_0]=\kappa f_2\qquad
        [f_2,f_0]=-\kappa f_1.
    \end{aligned}
\end{equation}
Let $h_1,h_2,h_0:T^*M\to \mathbb R$ be the associated Hamiltonian functions, i.e. $h_i(\lambda)=\langle \lambda, f_i\rangle$. Observe that, denoting $\lambda(r)=e^{r\vec{H}}(\lambda)$, the functions $h_i(r):=h_i(\lambda(r))$, satisfy the following equations
\begin{align}
        \dot h_0(r)& =\{H,h_0\}=0,\\
        \dot h_1(r)& =\{H,h_1\}=h_2(r)h_0(r),\\
        \dot h_2(r)& =\{H,h_2\}=-h_1(r)h_0(r).
\end{align}
Since $\lambda \in A^1\Gamma$ we have initial conditions $h_0(0)=0$ and $h_1(0)^2+h_2(0)^2=1$. Then
\begin{equation}
    h_0(r)=0,\qquad h_1(r)=h_1(0),\qquad h_2(r)=h_2(0),\qquad \forall\, r\in \mathbb R.
\end{equation}
In the following we omit the evaluation at $\lambda$ and denote 
\begin{equation}
    \dot\omega_r=\frac{d}{dr}\omega_r, \qquad \ddot\omega_r=\frac{d^2}{dr^2}\omega_r.
\end{equation}
We exploit the general computations in \cref{app:contidots}. Substituting the structural coefficients of the frame \eqref{eq:k-frame} in \eqref{eq:omegadott}, and exploiting the fact that $h_0,h_1,h_2$ are constant along $\lambda(r)$, we obtain 
\begin{equation}\label{eq:k-omega-evol}
\begin{cases}
    \ddot{\omega}_r = -\kappa\omega_r + d\theta,\\
    \omega_0=dz,\quad \dot\omega_0=0.
\end{cases}
\end{equation}
Note that here $(r\cos\theta,r\sin\theta,z)$ denote the cylindrical coordinates on $A\Gamma$ of \cref{sec:cylindcoords}, and the initial conditions in \eqref{eq:k-omega-evol} are obtained from \cref{lem:computationsdots}.
The solution to \eqref{eq:k-omega-evol} is
\begin{equation}
        \omega_r= \cos(\sqrt{\kappa} r) dz-\frac{2}{\kappa} \sin^2\left(\frac{\sqrt{\kappa}r}{2}\right) d\theta.
    \end{equation}
    Therefore the contact Jacobi curve can be computed explicitly:
    \begin{equation}
        \Omega_\lambda(r)=\left[\cos(\sqrt{\kappa} r): -\frac{2}{\kappa} \sin^2\left(\frac{\sqrt{\kappa}r}{2}\right)\right], \qquad \forall\,\lambda\in A^1\Gamma,\qquad \forall \,r\in [0,+\infty).
    \end{equation}
    Taking the Schwarzian derivative yields the result.
\end{proof}
\cref{lem:K_cont_schw} shows that for all $\kappa\in \mathbb R$ we have $ \tfrac{1}{2}\mathcal{S}(\Omega_\lambda)(r)\leq -\frac{3}{4r^2}$. It follows that we can choose $k_1=k_2=0$ in \cref{thm:comparisonschwarzian-intro}, obtaining $r_*(k_1,k_2)=+\infty$. Therefore, for all $\kappa\in\R$, we have 
\begin{equation}
    \rho_{\mathrm{Thm.B}}(\Gamma)=\min\{r_\inj(\Gamma),r_*(k_1,k_2) \}=r_\inj(\Gamma).
\end{equation}
Since $r_\inj(\Gamma)\geq r_\tight(\Gamma)\geq \rho_{\mathrm{Thm.B}}(\Gamma)=r_\inj(\Gamma)$, we also obtain $r_\tight(\Gamma)=r_\inj(\Gamma)$, which we had already proved for K-contact manifolds in \cref{thm_conj_time}. 

It remains to compute $r_{\inj}(\Gamma)$ in the three cases. We make use of the following fact.
\begin{lemma}\label{lem:K-sect}
    Let $(M,\omega,g)$ be a three-di\-men\-sional K-contact, left-invariant sub-Riemannian structure. Then, for each $n\in T^1M$, the sectional curvature (with respect to the Riemannian extension) of the plane $n^{\perp}$ is given by 
    \begin{equation}\label{eq:sec_k}
        \mathrm{sec}(n^{\perp})=\omega(n)^2\left(\kappa-1\right)+ \frac{1}{4},\qquad \mathrm{Ric}(f_0)=\frac{1}{2}.
    \end{equation}
\end{lemma}
\begin{proof}
Substituting the structural coefficients of the frame \eqref{eq:k-frame} into the formulae of \cite[Thm.\ 4.3]{Milnor}, we obtain the result.
\end{proof}
Substituting $\kappa=1$ in \eqref{eq:sec_k} we deduce that $\mathrm{SU}(2)$ with the Riemannian extension has constant sectional curvature, equal to $1/4$. Therefore, $\mathrm{SU}(2)$ with the Riemannian extension is isometric to the round $3-$sphere of radius $2$ (or Riemannian diameter $2\pi$). One can check that the orbits of the Reeb field are Riemannian geodesics. Thus $\Gamma$ is a great circle. According to \cref{rmk:equal_tubmaps} the Riemannian and sub-Riemannian tubular neighbourhood maps coincide in the $K$-contact case. It follows that $r_\inj(\Gamma)=\pi$ if $\kappa=1$. If $\kappa\leq 0$, then according to \cref{thm:contactHadamard}, $r_\inj(\Gamma)=+\infty$. Summing up, we obtained
\begin{equation}\label{eq:rinjKcont}
r_{\inj}(\Gamma) = \begin{cases} \pi & \text{if } \kappa =1,\\
+\infty & \text{if } \kappa \leq 0.
\end{cases}
\end{equation}
This concludes the computation of $\rho_{\mathrm{Thm.B}}(\Gamma)$.

Next, we compute $\rho_{\mathrm{Thm.C}}(\Gamma)$.
The canonical curvatures of K-contact structures have been computed for example in \cite[Thm.\ 6.7]{AAPL}, from which, recalling that $\lambda \in A^1\Gamma$ and hence $h_0(\lambda)=0$ and $2H(\lambda)=1$, we deduce
\begin{equation}
R_a^\lambda(r) = \kappa,\qquad R_c^\lambda(r) = 0,\qquad \forall\,\lambda \in A^1\Gamma,\quad \forall\, r>0.
\end{equation}
Therefore in \cref{thm:comparisoncurvature-intro} we can take $A=\sqrt{2}$, $C=1$ if $\kappa\in\{-1,1\}$, while $A=C=1$, if $\kappa =0$. Replacing in \eqref{eq:tauAC-intro} we obtain the following values for $\rho_{\mathrm{Thm.C}}(\Gamma)$:
\begin{equation}
\rho_{\mathrm{Thm.C}} =\min\{r_\inj(\Gamma), \tau(A,C)\} \approx \begin{cases}
  1.05 & \text{if $\kappa\in\{-1,1\}$},\\
 1.02 & \text{if $\kappa=0$},
\end{cases}
\end{equation}
where we have used also \eqref{eq:rinjKcont}.
\end{proof}

\section{Comparison with the state of the art}\label{sec:compare}

Let $\xi$ denote a contact distribution on a 3-manifold $M$. Following \cite[Def.\,2.7]{QDarboux}, a Riemannian metric $g$ is \emph{compatible} with $\xi$ if there exist a contact form $\omega$, an almost-complex structure $J:\xi\to \xi$, and a positive constant $\theta'$, such that 
    \begin{equation}\label{eq:Riem-comp}
		\frac{1}{\theta'}d\omega(X, J Y)=g(X, Y), \quad \forall \, X,Y \in \xi_p,\quad \forall\, p \in M,
	\end{equation}
 and the Reeb field is the oriented unit normal to $\xi$. The constant $\theta'$ is called the \emph{rotation speed}.

Given a co-orientable contact manifold $(M,\xi)$, there is a one-to-one correspondence between Riemannian metrics which are compatible with $\xi$ and normalized sub-Riemannian contact structures. Indeed, a choice of sub-Riemannian metric on $\xi$ induces a normalized contact form $\omega$ as described in \cref{sec:contactSR}. Promoting the Reeb field to a unit normal to $\xi$ we obtain a Riemannian metric, which we call the \emph{Riemannian extension} of the sub-Riemannian metric. Equation \eqref{eq:complex_strctr} shows that the Riemannian extension is compatible with $\xi$ and has rotation speed $\theta'=1$. Conversely, given a Riemannian metric which is compatible with $\xi$, we can restrict it to the contact distribution to obtain a sub-Riemannian metric which, in turn, induces a unique contact form $\omega$ that is normalized, i.e.\ satisfies \eqref{eq:compatibility}. 

\begin{remark}[Notation] 
By the above discussion, we work below with two different metric structures: the Riemannian one or the corresponding sub-Riemannian one. The symbol $g$ denotes either the Riemannian or sub-Riemannian metric, depending on the context, no confusion will arise. Of course the corresponding metric structures are different. Thus, all quantities with superscript or subscript $R$, such as $d_R$, $B^R_r(\Gamma)$, $r_{\tight}^R$, $r_{\inj}^R(\Gamma)$, \dots, are defined w.r.t.\ the Riemannian structure, while the corresponding quantities without affix or suffix denote the corresponding sub-Riemannian ones, in agreement with the rest of the paper. An exception is the injectivity radius of a Riemannian manifold $\inj(g)$, because it has no sub-Riemannian counterpart.
\end{remark}

\subsection{Estimates of the maximal tight Riemannian tube around a Reeb orbit}

Let $(M,\omega,g)$ be a contact sub-Riemannian manifold and $\Gamma$ be an embedded Reeb orbit. Let $E_R:A\Gamma\to M$ denote the Riemannian tubular neighbourhood map and let $r_{\inj}^R(\Gamma)$ denote the corresponding Riemannian injectivity radius. Let $B_r^R(\Gamma)$ be the \emph{Riemannian tube} of radius $r$ around $\Gamma$
\begin{equation}
    B_r^R(\Gamma):=\{q\in M\mid d_R(q,\Gamma)<r\},
\end{equation}
where $d_R(\cdot,\Gamma)$ denotes the Riemannian distance from $\Gamma$. The \emph{Riemannian tightness radius from $\Gamma$} is defined as 
\begin{equation}\label{eq:r_tight-tube-riemannian}
          r_\tight^R(\Gamma):=\sup\left\{0<r<r^R_\inj(\Gamma)\mid \left(B_r^R(\Gamma),\xi|_{B_r^R(\Gamma)}\right)\,\,\text{is tight}\right\}.
      \end{equation}

\begin{lemma}\label{lem:same-dist-tube}
Let $(M,\omega,g)$ be a three-dimensional complete contact sub-Riemannian manifold, and let $\Gamma$ be an embedded Reeb orbit. Assume that the sub-Riemannian distance from $\Gamma$ coincides with the one determined by the Riemannian extension, then
\begin{equation}
r_{\tight}(\Gamma)=r_{\tight}^R(\Gamma).
\end{equation}
\end{lemma}
\begin{proof}
    By definition, the sub-Riemannian metric $g$ is the restriction of the Riemannian extension to the contact distribution. In particular the length of horizontal curves is the same, whether is measured with the Riemannian or sub-Riemannian metric. This implies that the Riemannian geodesics realizing the distance from $\Gamma$ are horizontal and coincide with the sub-Riemannian ones, and the tubular neighbourhood maps coincide. As a consequence $r_\inj^{R}(\Gamma)=r_\inj(\Gamma)$ and $B_r^R(\Gamma) = B_r(\Gamma)$ for all $r\geq 0$. Since $\Gamma$ is an embedded Reeb orbit, we can make use of definition \eqref{eq:sr-tight-tube} of the sub-Riemannian tightness radius, therefore
    \begin{align}
        r_\tight(\Gamma)&=\sup\left\{0<r<r_\inj(\Gamma)\mid \left(B_r(\Gamma),\omega|_{B_r(\Gamma)}\right)\,\,\text{is tight}\right\}\\
        &=\sup\left\{0<r<r_\inj^{R}(\Gamma)\mid \left(B_r^{R}(\Gamma),\omega|_{B_r^{R}(\Gamma)}\right)\,\,\text{is tight}\right\}=r_\tight^{R}(\Gamma).\qedhere
    \end{align}
\end{proof}
In this section we compare \cref{thm:comparisonschwarzian-intro,thm:comparisoncurvature-intro} with \cref{thm:Riem_tube-intro} (corresponding to \cite[Thm.\,1.8]{QDarboux}). We start by proving \cref{thm:compare-K-cont-intro}, of which we recall the statement. The constants $\rho_{\mathrm{Thm.B}}(\Gamma)$, $\rho_{\mathrm{Thm.C}}(\Gamma)$ and $\rho_{\mathrm{EKM}}(\Gamma)$ are defined in \eqref{eq:SR1_rho_estimates-intro}, \eqref{eq:SR2_rho_estimates-intro} and \eqref{eq:R_rho_estimates-intro}, respectively.

\begin{prop}\label{thm:compare-K-cont}
    Let $(M,\omega,g)$ be a three-dimensional simply connected K-contact left-invariant sub-Riemannian structure, and let $\Gamma$ be an embedded Reeb orbit. Then, the Riemannian and sub-Riemannian distances from $\Gamma$ coincide. Up to a constant rescaling of the contact form and the sub-Riemannian metric, we have the following three cases:
    \begin{enumerate}[label = $(\roman*)$]
        \item $(M,\omega, g)$ is the left-invariant sub-Riemannian structure on the Heisenberg group (see \cref{ex:left_R3}). Moreover
        \begin{equation}
            \rho_{\mathrm{Thm.B}}(\Gamma)=r_{\tight}(\Gamma)=+\infty, \qquad \rho_{\mathrm{Thm.C}}(\Gamma)=\frac{2\pi}{3\sqrt{3}}\approx 1.02,\qquad  \rho_{\mathrm{EKM}}(\Gamma)\leq 1,
        \end{equation}
    \item $(M,\omega, g)$ is the left-invariant sub-Riemannian structure on $\mathrm{SU}(2)$ (see \cref{ex:left_SU(2)}). Moreover
        \begin{equation}
            \rho_{\mathrm{Thm.B}}(\Gamma)=r_{\tight}(\Gamma)=\pi, \qquad \rho_{\mathrm{Thm.C}}(\Gamma)\approx 1.05,\qquad  \rho_{\mathrm{EKM}}(\Gamma) \leq 1.38,
        \end{equation}    
    \item $(M,\omega, g)$ is the left-invariant sub-Riemannian structure on $\widetilde{\mathrm{SL}}(2)$ (see \cref{ex:left_SL(2)}). Moreover 
        \begin{equation}
            \rho_{\mathrm{Thm.B}}(\Gamma)=r_{\tight}(\Gamma)=+\infty, \qquad \rho_{\mathrm{Thm.C}}(\Gamma)\approx 1.05,\qquad \rho_{\mathrm{EKM}}(\Gamma) \leq 0.74.
        \end{equation}  
    \end{enumerate}
\end{prop}
\begin{proof}
    By \cref{rmk:equal_tubmaps} for any embedded Reeb orbit $\Gamma$, in the K-contact case, the sub-Rieman\-nian and Riemannian distances from $\Gamma$ coincide. By \cref{thm:K-left}, up to a constant rescaling, $(M,\omega,g)$ is one of the model structures of items ($i$)-($ii$)-($iii$). For each case, $\rho_{\mathrm{Thm.B}}(\Gamma)$ and $\rho_{\mathrm{Thm.C}}(\Gamma)$ are computed in \cref{prop:rho1-rho2-K-left}. We now compute $\rho_{\mathrm{EKM}}(\Gamma)$. 
    Furthermore, \cref{lem:K-sect} yields the sectional curvature of the corresponding models, from which we obtain the constants $a, b, K$ involved in \eqref{eq:R_rho_estimates-intro}. Thus we have:
    \begin{enumerate}[label = $(\roman*)$]
    \item $\kappa=0$, then
        $a=1$, $b=\frac{1}{2}$, $K=\frac{1}{4}$, $\theta'=1$. Therefore, according to \eqref{eq:R_rho_estimates-intro} we have
            \begin{equation}
        \rho_{\mathrm{EKM}}(\Gamma)=\min\left\{r_{\inj}(\Gamma), \frac{\inj(g) }{2}, \pi, 1\right\}\leq 1,
    \end{equation}
    \item $\kappa=+1$, then
        $a=\frac{1}{3}$, $b=\frac{1}{2}$, $K=\frac{1}{4}$, $\theta'=1$. Therefore, according to \eqref{eq:R_rho_estimates-intro} we have 
        \begin{equation}
        \rho_{\mathrm{EKM}}(\Gamma)=\min\left\{r_{\inj}(\Gamma), \frac{\inj(g) }{2}, \pi, \frac{4}{1+\sqrt{11/3}} \right\} \leq  1.38,
    \end{equation}
\item $\kappa=-1$, then
        $a=\frac{7}{3}$, $b=\frac{1}{2}$, $K=\frac{1}{4}$, $\theta'=1$. Therefore, according to \eqref{eq:R_rho_estimates-intro} we have
    \begin{equation}
        \rho_{\mathrm{EKM}}(\Gamma)=\min\left\{r_{\inj}(\Gamma), \frac{\inj(g) }{2}, \pi, \frac{2}{\sqrt{\frac{59}{12}}+\frac{1}{2}}\right\} \leq 0.74,
    \end{equation}
\end{enumerate}
    concluding the proof.
\end{proof}
 \begin{remark}\label{rmk:unbounded_curv}
	If one among $a,b$ and $K$ in the statement of \cref{thm:Riem_tube-intro} is infinite, then the latter yields the trivial estimate $r^R_{\tight}(\Gamma)\geq 0$. However, if these quantities are bounded on the tube $B_s^R(\Gamma)$ of radius $s>0$ for each $s>0$ (e.g. if $\Gamma$ is compact), one can still apply \cref{thm:Riem_tube-intro} on $B_s^R(\Gamma)$, obtaining that there are no overtwisted disks in $B_s^R(\Gamma)$ provided that $s$ is smaller than the right hand side of \eqref{eq:Riem_estimate-intro}, computed on $B_s^R(\Gamma)$. In other words:
	\begin{equation}\label{eq:rhoEKMnew}
		r_{\tight}(\Gamma)\geq \tilde{\rho}_{\mathrm{EKM}}(\Gamma):=\sup_{s>0}\min\{s,\rho^s_{\mathrm{EKM}}(\Gamma)\},
	\end{equation}
	where $\rho_{\mathrm{EKM}}^s(\Gamma)$ is the right hand side of \eqref{eq:Riem_estimate-intro} with all relevant bounds computed on $B_s(\Gamma)$, namely
	\begin{equation}\label{eq:Riem_estimate-s}
		\rho_{\mathrm{EKM}}^s(\Gamma):=\min\left\{r_{\inj}^R(\Gamma),\frac{\inj(g)}{2}, \frac{\pi}{2\sqrt{K_s}}, \frac{2}{\sqrt{2a_s+b_s^2}+b_s}\right\},
	\end{equation}
	where $K_s$ is an upper bound on the absolute value of the sectional curvature on $B_s^R(\Gamma)$, while
	\begin{equation}
		a_s:=\frac{4}{3}\left|\sec\left(g|_{B_{s}^R(\Gamma)}\right)\right|, \qquad b_s:=\frac{1}{2}+\sqrt{\frac{1}{4}-\frac{1}{2}\min_{q\in B_s^R(\Gamma)}\mathrm{Ric}(f_0)}.
	\end{equation}
	In contrast with \eqref{eq:R_rho_estimates-intro}, the number provided by \eqref{eq:rhoEKMnew} is positive. 
\end{remark}
We conclude by proving \cref{thm:compare_ot-intro}, of which we recall the statement.
\begin{prop}\label{thm:compare_ot}
    Let $(\mathbb R^3,\omega_{\mathrm{ot}}, g_\mathrm{ot})$ be the standard sub-Rieman\-nian overtwisted structure (see \cref{ex:st_ot_sr}). Let $\Gamma$ be the Reeb orbit
    \[
    \Gamma=\{(0,0,z)\mid z\in\mathbb R\}.
    \]
Then the Riemannian and sub-Riemannian distances from $\Gamma$ coincide. Moreover
    \begin{equation}
        \rho_{\mathrm{Thm.B}}(\Gamma)=r_\tight(\Gamma)=\sqrt{2\pi},
        \qquad  \tilde{\rho}_{\mathrm{EKM}}(\Gamma) \leq \sqrt[3]{2}.
    \end{equation}
\end{prop}
\begin{proof}
    According to \cref{ex:st_ot_sr} the  Riemannian extension of $g_\mathrm{ot}$ is determined by the orthonormal frame 
\begin{equation}
    f_1=\partial_r,\qquad f_2= \frac{1}{r} \cos\left(\frac{r^2}{2}\right)\partial_\theta-\frac{1}{r} \sin\left(\frac{r^2}{2}\right)\partial_z,\qquad f_0=\cos\left(\frac{r^2}{2}\right)\partial_z+\sin\left(\frac{r^2}{2}\right)\partial_\theta.
\end{equation}
Thus, for such Riemannian structure, the integral curves of $f_1$ are arc-length parametrized geodesics realizing the distance from $\Gamma$. It then follows from \cref{thm:alpha_beta} that the Riemannian and sub-Riemannian distances from $\Gamma$ coincide. The fact that $\rho_{\mathrm{Thm.B}}(\Gamma)=\sqrt{2\pi}=r_{\tight}(\Gamma)$ was observed in \cref{rmk:sharp}.

A computation shows that the Ricci curvature of $g_{\mathrm{ot}}$ is unbounded below on $\R^3$, thus  we adopt the definition of $\tilde{\rho}_{\mathrm{EKM}}(\Gamma)$ for structures of unbounded curvature, namely \eqref{eq:rhoEKMnew}. A standard computation yields that, at the point $q=(r\cos\theta,r\sin\theta,z)$, it holds
\begin{equation}
    \mathrm{Ric}(f_0)=\frac{1}{2}(1-r^4), \qquad \sec(\mathrm{span}\{f_2,f_0\})=\frac{1}{4}\left(r^2-1\right)^2, 
\end{equation}
and consequently 
\begin{equation}
a_s\geq \frac{1}{3}\left(s^2-1\right)^2 \qquad b_s=\frac{1}{2}(s^2+1).
\end{equation}
Then, according to \eqref{eq:Riem_estimate-s}, we have 
\begin{align}
    \rho^s_{\mathrm{EKM}}(\Gamma)&\leq \frac{2}{\sqrt{2a_s+b_s^2}+b_s}\leq 2\left(\sqrt{\frac{2}{3}\left(s^2-1\right)^2+\frac{1}{4}(s^2+1)^2}+\frac{1}{2}(s^2+1)\right)^{-1}\\
    &\leq 4\left(\sqrt{\left(s^2-1\right)^2+(s^2+1)^2}+(s^2+1)\right)^{-1}\leq \frac{2}{s^2}.
\end{align}
Finally, we can estimate
\begin{equation}
\rho_{\mathrm{EKM}}(\Gamma)=   \sup_{s>0}\min\{s,\rho_{\mathrm{EKM}}^s(\Gamma)\}\leq\sup_{s>0}\min\left\{s,\frac{2}{s^2}\right\}=\sqrt[3]{2},
\end{equation}
concluding the proof.
\end{proof}

\appendix
\crefalias{section}{appendix}

\section{Lie derivatives of contact Jacobi curves}\label{app:contidots}

We prove here some technical facts needed in the proof of \cref{thm:singwelldef}.

\begin{lemma}\label{lem:computationsdots}
Let $\omega_r$ be as in \cref{def:contactjacobi}, and let $(r\cos\theta,r\sin\theta,z)$ be cylindrical coordinates on $A\Gamma$ of \cref{sec:cylindcoords}. Then the following hold:
\begin{align}
\omega_0(\partial_\theta) & = 0, & \ddot{\omega}_0(\partial_\theta) &= 1, & \omega_0(\partial_z) & = 1,\\
\dot{\omega}_0(\partial_\theta) &= 0,  & \dddot{\omega}_0(\partial_\theta) &= 0,
& \dot{\omega}_0(\partial_z) & = 0.
\end{align}
\end{lemma}
\begin{proof}
Remember that $\omega_r = \pi \circ e^{r\vec{H}}|_{A^1\Gamma}$. We denote with the same symbol the un-restricted form on $T^*M$. To prove the statement we compute Lie derivatives of $\pi^*\omega$. Note that
\begin{equation}
\mathscr{L}_{\vec{H}} (\pi^*\omega) = \left(i_{\vec{H}}\circ d + d \circ i_{\vec{H}}\right)(\pi^*\omega) = (i_{\vec{H}}\circ d) (\pi^*\omega).
\end{equation}

Let $f_0,f_1,f_2$ be a local orthonormal frame, with dual frame $\nu_0 = \omega,\nu_1,\nu_2$. Let also $h_0,h_1,h_2 : T^*M \to \R$ the linear-on-fiber functions corresponding to the vector fields $f_0,f_1,f_2$. We adopt the following convention: free latin indices are understood to range from $1$ to $2$, while greek ones from $0$ to $2$. Repeated indices are summed over their range. For example, $[f_\alpha,f_\beta]= c_{\alpha\beta}^\gamma f_\gamma$, where $c_{\alpha\beta}^\gamma \in C^\infty(M)$ are the structure coefficients  (see \cref{sec:contactSR}).

Recall that since $2H = h_i h_i$ then $\vec{H} = h_i \vec{h}_i$. Furthermore $\vec{H}(h_\mu) = \{H,h_\mu\} = h_\ell c_{\ell\mu}^\alpha h_\alpha$. We start with a few preliminary and straightforward computations:
\begin{align}
\mathscr{L}_{\vec{H}} \omega = \mathscr{L}_{\vec{H}} \nu_{0} & = h_\ell c_{\ell j }^0 \nu_j , \label{eq:omegadot}\\
\mathscr{L}_{\vec{H}} \nu_{j} & = dh_j+h_i c_{i \alpha}^j \nu_\alpha  ,  \\
\mathscr{L}_{\vec{H}} (dh_j) & = h_\ell c_{j \ell}^\alpha dh_\alpha + h_\alpha c_{j\ell}^\alpha dh_\ell + h_\ell h_\alpha d c_{j\ell}^\alpha ,
\end{align}
where the pull-back of forms is omitted from the notation, for brevity. More precisely, we adopted the shorthand $\pi^*\nu_\alpha = \nu_\alpha$ and similarly $d f = d \pi^* f$ for $f \in C^\infty(M)$, so that the above formulas are understood as identities between forms on $T^*M$. For the second derivative we have
\begin{align}
\mathscr{L}^2_{\vec{H}} \omega & = \mathscr{L}_{\vec{H}} \left(h_i c_{i j }^0 \nu_j\right) \\
& = \vec{H}(h_i c_{i j }^0) \nu_j + h_i c_{i j }^0  (\mathscr{L}_{\vec{H}}\nu_j) \\
& = \vec{H}(h_i c_{i j }^0) \nu_j + h_i c_{i j }^0 \left(dh_j + h_\ell c_{\ell\alpha}^j \nu_\alpha \right) \\
& = \vec{H}(h_i c_{i j }^0) \nu_j +  h_i c_{i j }^0  dh_j+ h_i h_\ell c_{i j }^0 c_{\ell\alpha}^j \nu_\alpha . \label{eq:omegadott}
\end{align}
For the third Lie derivative we take one Lie derivative of \eqref{eq:omegadott} and apply Leibniz rule:
\begin{align}
\mathscr{L}^3_{\vec{H}} \omega & = \vec{H}(h_i c_{i j }^0) (\mathscr{L}_{\vec{H}}\nu_j) + \vec{H}(h_i c_{i j }^0)  dh_j+ h_i c_{i j }^0 (\mathscr{L}_{\vec{H}} dh_j) + h_i h_\ell c_{i j }^0 c_{\ell\alpha}^j (\mathscr{L}_{\vec{H}}\nu_\alpha) + X \\
 & = \vec{H}(h_i c_{i j }^0) dh_j + \vec{H}(h_i c_{i j }^0)  dh_j+ h_i c_{i j }^0 \left(h_\ell c_{j \ell}^\alpha dh_\alpha + h_\alpha c_{j\ell}^\alpha dh_\ell\right)
  + h_i h_\ell c_{i j }^0 c_{\ell k}^j dh_k +X \\
  & = 2\vec{H}(h_i c_{i j }^0) dh_j + h_i c_{i j }^0 \left(h_\ell c_{j \ell}^\alpha dh_\alpha + h_\alpha c_{j\ell}^\alpha dh_\ell
  +  h_\ell  c_{\ell k}^j dh_k \right)+X \\
    & = 2\vec{H}(h_i c_{i j }^0) dh_j - 2H dh_0+ h_i c_{i j }^0 \left(h_\ell c_{j \ell}^k dh_k + h_\alpha c_{j\ell}^\alpha dh_\ell
  +  h_\ell  c_{\ell k}^j dh_k \right)+X. \label{eq:omegadottt}
\end{align}
Here, $X$ denotes a linear combination of $\nu_\alpha = \pi^*\nu_\alpha$ and $df = \pi^* d f$ for $f \in C^\infty(M)$, since these terms will give no contribution for the relevant computations.

Now consider the vector fields $\partial_z,\partial_\theta$ on $A\Gamma$, in terms of the cylindrical coordinates of \cref{sec:cylindcoords}. In the notation of this appendix, one can see that
\begin{equation}\label{eq:projections}
\pi_*\partial_\theta =0,\qquad \pi_*\partial_z = f_0.
\end{equation}
We also note that, in the notation of this appendix
\begin{equation}\label{eq:partialtheta}
\partial_\theta = h_1 \partial_{h_2} - h_2 \partial_{h_1} = h_i c_{ij}^0 \partial_{h_j}.
\end{equation}
Recall that $\omega_r = \pi \circ e^{r\vec{H}}|_{A^1\Gamma}$. Using \eqref{eq:projections} one obtains
\begin{equation}
\omega_0(\partial_\theta) =0, \qquad \omega_0(\partial_z) = 1.
\end{equation}
Furthermore, using \eqref{eq:omegadot} and \eqref{eq:projections} we obtain
\begin{equation}
\dot{\omega}_0(\partial_\theta) = 0, \qquad \dot{\omega}_0(\partial_z)= 0.
\end{equation}
Finally, using \eqref{eq:omegadott}, \eqref{eq:projections} and \eqref{eq:partialtheta} one obtains
\begin{equation}
\ddot{\omega}_0(\partial_\theta)  = h_i c_{ij}^0 dh_j\left( h_\ell c_{\ell k}^0 \partial_{h_k}\right) =  h_i h_\ell c_{ij}^0  c_{\ell j}^0 = 2H,
\end{equation}
where we used the fact that $c_{ij}^0 c_{\ell j}^0 = \delta_{i\ell}$. In particular $\ddot{\omega}_0(\partial_\theta)=1$ when compute at $\lambda \in A^1\Gamma$.

We now compute the last term we need, $\dddot{\omega}_0(\partial_\theta)$. We use \eqref{eq:omegadottt}, \eqref{eq:projections} and \eqref{eq:partialtheta}. We obtain
\begin{align}
\dddot{\omega}_0(\partial_\theta)  & = 2\vec{H}(h_i c_{i j }^0) h_\ell c_{\ell j}^0 + h_i c_{i j }^0 \left( \bcancel{h_\ell  c_{j \ell}^k h_s c_{sk}^0}  + h_\alpha c_{j k}^\alpha h_\ell c_{\ell k}^0 +  \bcancel{h_\ell  c_{\ell k}^j h_s c_{sk}^0} \right) \\
& = \vec{H}\left(h_i c_{ij}^0 h_{\ell} c_{\ell j}^0\right) + \xcancel{h_\alpha \left(h_i c_{i j }^0\right)\left(h_\ell c_{\ell k}^0\right)  c_{j k}^\alpha }  = \vec{H}(2H) = 0.
\end{align}
This concludes the proof of the lemma.
\end{proof}

\begin{lemma}\label{lem:claimdots}
The one-forms $\omega_r$ and two-forms $\omega_r\wedge\dot\omega_r$ on $A^1\Gamma$ (see \cref{def:contactjacobi}) have finite order at all $r$. More precisely, for any $\lambda \in A^1\Gamma$ and $\bar{r}\geq 0$ the following hold:
\begin{enumerate}[label = $(\roman*)$]
\item there exists $m\in \{0,1,2,3\}$, and a non-zero one-form $\beta$ on $A^1\Gamma$ such that 
\begin{equation}\label{eq:claim1-appendix}
\omega_r|_{\lambda} = (r-\bar{r})^m \beta|_{\lambda} + O\left((r-\bar{r})^{m+1}\right),
\end{equation}
\item there exist $n \in \{0,1,2,3,4\}$, and a non-zero two-form $\alpha$ on $A^1\Gamma$ such that
\begin{equation}\label{eq:claim2-appendix}
\omega_r\wedge\dot\omega_r|_{\lambda} = (r-\bar{r})^n \alpha|_{\lambda} + O\left((r-\bar{r})^{n+1}\right),
\end{equation}
\end{enumerate}
where $O(r-\bar{r})$ denote smooth remainders. 
\end{lemma}
\begin{proof}
We first show that $\mathscr{L}^n_{\vec{H}}(\pi^*\omega)$, for $n=0,1,2,3$, the tautological one-form $\tau$, and $dH$, are a basis of one-forms on $T^*M$ (out of the zero section), namely for any $\lambda$ with $H(\lambda)\neq 0$ it holds
\begin{equation}\label{eq:span1}
\spn\{\tau,\, dH,\, \mathscr{L}^{(n)}_{\vec{H}}(\pi^*\omega)\mid n=0,1,2,3 \} = T_{\lambda}^*(T^*M).
\end{equation}
We check \eqref{eq:span1} by inspection of the computations already done in the proof of \cref{lem:computationsdots}. Set $(Jh)_i = h_\ell c_{\ell i}^0$, and note that $Jh$ and $h$ are independent. We rewrite \eqref{eq:omegadot} as
\begin{equation}
\mathscr{L}_{\vec{H}}(\pi^*\omega) = (Jh)_{\ell} \nu_\ell.
\end{equation}
Since $\tau = h_i \pi^*\nu_i + h_0 \pi^*\nu_0$, and $\omega = \nu_0$, the above implies
\begin{equation}
\spn\{\pi^*\omega,\mathscr{L}_{\vec{H}}(\pi^*\omega),\tau\} = \spn\{\pi^*\nu_0,\pi^*\nu_1,\pi^*\nu_2\}.
\end{equation}
Furthermore we rewrite \eqref{eq:omegadott} as
\begin{equation}
\mathscr{L}^2_{\vec{H}}(\pi^*\omega) = (Jh)_j dh_j + \spn\{\pi^*\nu_0,\pi^*\nu_1,\pi^*\nu_2\}.
\end{equation}
Since $dH = h_j dh_j$, the above implies
\begin{equation}
\spn\{\pi^*\omega,\mathscr{L}_{\vec{H}}(\pi^*\omega),\mathscr{L}_{\vec{H}}^2(\pi^*\omega),\tau,dH \} = \spn\{\pi^*\nu_0,\pi^*\nu_1,\pi^*\nu_2,dh_1,dh_2\}.
\end{equation}
Finally, we rewrite \eqref{eq:omegadottt} as
\begin{equation}
\mathscr{L}^3_{\vec{H}}(\pi^*\omega)  = 2H dh_0 + \spn\{\pi^*\nu_0,\pi^*\nu_1,\pi^*\nu_2,dh_1,dh_2\}.
\end{equation}
Since $\{\pi^*\nu_\alpha,dh_\alpha\}_{\alpha=0}^3$ is a basis of one-forms on $T^*M$, \eqref{eq:span1} follows.

Note that the sub-bundle of one-forms generated by $dH,\tau$ is stable under the action of the Hamiltonian flow. More precisely, since $\mathscr{L}_{\vec{H}}(dH)=0$, $\mathscr{L}_{\vec{H}}(\tau)=dH$, it holds
\begin{equation}\label{eq:stabilitydHtau}
e^{r\vec{H}*} dH = dH,\qquad e^{r\vec{H}*} \tau = \tau + r\, dH.
\end{equation}

\textbf{Proof of $(i)$.} Recall that $\omega_r = e^{r\vec{H}*}\circ \pi^* \omega|_{A^1\Gamma}$ (see \cref{def:contactjacobi}). Taking the pull-back with $e^{r\vec{H}*}$ of \eqref{eq:span1}, using the stability of $\spn\{dH,\tau\}$, taking the restriction on $A^1\Gamma$, and noting that $\tau|_{A^1\Gamma}=dH|_{A^1\Gamma}=0$. We conclude that
\begin{equation}\label{eq:listoftods}
\spn\{\omega_r,\dot{\omega}_r,\ddot{\omega}_r,\dddot{\omega}_r\}|_{\lambda} = T_\lambda^*(A^1\Gamma),\qquad \forall  \lambda\in A^1\Gamma,\quad \forall\, r\in \R.
\end{equation}
In particular, at least one (actually, two) of the generators in \eqref{eq:listoftods} is non-zero, and \eqref{eq:claim1-appendix} follows.

\textbf{Proof of $(ii)$.}
Let $B=\{\lambda\in T^*M\mid H(\lambda)\neq 0\}$ and $E$ be the rank $4$ vector bundle over $B$:
\begin{equation}
    E=\ker\{dH,\tau\}|_B.
\end{equation}
It follows from \eqref{eq:stabilitydHtau} that $e^{r\vec{H}*}E=E$. It is convenient to introduce the following one-parameter family of one-forms:
\begin{equation}
\chi_r:=e^{r\vec{H}*}(\pi^*\omega), \qquad r\in \R,
\end{equation}
to distinguish them from the restriction $\omega_r = \chi_r|_{A^1\Gamma}$. We are interested in the evolution of the following one-parameter family of two-forms
\begin{equation}
\alpha_r:= \chi_r\wedge\dot\chi_r, \qquad r\in \R,
\end{equation}
 which (when restricted to $E$) is a global section of $\Lambda^2E^*$.
Let $\sigma$ denote the restriction of the symplectic form of $T^*M$ to $E$. Using \eqref{eq:omegadot} in the proof of \cref{lem:claimdots} one sees that $\alpha_0\wedge\sigma=0$. Consider the following vector bundle over $B$:
\begin{equation}
    L=\{\eta\in\Lambda^2 E^*\mid \eta\wedge\sigma=0\}.
\end{equation}
Since $\sigma$ is non-degenerate on $E$, and since $E$ has rank $4$, the condition $\eta\wedge \sigma =0$ is one non-trivial linear equation. Thus $L$ has rank $5$. It follows from $e^{r\vec{H}*}E=E$ and $e^{r\vec{H}*}\sigma=\sigma$ that $e^{r\vec{H}*}L=L$. Consequently, since $\alpha_0\in L$, then $\alpha_r =e^{r\vec{H}*}\alpha_0 \in L$  for all $r\in\mathbb R$. It follows that
\begin{equation}
\alpha^{(i)}_r:=\frac{d^i}{dr^i}\alpha_r \in \Gamma^\infty(L),\qquad \forall\, r\in \R,\quad i\in \N.
\end{equation}
We claim that $\alpha_r^{(i)}$, for $i=0,\dots,4$, are linearly independent sections of $L$. To see this, by elementary computations using the Leibniz property of the Lie derivative one can show that:
\begin{equation}\label{eq:matriciona}
\begin{pmatrix}
\alpha_r \\
 \alpha_r^{(1)} \\
 \alpha_r^{(2)} \\
 \alpha_r^{(3)} \\
 \alpha_r^{(4)} \\
\end{pmatrix} = \begin{pmatrix}
 1 & 0 & 0 & 0 & 0 & 0 \\
 0 & 1 & 0 & 0 & 0 & 0 \\
 0 & 0 & 1 & 1 & 0 & 0 \\
 * & * & * & 0 & 2 & 0 \\
 * & * & * & * & * & 2
\end{pmatrix}
\begin{pmatrix}
\chi_r\wedge \dot\chi_r \\
\chi_r \wedge \ddot\chi_r \\
\chi_r \wedge \dddot\chi_r \\
\dot \chi_r \wedge \ddot\chi_r \\
\dot \chi_r \wedge \dddot\chi_r \\
\ddot \chi_r \wedge \dddot \chi_r
\end{pmatrix},
\end{equation}
where $*$ denotes a possibly non-zero entry. By \eqref{eq:span1}, for any $r\in \R$, the one-forms $\{\chi_r^{(i)}\}_{i=0}^3$ are independent and, using also \eqref{eq:stabilitydHtau}, when restricted to $E$ yield a trivialization of $E^*$. Therefore, for any $r\in \R$, the two forms $\{\chi_r^{(i)}\wedge\chi_r^{(j)}\}_{0\leq i<j\leq 3}$ are independent sections of $\Lambda^2 E^*$. The matrix in \eqref{eq:matriciona} has rank $5$, and the claim follows.

As a consequence, since $L$ has rank $5$, for any $r\in\mathbb R$, the forms $\{\alpha_r^{(k)}\}_{k=0}^4$ constitute a trivializing basis for $L$. It follows that, for any $r\in\mathbb R$ and $\varepsilon\in \R$
\begin{equation}\label{eq:alpha0inspan}
    \alpha_\varepsilon\in \spn\left\{\alpha_r^{(k)}\,\Big|\, k=0,\dots,4\right\}.
\end{equation}
Note that, since $dH|_{A^1\Gamma}=\tau|_{A^1\Gamma}=0$, $T(A^1\Gamma)\subset E$. Therefore one can take the restriction of \eqref{eq:alpha0inspan} to $A^1\Gamma$. Recalling that, by definition, $\alpha_\varepsilon|_{A^1\Gamma} = \omega_\varepsilon\wedge\dot\omega_\varepsilon$, we obtain:
\begin{equation}
    \omega_\varepsilon\wedge\dot\omega_\varepsilon\in \spn\left\{\frac{d^k}{dr^k}\omega_r\wedge\dot\omega_r \,\Big|\, k=0,\dots,4\right\}, \qquad \forall\,r\in\R.
\end{equation}
We recall that by the contact condition \eqref{eq:Estarcontact}, the left hand side of the above equation is non-zero for sufficiently small $\varepsilon>0$, then at least one of the generators of the right hand side is non-zero. This concludes the proof.
\end{proof}

\section{A formula for the Schwarzian derivative of contact Jacobi curves}\label{app:formula_for_S}

Let $(M,\omega,g)$ be a contact sub-Riemannian manifold and let $\Gamma \subset M$ be an embedded piece of Reeb orbit. We assume here that $\rho=r_{\inj}(\Gamma)>0$, which is always the case if $\Gamma$ has compact closure by \cref{thm:exp_map}, and $\delta$ (the distance function from $\Gamma$) is smooth on $E(A^{<\rho}\Gamma)\setminus \Gamma$. There, its horizontal gradient is well-defined, as the unique horizontal vector field $\nabla \delta$ satisfying
	\begin{equation}
		d\delta(Y)=g(\nabla\delta, Y),\qquad \forall \, Y\in \xi,
	\end{equation}
	which is also called the \emph{horizontal normal} to $\Gamma$, and hence denoted by $N = \nabla \delta$.
	
\begin{defi}\label{def:adaptedframe}
Let $N=\nabla \delta$ be the horizontal normal, and let $J: \xi \to \xi$ be the almost-complex structure \eqref{eq:complex_strctr}. The ordered triple $\{f_0,N,JN\}$, which is an oriented orthonormal frame defined on $E(A^\rho\Gamma)\setminus \Gamma$, is called \emph{adapted frame} associated to $\Gamma$.
\end{defi}	

The next proposition shows that the Schwarzian derivative of contact Jacobi curves can be expressed in terms of the adapted frame and the Riemannian extension.
\begin{prop}\label{prop:schw_expr}
 Given $\lambda\in A^1\Gamma$ and $r\in (0,r_{\inj}(\Gamma))$ the Schwarzian derivative of the contact Jacobi curve \eqref{eq:contact_Jacobi} can be expressed in terms of the adapted frame \cref{def:adaptedframe} as 
\begin{equation}\label{eq:schw_expr}
\frac{1}{2}\mathcal S(\Omega_{\lambda})(r)=g([N,f_0],JN)-\frac{1}{2}Ng([N,JN],JN)-\frac{1}{4}g([N,JN],JN)^2,
\end{equation}
where the right hand side is evaluated at $E(r\lambda)$.
\end{prop}

To prove Proposition \ref{prop:schw_expr} we write the adapted frame in cylindrical coordinates.
\begin{prop}\label{prop:adaptedframe}
Let $(M,\omega, g)$ be a complete three-dimensional contact sub-Riemannian structure. Let $\Gamma \subset M$ be an embedded piece of Reeb orbit, with $\rho=r_{\inj}(\Gamma)>0$. Let $E: A^{<\rho}\Gamma \to M$ be the tubular neighbourhood map. Let $\omega_r$ be the following smooth $1$-parameter family of one-forms on $A^1\Gamma$:
  \begin{equation}\label{eq:moving_w}
      \omega_r:=(\pi\circ e^{r\vec{H}})^*\omega|_{A^1\Gamma},\qquad r\in \R.
  \end{equation}
The adapted frame $\{f_0,N,JN\}$ is oriented and orthonormal, and in cylindrical coordinates at $\lambda=(z,r\cos\theta,r\sin\theta)$ it holds
\begin{align}
E^{-1}_* N 		& = \partial_r, \label{normal_frame_coords1}\\
E^{-1}_* JN 	& = \frac{1}{a}\bigg(\omega_r(\partial_z)\partial_\theta-\omega_r(\partial_\theta)\partial_z\bigg), \label{normal_frame_coords2}\\
E^{-1}_*f_0		& = \frac{1}{a}\bigg(\dot\omega_r(\partial_\theta)\partial_z-\dot\omega_r(\partial_z)\partial_\theta+d\omega_r(\partial_\theta,\partial_z)\partial_r\bigg), \label{normal_frame_coords3}
\end{align}
where $a:=\omega_r\wedge \dot\omega_r(\partial_z,\partial_\theta)$, the dot denoting the derivative w.r.t.\ $r$.
	\end{prop}
	\begin{remark}
	Since $\omega_r$ is a one-parameter family of one-forms on $A^1\Gamma$ one must be careful in evaluation of \cref{normal_frame_coords1,normal_frame_coords2,normal_frame_coords3}. For clarity, e.g. \eqref{normal_frame_coords2} evaluated at a point $r\lambda \in A\Gamma$ for some $\lambda \in A^1\Gamma$, it corresponds to
	\begin{equation}
	E^{-1}_* JN|_{r\lambda} = \frac{1}{a}\bigg(\omega_r(\partial_z|_{\lambda})\partial_\theta|_{r\lambda}-\omega_r(\partial_\theta|_{\lambda})\partial_z|_{r\lambda}\bigg), \qquad a =\omega_r\wedge \dot\omega_r(\partial_z|_{\lambda},\partial_\theta|_{\lambda}).
	\end{equation}
	\end{remark}
	\begin{proof}
The fact that the frame $\{f_0,N,JN\}$ is smooth, orthonormal and oriented on $E(A^\rho\Gamma)\setminus \Gamma$ follows immediately from the definitions. Thus we only have to compute it in a set of cylindrical coordinates.
	It follows from the homogeneity property \eqref{eq:rescaling} that  in cylindrical coordinates
  \begin{equation}
      E(r\cos\theta,r\sin\theta,z)=\pi\circ e^{r\vec{H}}(\cos\theta\nu_1+\sin\theta\nu_2)|_{\gamma(z)}.
  \end{equation}
  Let $r\lambda \in A\Gamma$ for some $\lambda = (z,\cos\theta,\sin\theta) \in A^1\Gamma$. Let $\lambda_r = e^{r\vec{H}}(\lambda)$. As explained in \cite[Lemma 2.3]{subOnR3}, it follows from the Pontryagin maximum principle that
		\begin{equation}
			\langle\lambda_r,E_*\partial_\theta\rangle=\langle\lambda_r,E_*\partial_z\rangle= 0,\qquad \langle \lambda_r, E_* \partial_r \rangle = 1,
		\end{equation}
		and thus $E^*\lambda_r = dr$.

The vector field $E_*\partial_r =\pi_*\vec{H}$ is horizontal and has norm one, hence $E_*\partial_r, J E_* \partial_r$ is an horizontal  oriented orthonormal frame. Therefore using this frame in \eqref{eq:ham_vect} we obtain
\begin{equation}
E_*\partial_r =\pi_*\vec{H} = \langle \lambda_r, E_* \partial_r\rangle E_* \partial_r +  \langle \lambda_r, J E_* \partial_r\rangle J E_* \partial_r.
\end{equation}
It follows that $\langle\lambda_r, J E_* \partial_r\rangle=0$.
		
We compute the pull-back of the contact form. Using the definition \eqref{eq:moving_w} of $\omega_r$ we obtain
\begin{align}
(E^*\omega)|_{r\lambda} & = \omega(E_*\partial_r|_{r\lambda})dr+\omega(E_*\partial_\theta|_{r\lambda})d\theta+\omega(E_*\partial_z|_{r\lambda})dz\\
& = \omega\left(\pi_*\vec{H}|_{e^{\vec{H}}(r\lambda)}\right)dr+\omega\left((\pi \circ e^{r\vec{H}})_*\partial_\theta|_{\lambda}\right)d\theta+\omega\left((\pi \circ e^{r\vec{H}})_*\partial_z|_{\lambda}\right)dz\\
& = \omega_r(\partial_\theta|_{\lambda})d\theta+\omega_r(\partial_z|_{\lambda})dz.\label{eq:omega_coords}
    \end{align}
		Thus, the equalities $E^*\lambda_r=dr$, $\langle \lambda_r, JE_* \partial_r\rangle=0$ and $\omega(JE_*\partial_r)=0$ together imply
		\begin{equation}
			\tilde{J}\partial_r:=E_*^{-1}JE_*\partial_r \propto \omega_r(\partial_z)\partial_\theta-\omega_r(\partial_\theta)\partial_z.
		\end{equation}
		The factor of proportionality is found imposing $E^*d\omega( \partial_r, \tilde{J}\partial_r)=1$, and since
		\begin{equation}\label{eq:domega_coords}
			E^* d\omega|_{r\lambda} = dr\wedge\Big(\dot\omega_r(\partial_\theta|_{\lambda})d\theta+\dot\omega_r(\partial_z|_{\lambda})dz\Big)+d\omega_r(\partial_\theta|_{\lambda},\partial_z|_{\lambda})d\theta \wedge dz,
		\end{equation}
		we find $\iota_{\partial_r}E^* d\omega=\dot\omega_r(\partial_\theta)d\theta+\dot\omega_r(\partial_z)dz$ (we omit evaluations for brevity) and consequently
		\begin{equation}
			\tilde{J}\partial_r=\frac{1}{a}\Big(\omega_r(\partial_z)\partial_\theta-\omega_r(\partial_\theta)\partial_z\Big).
		\end{equation}
		To prove that $N=\nabla\delta$ satisfies $N=E_{*}\partial_r$, observe that according to \eqref{eq:gamma_dist_smooth} we have 
		\begin{equation}
			\delta(E(r\cos\theta,r\sin\theta,z))=r.
		\end{equation}
		Thus for any horizontal vector field $Y$, since $\langle dr,\tilde{J}\partial_r\rangle=0$, we have 
		\begin{equation}
			Y(\delta)=\langle dr,\left(g(Y,E_*\partial_r)\partial_r+g(Y,JE_*\partial_r)\tilde{J}\partial_r\right)\rangle=g(Y,E_*\partial_r).
		\end{equation}
		This yields \eqref{normal_frame_coords1} and \eqref{normal_frame_coords2}. Finally from the coordinate presentations \eqref{eq:omega_coords} for $E^{*}\omega$ and \eqref{eq:domega_coords} for $E^{*}d\omega$, one can check that the expression of $f_0$ appearing in equation \eqref{normal_frame_coords3} satisfies the characterizing equations $\omega(f_0)=1,\,d\omega(f_0,\cdot)=0$.
	\end{proof}
\begin{proof}[Proof of Proposition \cref{prop:schw_expr}]
To lighten notation, let us set 
\begin{equation}
	A=g([N,JN],JN),\quad \dot A=N(A),\quad B=g([N,f_0],JN). 
\end{equation}
An elementary computation exploiting the expressions \eqref{normal_frame_coords1}-\eqref{normal_frame_coords2}-\eqref{normal_frame_coords3} shows
\begin{equation}
A=\frac{1}{a}{\omega}_r\wedge\ddot{{\omega}}_r(\partial_\theta,\partial_z),\quad
B=-\frac{1}{a}\dot{{\omega}}_r\wedge\ddot{{\omega}}_r(\partial_\theta,\partial_z),
\end{equation}
where $a=\dot\omega_r\wedge\omega_r(\partial_\theta,\partial_z)$. We want to prove formula \eqref{eq:schw_expr}, which is
\begin{equation}
\mathcal S(\Omega_\lambda)=S:=B-\frac{\dot A}{2}-\frac{A^2}{4}.
\end{equation}
Let $f:(0,\rho)\to (0,+\infty)$ be a positive function, define $\bar\omega_r=f\omega_r$ as well as the corresponding barred quantities
\begin{equation}
\bar A=\frac{1}{\bar a}{\bar\omega}_r\wedge\ddot{{\bar\omega}}_r(\partial_\theta,\partial_z),\quad
\bar B=-\frac{1}{\bar a}\dot{{\bar \omega}}_r\wedge\ddot{{\bar \omega}}_r(\partial_\theta,\partial_z),
\end{equation}
where $\bar a=\dot{{\bar \omega}}_r\wedge\bar\omega_r(\partial_\theta,\partial_z)$. We claim that $\bar S=S$. Indeed, computing shows the relations
\begin{equation}
\bar A=A-2\frac{\dot f}{f},\quad \bar B=B-\frac{\dot f}{f}\bar A-\frac{\ddot f}{f}.
\end{equation}
Therefore we can compute
\begin{align}
\bar S= \bar B-\frac{\dot{\bar A}}{2}-\frac{{\bar A}^2}{4}&=
\left(B-\frac{\dot f}{f}\bar A-\frac{\ddot f}{f}\right)-\frac{1}{2}\frac{d}{dr}\left(A-2\frac{\dot f}{f}\right)-\frac{1}{4}\left(A-2\frac{\dot f}{f}\right)^2\\
&=S-\left(\frac{\dot f}{f}\bar A+\frac{\ddot f}{f}\right)+\frac{d}{dr}\frac{\dot f}{f}-\left(\left(\frac{\dot f}{f}\right)^2-A\frac{\dot f}{f}\right)\\
&=S-\frac{\dot f}{f}\bar A-2\left(\frac{\dot f}{f}\right)^2+A\frac{\dot f}{f}=S.
\end{align}
If we choose $f=1/\omega_r(\partial_\theta)$, then, denoting $v=\omega_r(\partial_z)/\omega_r(\partial_\theta)$, we have 
\begin{equation}
	\bar\omega_r=vdz+d\theta,\quad \dot{{\bar \omega}}_r=\dot vdz,\quad \ddot{{\bar \omega}}_r=\ddot v dz,
\end{equation}
consequently 
\begin{equation}
\dot{{\bar \omega}}_r\wedge\bar\omega_r=-\dot vd\theta\wedge dz,\quad  {\bar \omega}_r\wedge\ddot{{\bar \omega}}_r=\ddot vd\theta\wedge dz\quad \dot{{\bar \omega}}_r\wedge \ddot{{\bar \omega}}_r=0.
\end{equation}
We deduce that $\bar A=-\ddot v/\dot v$ and $\bar B=0$, therefore 
\begin{equation}
\begin{aligned}
S&=\bar S=-\frac{\dot{\bar A}}{2}-\frac{{\bar A}^2}{4}=\frac{1}{2}\left(\frac{\dddot v}{\dot v}-\left(\frac{\ddot v}{\dot v}\right)^2\right)-\frac{1}{4}\left(\frac{\ddot v}{\dot v}\right)^2\\
&=\frac{1}{2}\left(\frac{\dddot v}{\dot v}-\frac{3}{2}\left(\frac{\ddot v}{\dot v}\right)^2\right)=\frac{1}{2}\mathcal S(\Omega_\lambda).
\end{aligned}
\end{equation}
Formula \eqref{eq:schw_expr} is thus proved.
\end{proof}

\addtocontents{toc}{\SkipTocEntry}
\section*{Declarations}

\addtocontents{toc}{\SkipTocEntry}
\subsection*{Conflict of interest} On behalf of all authors, the corresponding author states that there is no conflict of interest.

\addtocontents{toc}{\SkipTocEntry}
\subsection*{Data availability statement} Data sharing is not applicable to this article as no new data were created or analyzed in this study.

\bibliographystyle{alphaabbr}
	\bibliography{bibliography-tightness}

\end{document}